\documentclass[twoside,11pt]{article}

\makeatletter
\let\Ginclude@graphics\@org@Ginclude@graphics
\makeatother

\usepackage[preprint]{jmlr2e}

\usepackage{xr}
\usepackage{hyperref}
\usepackage{algorithm,algorithmic}
\usepackage{makecell}
\usepackage{amssymb, multirow, paralist, color,amsmath}
\usepackage{enumitem}
\usepackage{textcase}
\usepackage[T1]{fontenc}
\usepackage{mathtools}
\usepackage{cleveref}
  
\usepackage{empheq}
\usepackage{color}
\definecolor{darkgreen}{rgb}{0.00,0.5,0.00}
\usepackage{fancybox}
\usepackage{tablefootnote}
\usepackage{natbib}
\newtheorem{asm}{Assumption}

\usepackage{subfigure}

\hypersetup{
	colorlinks=true,        
	linkcolor=blue,         
	citecolor=darkgreen,         
	urlcolor=darkgreen           
}

\definecolor{myteal}{RGB}{27,158,119}
\definecolor{myorange}{RGB}{217,95,2}
\definecolor{myred}{RGB}{231,41,138}
\definecolor{mypurple}{RGB}{152,78,163}
\definecolor{myblue}{rgb}{.9, .9, 1}
\definecolor{mygreen}{RGB}{0,100,0}
\definecolor{mycyan}{rgb}{0.88,1,1}
\definecolor{mydarkred}{RGB}{192,47,25}

\def\inner#1#2{\langle #1, #2 \rangle}

\def \bD {\mathbf{D}}
\def \S {\mathbf{S}}

\def \X {\mathcal{X}}

\def \R {\mathbb{R}}

\def \w {\w}
\def \v {\mathbf{v}}

\def \x {\mathbf{x}}

\def \E {\mathbb{E}}

\def \x {\mathbf{x}}

\def \e {\mathbf{e}}

\def \1 {\mathbf{1}}
\def \z {\mathbf{z}}

\def \y {\mathbf{y}}
\def \u {\mathbf{u}}

\def \L {\mathcal{L}}

\newcommand{\norm}[1]{\|#1\|}

\newcommand{\Norm}[1]{\left\|#1\right\|}

\usepackage[skins]{tcolorbox}

\def \F {\mathcal{F}}
\def \I {\mathbf{I}}

\def \Q {\mathcal{Q}}

\usepackage{bbding}

\usepackage{cancel}
\usepackage{colortbl}

\def \B {\mathcalB}
\def \C {\mathbf C}
\def \O {\mathcal O}

\usepackage{bbm}
\usepackage{booktabs}
\usepackage{tablefootnote}

\newtheorem{cor}{Corollary}
\newcommand*\tcircle[1]{%
  \raisebox{-0.5pt}{%
    \textcircled{\fontsize{7pt}{0}\fontfamily{phv}\selectfont #1}%
  }%
}

\newlength\mytemplen
\newsavebox\mytempbox

\makeatletter
\newcommand\mybluebox{%
	\@ifnextchar[
	{\@mybluebox}%
	{\@mybluebox[0pt]}}

\def\@mybluebox[#1]{%
	\@ifnextchar[
	{\@@mybluebox[#1]}%
	{\@@mybluebox[#1][0pt]}}

\def\@@mybluebox[#1][#2]#3{
	\sbox\mytempbox{#3}%
	\mytemplen\ht\mytempbox
	\advance\mytemplen #1\relax
	\ht\mytempbox\mytemplen
	\mytemplen\dp\mytempbox
	\advance\mytemplen #2\relax
	\dp\mytempbox\mytemplen
	\colorbox{myblue}{\hspace{1em}\usebox{\mytempbox}\hspace{1em}}}

\def \y {\mathbf{y}}
\def \E {\mathbb{E}}

\def \x {\mathbf{x}}

\def \D {\mathbf{D}}
\def \z {\mathbf{z}}
\def \u {\mathbf{u}}

\def \Z {\mathcal{Z}}

\def \w {\mathbf{w}}

\def \R {\mathbb{R}}
\def \S {\mathcal{S}}

\def \Q {\mathcal{Q}}

\def \q {\mathbf{q}}
\def \v {\mathbf{v}}

\def \q {\mathbf{q}}

\def \bg {\mathbf{g}}

\usepackage{amssymb}
\usepackage{pifont}

\usepackage{bm}

\def \B {\mathcal{B}}

\def \C {\mathcal{C}}

\def \bff {\mathbf{f}}

\def \X {\mathcal{X}}

\def \F {\mathcal{F}}

\usepackage{fontawesome}
\usepackage{threeparttable,adjustbox}

\usepackage[colorinlistoftodos,bordercolor=green,backgroundcolor=white,linecolor=green,textsize=scriptsize]{todonotes}

\newcommand{\CC}[1]{\cellcolor{gray!30}}

 \usepackage{tikz}
\usetikzlibrary{automata, positioning, arrows}
\usepackage{xcolor}

\firstpageno{1}

\begin{document}
	
\title{Finite-Sum Coupled Compositional Stochastic Optimization: Theory and Applications\thanks{
This is an extended version of our earlier proceeding's version in ICML'22. In this version, we make two updates. First,  we add an improved convergence result of the objective gap in Theorem 6 regarding the strongly convex objective. Second, we correct a statement of Theorem 3 in previous version (Theorem~\ref{thm:cvx_sox-simp} in this version). In present proof of Theorem~\ref{thm:cvx_sox-simp}, which is the same as previous version, we can only plugin any fixed primal dual variables that are independent of the randomness of the algorithm. Hence, Theorem 5 is only for the convergence of the weak primal-dual gap. This is similar to prior work~\cite{DBLP:conf/icml/SongWD21} on algorithms with stochastic dual coordinate updates. 
}}
	
\author{\name Bokun Wang \email bokun-wang@tamu.edu\\
\addr Department of Computer Science and Engineering\\
Texas A\&M University~\thanks{Most work was done when the authors were at The University of Iowa. }\\
College Station, TX~77843, USA
\AND 
\name Tianbao Yang \email tianbao-yang@tamu.edu\\
\addr Department of Computer Science and Engineering\\
Texas A\&M University~$^\dagger$\\
College Station, TX 77843, USA}
	
\editor{}
	
\maketitle

\begin{abstract}
This paper studies stochastic optimization for a sum of compositional functions, where the inner-level function of each summand is coupled with the corresponding summation index. We refer to this family of problems as finite-sum coupled compositional optimization (FCCO). It has broad applications in machine learning for optimizing non-convex or convex compositional measures/objectives such as average precision (AP), $p$-norm push,  listwise ranking losses, neighborhood component analysis (NCA),  deep survival analysis,  deep latent variable models, etc., which deserves finer analysis. Yet, existing algorithms and analyses are restricted in one or other aspects. The contribution of this paper is to provide a comprehensive convergence analysis of a simple stochastic algorithm for both non-convex and convex objectives. Our key result is the {\bf improved oracle complexity with the parallel speed-up} by using the moving-average based estimator with mini-batching. Our theoretical analysis also exhibits new insights for improving the practical implementation by sampling the batches of equal size for the outer and inner levels. Numerical experiments on AP maximization, NCA and $p$-norm push corroborate some aspects of the theory.
\end{abstract}
 \section{Introduction}
 A fundamental problem in machine learning (ML) that has been studied extensively is the empirical risk minimization (ERM), whose objective is a sum of individual losses on training examples, i.e., 
 \begin{align*}
 	\min_{\w\in\Omega}F(\w),\quad F(\w):=\frac{1}{n}\sum_{\z_i\in\D}\ell(\w; \z_i),
 \end{align*} 
where $\w$ and $\Omega$ denotes the model parameter and its domain ($\Omega\subseteq \R^d$), $\D$ denotes the training set of $n$ examples,  and $\z_i$ denotes an individual data. However, ERM may hide the complexity of individual loss and gradient computation in many interesting measures/objectives
. Instead, in this paper we study a new family of problems that aims to optimize the following compositional objective:
 \begin{align}\label{eqn:FCC}
 	\min_{\w\in\Omega}F(\w),\quad F(\w) := \frac{1}{n}\sum_{\z_i\in\D}f_i(g(\w; \z_i, \S_i)),
 \end{align}
 where $g:\Omega\mapsto\mathcal{R}$ \footnote{$\mathcal R$ refers to the $d'$-dimensional codomain of $g$ ($d'\geq 1$).}, $f_i:\mathcal{R} \mapsto \R$, and $\S_i$ denotes another (finite or infinite \footnote{If $\S_i$ is a finite set, we define $g(\w;\z_i,\S_i) = \frac{1}{|\S_i|}\sum_{\xi_{ij}\in\S_i} g(\w;\z_i,\xi_{ij})$; If $\S_i$ is an infinite set, we define $g(\w;\z_i,\S_i) = \E\left[g(\w;\z_i,\xi_i)\mid \z_i\right]$, where $\xi_i\in\S_i$.}) set of examples that could be either dependent or independent of $\z_i$. We give an example for each case: 1) In the bipartite ranking problem, $\D$ represents the positive data while $\S = \S_i$ represents the negative data; 2) In the robust learning problem (e.g. the invariant logistic regression in \citealt{hu2020biased}), $\D$ represents the training data set while $\S_i$ denotes the set of perturbed observations for data $\z_i\in\D$, where $\S_i$ depends on $\z_i$. We are particularly interested in the case that set $\S_i$ is infinite or contains a large number of items, and assume that an unbiased stochastic estimators of $g$ and $\nabla g$ can be computed via sampling from $\S_i$. We refer to~(\ref{eqn:FCC}) as finite-sum coupled compositional optimization  (FCCO) and its objective as finite-sum coupled compositional risk (FCCR), where for each data $\z_i$ the risk $f_i(g(\w; \z_i, \S_i))$ is of a compositional form such that $g$ couples each $\z_i$ and items in $\S_i$. It is notable that $f_i$ could be stochastic or has a finite-sum structure that depends on a large set of items. For simplicity of presentation and discussion, we focus on the case that $f_i$ is a simple deterministic function whose value and gradient can be easily computed, which covers many interesting objectives  of interest. The algorithms and analysis can be extended to the case that  $\{f_i, \nabla f_i\}$ are estimated by their unbiased stochastic versions using random samples (discussed in Appendix~\ref{sec:ext}). 
 
 {\bf Applications of FCCO.} The Average Precision (AP) maximization problem studied in \citet{qi2021stochastic} is an example of FCCO. Nevertheless, we notice that the application of FCCO is much broader beyond the AP maximization, including but are not limited to $p$-norm push optimization~\citep{rudin2009p}, listwise ranking objectives~\citep{ListNet,ListMLE} (e.g. ListNet, ListMLE, NDCG), neighborhood component analysis (NCA)~\citep{goldberger2004neighbourhood}, deep survival analysis~\citep{deepsurv}, deep latent variable models~\citep{DBLP:conf/icml/GuuLTPC20}, etc. We postpone the details of some of these problems to Section~\ref{sec:exp} and~\ref{sec:app}. We would like to emphasize that efficient stochastic algorithms for these problems are lacking or under-developed when the involved set $\S_i$ is big and/or the predictive model is nonconvex. 
 
 In this paper, we propose the  \underline{\bf S}tochastic \underline{\bf O}ptimization of the \underline{\bf X} objectives listed in Sections~\ref{sec:exp} and~\ref{sec:app} (SOX) and establish its convergence guarantees for several classes of functions. 
 
 \section{Related Work}
 
 In this section, we connect the FCCO problem to Conditional Stochastic Optimization (CSO) and Stochastic Compositional Optimization (SCO) in the literature and discuss the limitations of existing algorithms for FCCO. Then, we position SOX in previous studies and list our contributions.
 
 \subsection{Conditional Stochastic Optimization (CSO)}
 
 The most straightforward approach to solve the FCCO problem in \eqref{eqn:FCC} is to compute the gradient $\nabla F(\w) = \frac{1}{n}\sum_{\z_i\in\D} \nabla g(\w;\z_i,\S_i) \nabla f_i(g(\w;\z_i,\S_i))$ and then use the gradient descent method. However, it can be seen that computing the gradient $\nabla F(\w)$ is very expensive, if not infeasible, when $|\D|$ or $|\S_i|$ is large. Thus, a natural idea is to sample mini-batches $\B_1\subset \D$ and $\B_{i,2}\subset \S_i$ and compute a stochastic gradient of the form $\v = \frac{1}{|\B_1|}\sum_{\z_i\in\B_1} \nabla g(\w;\z_i,\B_{i,2}) \nabla f_i(g(\w;\z_i,\B_{i,2}))$ and update the parameter as $\w\leftarrow \w - \eta \v$, where $g(\w; \z_i, \B_{i,2})\coloneqq \frac{1}{|\B_{i,2}|}\sum_{\xi_{ij}\in\B_{i,2}}g(\w;\z_i,\xi_{ij})$. The resulting algorithm is named as biased stochastic gradient descent (BSGD) in \citet{hu2020biased} and the convergence guarantees of BSGD under different assumptions are established. 
 Actually, \citet{hu2020biased} study the more general problem $\E_{\xi}f_{\xi}(\E_{\zeta|\xi}[g_\zeta(\w; \xi)])$, which is referred to as conditional stochastic optimization (CSO) and show that BSGD has the optimal oracle complexity for the general CSO problem. The FCCO problem can be mapped to the CSO problem by $\z_i=\xi$, $\zeta\in\S_i$ and the only difference between these two is the \emph{finite-sum structure} of the outer-level function in FCCO. Unfortunately, BSGD requires unrealistically large batch sizes to ensure the convergence from the theoretical perspective (Please refer to columns 5 and 6 of Table~\ref{tab:comp}). As a comparison, our approach explicitly exploits the finite support of the outer level and leads to improved oracle complexity with mini-batch sizes $|\B_1| = O(1)$, $|\B_{i,2}| = O(1)$.
 
 \begin{table*}[t]
 	\caption{Summary of {\bf iteration complexity}  of different methods for different problems. ``NC" means non-convexity of $F$, ``C" means convexity of $F$,``SC" means strongly convex and ``PL" means Polyak-Lojasiewicz condition. Iteration complexity is the number of iterations to reduce the quantities (NC: $\E\|\nabla F(\w)\|$; C: $\E [F(\w) - F(\w^*)]$; SC (PL): $\E \|\w - \w^*\|_2^2$) below $\epsilon$.  ``N/A" means not applicable or not available. For the complexity of  our method, we omit constants that are independent of $n$ or the batch sizes. $^*$ denotes that additional assumptions (bounded $\Omega$, convex and monotone $f$) are needed and the the rate is for a weaker optimality measure. $\dagger$ assumes a stronger SC condition, which is imposed on a stochastic function instead of the original objective. We suppose $B_1 = |\B_1^t|$ and $B_2 = |\B_{i,2}^t|$ for simplicity. $^\diamond$ denotes that the bound is only for strongly convex FCCO problems (see Section~\ref{sec:scvx} in Appendix).}
 	\label{tab:comp}
 	\vskip 0.15in
 	\begin{center}
 		\scalebox{0.8}{
 			\begin{tabular}{ccccccc}
 				\toprule[0.02in]
 				Method & NC & C &  SC (PL) & \makecell{Outer Batch \\Size $|\B_1|$} & \makecell{Inner Batch \\Size $|\B_{i,2}|$}  & \makecell{Parallel \\Speed-up}  \\
 				\midrule[0.01in]
 				BSGD~\citep{hu2020biased}&$O(\epsilon^{-4})$ & $O\left(\epsilon^{-2}\right)$ &  $O\left(\mu^{-2}\epsilon^{-1} \right)^\dagger$ & 1 &  \makecell{$O(\epsilon^{-2})$ (NC) \\ $O(\epsilon^{-1})$ (C/SC)}  & N/A \\
 				SOAP~\citep{qi2021stochastic} &$O(n\epsilon^{-5})$&   -& - & 1 &1 & N/A\\
 				MOAP~\citep{wang2021momentum} & $O\left(\frac{n\epsilon^{-4}}{B_1}\right)$&   -& - & $B_1$ & 1 & Partial \\
 				\midrule[0.005in] \rowcolor{gray!40}
 				SOX/SOX-boost (this work)  & $O\left(\frac{n\epsilon^{-4}}{B_1B_2}\right)$ &   $O\left(\frac{n\epsilon^{-3}}{B_1B_2}\right)$  &  $O\left(\frac{n\mu^{-3}\epsilon^{-1}}{B_1B_2}\right)$& $B_1$ &$B_2$ & Yes\\\rowcolor{gray!40}
 				SOX ($\beta=1$) (this work) & - &  $O\left(\frac{n\epsilon^{-2}}{B_1}\right)^*$ & \textcolor{blue}{$\tilde{O}\left(\frac{n \mu^{-2} \epsilon^{-1}}{B_1}\right)$ $^\diamond$} & $B_1$ & $B_2$ & Partial\\
 				\bottomrule[0.02in]
 		\end{tabular}}
 	\end{center}
 \end{table*}

\subsection{Stochastic Compositional Optimization (SCO)}\label{sec:sco}
 
 A closely related class of problems: stochastic compositional optimization (SCO) has been extensively studied in the literature. In particular, the SCO problem with the finite support in the outer level is in the form of $F(\w) = \frac{1}{n}\sum_{\z_i\in\D}f_i(g(\w;\S))$, where $\S$ might be finite or not. The difference between FCCO and SCO is that the inner function $g(\w;\S)$ in SCO does not depend on $\z_i$ of the outer summation. The SCGD algorithm \citep{wang2017stochastic} is a seminal work in this field, which tracks the unknown $g(\w_t;\S)$ with an auxiliary variable $u_t$ that is updated by the exponential moving average $u\leftarrow (1-\gamma)u + \gamma g(\w;\B_2)$, $\gamma\in(0,1)$ based on an unbiased stochastic estimator $g(\w;\B_2)$ of $g(\w;\S)$, which circumvents the unrealistically large batch size required by the sample average approximation approach. For example, we can sample a mini-batch $\B_2 \subset \S$ and compute $g(\w;\B_2)\coloneqq \frac{1}{|\B_2|}\sum_{\xi_j\in\B_2} g(\w;\xi_j)$. Then, the stochastic estimator of $\nabla F(\w)$ can be computed as $\v_t = \frac{1}{|\B_2|}\sum_{\z_i\in\B_1}\nabla g(\w;\B_2)\nabla f_i(u)$. More recently, the NASA algorithm~\citep{Ghadimi2020AST} modifies the SCGD algorithm by adding the exponential moving average (i.e., the momentum) to the gradient estimator, i.e., $\v\leftarrow (1-\beta)\v + \beta \frac{1}{|\B_1|}\sum_{\z_i\in\B_1}\nabla g(\w;\B_2)\nabla f_i(u)$, $\beta\in(0,1)$, which improves upon the convergence rates of SCGD. When $f_i$ is convex and monotone ($d'=1$) and $g$ is convex, \citet{zhang2020optimal} provide a more involved analysis for the two-batch SCGD\footnote{In the original SCGD algorithm~\citep{wang2017stochastic}, they use the same batch $\B_2$ to update $u$ by $g(\w;\B_2)$ and to compute the gradient estimator by $\nabla g(\w;\B_2)$. In the work of \citet{zhang2020optimal}, they analyze the two-batch version SCGD which uses independent batches $\B_2$ and $\B_2'$ for $g(\w;\B_2)$ and $\nabla g(\w;\B_2')$. The two-batch version with independent $\B_2$, $\B_2'$ is definitely less efficient, but it considerably simplies the analysis.} in its primal-dual equivalent form and derive the optimal rate for a special class of problems that inner function is convex while the outer function is convex and monotone.
 
 {\bf SCO reformulation of FCCO.} Given the union data set $\S = \S_1\cup\cdots\S_i\cdots\cup\S_n$, we can define $\bg(\w;\S)=[g(\w;\z_1,\S_1)^\top, \ldots, g(\w;\z_n,\S_n)^\top]^\top\in\R^{nd'}$ and $\hat{f}_i(\cdot) \coloneqq f_i(\I_i \cdot)$, $\I_i \coloneqq \left[0_{d\times d},\dotsc, I_{d\times d},\dotsc, 0_{d\times d}\right]\in\R^{d'\times nd'}$ (the $i$-th block in $\I_i$ is the identity matrix while the others are zeros), the FCCO problem $F(\w) = \frac{1}{n}\sum_{i=1}^n f_i(g(\w;\z_i,\S_i))$ can be reformulated as an SCO problem $F(\w) = \frac{1}{n}\sum_{i=1}^n \hat{f}_i(\bg(\w;\S))$ such that the existing algorithms for the SCO problem can be directly applied to our FCCO problem. Unfortunately, applying SCGD and NASA on the FCCO problem via the SCO reformulation need $n$ oracles for the inner function $\bg(\w;\S)$ (one oracle for each $g(\w;\z_i,\S_i)$) and update all $n$ components of $\u = [u_1,\dotsc,u_n]^\top$ at any time step $t$ even if we only sample one data point $\z_i\in \D$ in the outer level, which could be expensive or even infeasible.  
 
 Apart from adopting the rather na\"ive SCO reformulation, we can also alter the algorithm according to the special structure of FCCO. SCGD and NASA algorithms could be better tailored for the FCCO problem if it selectively samples $\B_{i,2}$ and selectively updates those coordinates $u_i$ for those sampled $\z_i\in\B_1$ at each time step, instead of sampling $\B_{i,2}$ for all $\z_i\in\D$ and update all $n$ coordinates of $\u= [u_1^\top,\dotsc, u_n^\top]^\top$. Formally, the update rule of $\u = [u_1^\top ,\dotsc, u_n^\top]^\top$ can be expressed as 
 \begin{align}\label{eq:update_v1}
 	u_i \leftarrow \begin{cases}
 		(1-\gamma) u_i + \gamma g(\w;\z_i,\B_{i,2}), & i\in\B_1\\
 		u_i, & i\notin \B_1
 	\end{cases}.
 \end{align}
 The update rule above has been exploited by some recent works (e.g.~SOAP in ~\citealp{qi2021stochastic}) to solve the average precision (AP) maximization problem, which is a special case of FCCO. However, the convergence guarantees of these algorithms are only established for smooth nonconvex problem and do not enjoy the parallel speed-up by mini-batching. In this work, we build convergence theory for a broader spectrum of problems and show the parallel speed-up effect. Moreover, we ressolve several issues of existing approaches from the algorithmic and theoretical perspectives (See Table~\ref{tab:comp} and Section~\ref{sec:sox_comparison} for details).
 
 \subsection{Our Contributions}
 
 Our contributions can be summarized as follows.
 
 $\bullet$ On the convex and nonconvex problems, our SOX algorithm can guarantee the convergence but does not suffer from some limitations in previous methods such as the unrealistically large batch of BSGD~\citep{hu2020biased}, the two independent batches for oracles of the inner level in SCGD~\citep{zhang2020optimal}, and the possibly inefficient/unstable update rule in MOAP~\citep{wang2021momentum}.
 
 $\bullet$ On the smooth nonconvex problem, SOX has an improved rate compared to SOAP and enjoys a better dependence on $|\B_{i,2}|$ compared to MOAP.  
 
 $\bullet$ Beyond the smooth nonconvex problem, we also establish the convergence  guarantees of SOX for problems that $F$ is convex/strongly convex/PL, which are better than BSGD in terms of oracle complexity. 
 
 $\bullet$ Moreover, we carefully analyze how mini-batching in the inner and outer levels improve the worst-case convergence guarantees of SOX in terms of iteration complexity, i.e., the parallel speed-up effect. The theoretical insights are numerically verified in our experiments.
 
 \begin{algorithm}[t]
 	\caption{SOX($\w^0$,$\u^0$,$\v^0$,$\eta$, $\beta$, $\gamma$, $T$)}
 	\begin{algorithmic}[1]
 		\FOR{$t=1,\ldots,T$}
 		\STATE Draw a batch of samples $\B_1^t\subset \D$
 		\IF{$\z_i \in \B_1^t$} 
 		\STATE Update the estimator of function value $g_i(\w^{t})$ $$u^{t}_i =  (1-\gamma)u^{t-1}_i  + \gamma g(\w^{t};\z_i, \B_{i,2}^t)$$
 		\ENDIF 
 		\STATE Update the estimator of gradient $\nabla F(\w^{t})$ by
 		\begin{align*}
 			\v^{t} & =(1-\beta)\v^{t-1}  + \beta\frac{1}{B_1}\sum_{\z_i\in\B_1^t}\nabla g(\w^{t};\z_i,\B_{i,2}^t) \nabla f_i(\colorbox{gray!40}{$ u^{t-1}_i$}) 
 		\end{align*}
 		\STATE Update the model parameter $\w^{t+1} = \w^t - \eta_t \v^t$
 		\ENDFOR
 	\end{algorithmic}
 	\label{alg:sox}
 \end{algorithm}
 
 \section{Algorithm and Convergence Analysis}\label{sec:main}

{\bf Notations.} For machine learning applications, we let $\D=\{\z_1, \ldots, \z_n\}$ denote a set of training examples for general purpose, let $\w\in\Omega$ denote the model parameter (e.g., the weights of a deep neural network). Denote by $h_\w(\z)$ a prediction score of the model on the data $\z$. A function $f$ is Lipchitz continuous on the domain $\Omega$ if there exists $C>0$ such that $\norm{f(\w) - f(\w')}\leq C\norm{\w-\w'}$ for any $\w,\w'\in\Omega$, and is smooth if its gradient is Lipchitz continuous. A function $F$ is convex if it satisfies $F(\w)\geq F(\w')+\nabla F(\w')^{\top}(\w - \w')$ for all $\w, \w'\in\Omega$, is $\mu$-strongly convex if  there exists $\mu>0$ such that  $F(\w)\geq F(\w')+\nabla F(\w')^{\top}(\w - \w') + \frac{\mu}{2}\|\w-\w'\|^2$ for all $\w, \w'\in\Omega$. A smooth function $F$ is said to satisfy $\mu$-PL condition if there exists $\mu>0$ such that $\|\nabla F(\w)\|^2\geq \mu (F(\w) - \min_{\w} F(\w))$, $\w\in\Omega$. 
 
We make the following assumptions throughout the paper\footnote{The result in Theorem~\ref{thm:cvx_sox-simp} does not need $g_i$ to be smooth.}.
 
 \begin{asm}\label{asm:lip}
 	We assume that (i) $f_i(\cdot)$ is differentiable, $L_f$-smooth and $C_f$-Lipchitz continuous; (ii) $g(\cdot;\z_i,\S_i)$ is differentiable, $L_g$-smooth and $C_g$-Lipchitz continuous for any $\z_i\in\D$; (iii) $F$ is lower bouned by $F^*$.
 \end{asm}
 {\bf Remark:} If the assumption above is satisfied, it is easy to verify that $F(\w)$ is $L_F$-smooth, where $L_F\coloneqq C_fL_g + C_g^2 L_f$ (see Lemma 4.2 in \citealt{zhang2021multilevel}). The assumption that $f_i$ is smooth and Lipchitz continuous seems to be strong. However, the image of $g_i$ is bounded on domain $\Omega$ in many applications (otherwise there might be a numerical issue), hence $f_i$ is smooth and Lipchitz continuous in a bounded domain is enough for our results. 
 
 \subsection{A Better Stochastic Algorithm for FCCO}\label{sec:sox_comparison}
 
 We follow the idea of tailoring SCGD/NASA to solve the FCCO problem by \emph{selective sampling} and \emph{selective update} as described in Section~\ref{sec:sco}. Next, we thoroughly discuss the relation of our SOX algorithm to the existing algorithms SOAP and MOAP for the FCCO problem.  
 
 \begin{figure*}[h]
 	\centering
 	\begin{tikzpicture}[->, >=stealth', node distance=2.5cm, every state/.style={thick, fill=gray!10}, initial text=$ $]
 		\node[rectangle] at (-2, 0) (1) {\makecell{SCGD\\\citep{wang2017stochastic}}};
 		\node[rectangle] at (3, 0) (2) {\makecell{NASA\\\citep{Ghadimi2020AST}}};
 		\node[rectangle] at (8, -1.5) (3) {\bf SOX (this work)};
 		\node[rectangle] at (8, 0) (4) {\makecell{MOAP\\\citep{wang2021momentum}}};
 		\node[rectangle] at (3, -1.5) (5) {\makecell{SOAP\\\citep{qi2021stochastic}}};
 		\node[rectangle] at (-2, -1.5) (6) {\makecell{BSGD\\\citep{hu2020biased}}};
 		\draw (1) edge[above] node{\small +GM} (2)
 		(2) edge[right] node{\small +SS, SU} (3)
 		(2) edge[above] node{\small +SS} (4)
 		(5) edge[above] node{\small $\gamma=1$} (6)
 		(5) edge[above] node{\small +GM} (3)
 		(1) edge[right] node{\small +SS, SU} (5);
 	\end{tikzpicture}
 	\caption{The algorithmic relationship among our SOX and the existing algorithms. ``GM'' refers to the \underline{g}radient \underline{m}omentum; ``SS'' refers to the property that the algorithm only needs to \underline{s}electively \underline{s}ample $\B_{i,2}$ for those $\z_i\in\B_1$ at each time step; ``SU'' refers to the property that the algorithm only needs to \underline{s}electively \underline{u}pdate those coordinates $u_i$ for $\z_i\in\B_1$. }
 	\label{fig:state_diagram}
 \end{figure*}
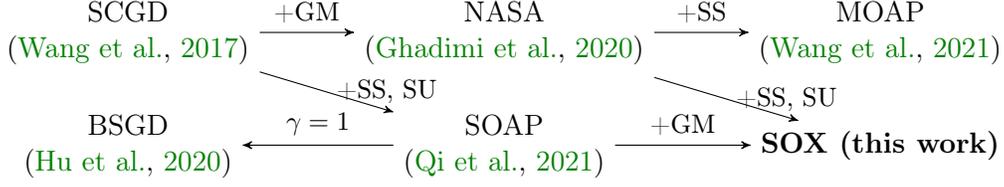
 
 SOAP algorithm~\citep{qi2021stochastic} combines~\eqref{eq:update_v1} with the gradient step $\w\leftarrow \w - \eta \v$, $\v \leftarrow \frac{1}{|\B_1|}\sum_{\z_i\in\B_1}\nabla g(\w;\z_i,\B_{i,2}) \nabla f_i(u_i)$. \citet{wang2021momentum} attempted to do the same adaptation for NASA~\citep{Ghadimi2020AST} by an algorithm called MOAP, which applies the uniform random sparsification~\citep{wangni2018gradient} or the uniform randomized block coordinate sampling~\citep{nesterov2012efficiency} to the whole $\bg(\w)$ and derives the improved rate compared to SOAP. To be specific, the update rule of $\u = [u_1,\dotsc, u_n]^\top$ in MOAP is
 \begin{align}\label{eq:update_v2}
 	u_i \leftarrow \begin{cases}
 		(1-\gamma) u_i + \gamma \frac{n}{|\B_1|}g(\w;\z_i,\B_{i,2}), & i\in\B_1\\
 		(1-\gamma)u_i, & i\notin \B_1.
 	\end{cases}
 \end{align}
 It is worth mentioning that the convergence guarantees for SOAP and MOAP are only established for the smooth nonconvex problems. Besides, SOAP and MOAP are only analyzed when $|\B_1^t| = 1$.  The update rule \eqref{eq:update_v2} of MOAP also has several extra drawbacks: a) It requires extra costs to update all $u_i$ at each iteration, while \eqref{eq:update_v1} only needs to update $u_i$ for the sampled $\z_i\in\B_1$; b) For the large-scale problems (i.e., $n$ is large), multiplying $g(\w; \z_i,\B_{i,2})$ by $\frac{n}{|\B_1|}$ might lead to numerical issue; c) Due to the property of random sparisifcation/block coordinate sampling (see 
 Proposition 3.5 in \citealt{khirirat2018distributed}), it does not enjoy any benefit of mini-batch $\B_{i,2}$ in terms of iteration complexity. 
 
 {\bf Main idea of SOX:} We make subtle modifications on SOAP --- 1) directly adding the gradient momentum; 2)  using $u_i^{t-1}$ instead of $u_i^t$ in step 6 of Algorithm~\ref{alg:sox}, which are crucial for us to improve the convergence rate. In particular, taking expectation of the estimation error $\|u_i^{t-1} - g(\mathbf w^t;\z_i,\S_i)\|^2$ over the randomness in $\mathbf z_i\in\mathcal B_1^t$ (due to independence between $\mathbf u^{t-1}$ and the randomness in $\mathbf z_i\in\mathcal B_1^t$) leads to bounding the average error $\frac{1}{n}\sum_{\z_i\in\D} \norm{u_i^{t-1} - g(\mathbf w^t;\z_i,\S_i)}^2$ over all $\mathbf z_i\in\mathcal D$, which can be decomposed into $\frac{1}{n}\sum_{\z_i\in\D}\norm{u_i^t - u_i^{t-1}}^2$ and $\frac{1}{n}\sum_{\z_i\in\D}\norm{u_i^t - g(\mathbf w^t;\z_i,\S_i)}^2$, where the latter term is bounded as in Lemma~\ref{lem:fval_recursion} and the first term is cancelled with the highlighted negative term in Lemma~\ref{lem:fval_recursion}.

 \subsection{Improved Rate for the Nonconvex Problems}
 
 In this subsection, we present the convergence analysis for the smooth nonconvex problems. We will highlight the key differences from the previous analysis. We use the following assumption, which is also used in previous works~\citep{qi2021stochastic,wang2021momentum}.
 \begin{asm}\label{asm:var}
 	We assume that $\E[\|g(\w; \z_i,\xi_i) - g(\w;\z_i,\S_i)\|^2\mid \z_i]\leq \sigma^2$ and $\E[\|\nabla g(\w; \z_i,\xi_i) - \nabla g(\w;\z_i,\S_i)\|^2\mid \z_i]\leq \zeta^2$ for any $\w$, $\z_i\in\D$, and $\xi_i\in\S_i$.
 \end{asm}
We aim to find the approximate stationary points.
 \begin{definition}
 	$\w$ is an $\epsilon$-stationary point if $\Norm{\nabla F(\w)}\leq \epsilon$. 
 \end{definition}
The recursion for the variance of inner function value estimation $\Xi_t\coloneqq \frac{1}{n}\norm{\u^t - \bg(\w^t;\S)}^2$ is crucial for our analysis. 
 \begin{lemma}\label{lem:fval_recursion}
 	If $\gamma\leq 1/5$, function value variance $\Xi_t\coloneqq \frac{1}{n}\norm{\u^t - \bg(\w^t;\S)}^2$ can be bounded as
 	\begin{align*}
 		\E\left[\Xi_{t+1}\right] &\leq \left(1-\frac{\gamma B_1}{4n}\right)\E\left[\Xi_t\right] + \frac{5n\eta^2C_g^2\E\left[\norm{\v^t}^2\right]}{\gamma B_1} + \frac{2\gamma^2\sigma^2B_1}{nB_2}- \colorbox{gray!40}{$\frac{1}{4n}\E\left[\sum_{\z_i\in\B_1^t}\norm{u_i^{t+1} - u_i^t}^2\right]$}.
 	\end{align*}	
 \end{lemma}	
Instead of sampling a singleton $\B_1^t=\{i_t\}$ at each iteration and bounding $\sum_t[\|u_{i_t}^t - g(\w^t;\z_{i_t},\S_{i_t})\|^2]$ for $\z_{i_t}$ in \citet{qi2021stochastic} \footnote{Please refer to the comments above (27) in \citet{wang2021momentum} for the issue of bounding $\sum_t[\|u_{i_t}^t - g(\w^t;\z_{i_t},\S_{i_t})\|^2]$.}, Lemma~\ref{lem:fval_recursion} bounds $\sum_t\|\u^t - \bg(\w^t;\S)\|^2$ that includes all coordinates of $\u^t$ at each iteration. To build the recursion, we consider a strongly convex minimization problem $	\min_{\u}\frac{1}{2}\|\u - \bg(\w^t;\S)\|^2$ that is equivalent to 
 \begin{align}\label{eqn:u}
 	\min_{\u=[u_1,\dotsc, u_n]^\top}\frac{1}{2}\sum_{\z_i\in\D}\|u_i - g(\w^t;\z_i,\S_i)\|^2.
 \end{align}
 Then, the step 4 in SOX can be viewed as stochastic block coordinate descent algorithm applied to \eqref{eqn:u}, i.e., 
 \begin{align}\label{eq:bcd}
 	u^t_i = \begin{cases}  u^{t-1}_i - \gamma\left(u^{t-1}_i - g(\w^t;\z_i,\B_{i,2}^t)\right),& \z_i\in\B_1^t\\ u^t_i, & \z_i\notin \B_1^t,\end{cases}
 \end{align}
 where $u^{t-1}_i - g(\w^t;\z_i,\B_{i,2}^t)$ is the stochastic gradient of the $i$-th coordinate in the objective~(\ref{eqn:u}). This enables to us to use the proof technique of stochastic block coordinate descent methods to build the recursion of $\Xi_t$ and derive the improved rate compared to previous works listed in Table~\ref{tab:comp}.  
 
 By combining the lemma above with Lemma~\ref{lem:nonconvex_starter} and Lemma~\ref{lem:grad_recursion} in the supplement, we prove the convergence to find an $\epsilon$-stationary point, as stated in the following theorem. 
 \begin{theorem}\label{thm:sox}
 	Under Assumption~\ref{asm:lip} and \ref{asm:var}, SOX~(Algorithm~\ref{alg:sox}) with $\beta=O(\min\{B_1,B_2\}\epsilon^2)$, $\gamma = O(B_2\epsilon^2)$, $\eta = \min\left\{\frac{\beta}{4 L_F}, \frac{\gamma B_1}{30L_f n C_1 C_g}\right\}$ can find an $\epsilon$-stationary point in $$T=O\left(\max\left\{\frac{n}{B_1B_2\epsilon^4}, \frac{1}{\min\{B_1, B_{2}\}\epsilon^4}\right\}\right)$$ iterations. 
 \end{theorem}
 {\bf Remark:} The above theory suggests that given a budget on the mini-batch size $B_1+B_2=B$, the best value of $B_1$ is $B_1=B/2$. We will verify this result in experiments.  
 
 \subsection{Improved Rate for (Strongly) Convex Problems}
 
 In this subsection, we prove improved rates for SOX for convex and strongly convex objectives compared to previous work BSGD~\citep{hu2020biased}. One might directly analyze SOX with different decreasing step sizes for convex and strongly convex objectives separately as in~\citet{hu2020biased}. However, to our knowledge, this strategy does not yield an optimal rate for strongly convex functions. To address this challenge, we provide a unified algorithmic framework for both convex and strongly convex functions and derive the improved rates. The idea is to use the stagewise framework given in Algorithm~\ref{alg:sox-boost} to boost the convergence. Our strategy is to prove an improved rate for an objective that satisfies a $\mu$-PL condition $\|\nabla F(\w)\|^2\geq\mu(F(\w) - F(\w^*))$, 
 where $\w^*$ is a global minimum. Then, we use this result to derive the improved rates for (strongly) convex objectives. 
 
 \begin{theorem}\label{thm:sox-boost-simp}
 	Assume $F$ satisfying the PL condition, by setting $\epsilon_k=O(1/2^{k-1})$, $\beta_k=O(\eta_k), \gamma_k=O(\frac{n\eta_k}{B_1}), T_k=O(\frac{1}{\mu\eta_k}), \eta_k=O(\min(\mu\min(B_1, B_2)\epsilon_k, \frac{\mu B_1B_2\epsilon_k}{n}))$, and  $K=\log(1/\epsilon)$,  SOX-boost ensures that $\E[F(\w^K) - F(\w^*)]\leq \epsilon$, which implies a total iteration complexity of $T=O(\max(\frac{n}{\mu^2 B_1B_2\epsilon}, \frac{1}{\mu^2\min(B_1, B_2)\epsilon})).$
 \end{theorem}

Specific values of the parameters in Theorem~\ref{thm:sox-boost-simp} can be found in Theorem~\ref{thm:sox-boost} in the appendix. The result above directly implies the improved complexity for $\mu$-strongly convex function, as it automatically satisfies the PL condition. For a convex function, we use a common trick to make it strongly convex by contructing $\hat F(\w) = F(\w) + \frac{\mu}{2}\|\w\|^2$, then we use SOX-boost to optimize $\hat F(\w)$ with a small $\mu$. Its convergence is summarized by the following corollary. 
\begin{cor}\label{cor:cvx}
	Assume $F$ is convex, by setting $\mu=O(\epsilon)$, $\eta_k, \gamma_k, \beta_k, T_k$ according to Theorem~\ref{thm:sox-boost-simp}, then after $K=\log(1/\epsilon)$-stages SOX-boost for optimizing $\hat F$ ensures that $\E[\hat F(\w^K) - \min_{\w}\hat F(\w)]\leq \epsilon$, which implies a total iteration complexity of $T=O(\max(\frac{n}{\mu^2 B_1B_2\epsilon}, \frac{1}{\mu^2\min(B_1, B_2)\epsilon}))$ for ensuring $\E[F(\w^K) - F(\w^*)]\leq \epsilon$. 
\end{cor}
{\bf Remark:} The above result implies an complexity of $T=O(\max(\frac{n}{B_1B_2\epsilon^3}, \frac{1}{\min(B_1, B_{2})\epsilon^3}))$ for a convex function. 

\begin{algorithm}[t]
	\caption{SOX-boost($\w_1$, $\u_1$, $\v_1$, $K$)}
	\begin{algorithmic}[1]
		\FOR{epochs $k=1,\ldots,K$}
		\STATE Update $\w$, $\u$, $\v$ by SOX($\w^k$, $\u^k$, $\v^k$, $\eta_k$, $\beta_k$, $\gamma_k$, $T_k$) 
		\STATE Update $\eta_k, \beta_k, \gamma_k, T_k$ according to Theorem~\ref{thm:sox-boost-simp}
		\ENDFOR
	\end{algorithmic}
	\label{alg:sox-boost}
\end{algorithm}

\subsection{Optimal Rate for A Class of Convex Problems}
In this section, we consider a class of FCCO problems on a closed, non-empty, and convex domain $\Omega$ and $d' = 1$. 
For simplicity, we denote $g(\w; \z_i, \S_i)$ by $g_i(\w)$ and its stochastic estimator $g(\w; \z_i,\xi_i)$ by $g_i(\w;\xi_i)$ in this section. We additionally make the assumption below.
\begin{asm}\label{asm:bounded_dom}
	Assume that $d'=1$, $f_i$ is monotonically increasing and convex, while $g_i$ is convex. The domain $\Omega$ is bounded such that  $\max_{\w\in\Omega}\|\w - \w_*\|\leq C_{\Omega}$ and $\max_{\w\in\Omega}\|g(\w; \z_i,\xi_i)\|\leq D_g$ for any $i$ and r.v. $\xi_i$.
\end{asm}
This FCCO can be reformulated as a saddle point problem.
\begin{align*}
	\min_{\w\in\Omega}\max_{\pi_1\in\Pi_1}\max_{\pi_2\in\Pi_2} \L(\w,\pi_1,\pi_2),
\end{align*}
where $\L(\w,\pi_1,\pi_2) = \frac{1}{n}\sum_{\z_i\in\D} \L_i(\w,\pi_{i,1},\pi_{i,2})$, $\pi_1 = [\pi_{1,1},\dotsc, \pi_{n,1}]^\top$, $\pi_2 = [\pi_{1,2},\dotsc, \pi_{n,2}]^\top$, $\L_{i,1}(\w,\pi_{i,1},\pi_{i,2}) = \pi_{i,1} \L_{i,2}(\w,\pi_{i,2}) - f_i^*(\pi_{i,1})$ and $\L_{i,2}(\w,\pi_{i,2}) = \inner{\pi_{i,2}}{\w} - g_i^*(\pi_{i,2})$. Here $f_i^*(\cdot)$ and $g_i^*(\cdot)$ are the convex conjugates of $f_i$ and $g_i$, respectively.  We analyze the SOX algorithm with $\beta=1$, $\gamma = \frac{1}{1+\tau}$ ($\tau>0$) and the projection onto $\Omega$, which is equivalent to the following primal-dual update formula:
\begin{align}\label{eq:pd1}
	& \pi_{i,2}^{t+1} = \arg\max_{\pi_{i,2}} \inner{\pi_{i,2}}{\w^t} - g_i^*(\pi_{i,2}),\quad  \z_i\in\B_1^t,\\\nonumber
	& \pi_{i,1}^{t+1} = \begin{cases}\makecell{\arg\max_{\pi_{i,1}} \pi_{i,1}\L_{i,2}(\w^t,\pi_{i,2}^t(\B_{i,2}^t))  - f_i^*(\pi_{i,1}) - \tau D_{f_i^*}(\pi_{i,1}^t,\pi_{i,1}),}& \z_i\in\B_1^t\\ \pi_{i,1}^t, &\z_i\not\in\B_1^t,\end{cases}
\\\nonumber
	& \w^{t+1} = \arg\min_{\w\in\Omega} \frac{1}{B_1}\sum_{\z_i\in\B_1^t} \pi_{i,1}^t \pi_{i,2}^{t+1}(\B_{i,2}^t)\w + \frac{\eta}{2}\Norm{\w-\w^t}^2,
\end{align}
where $\pi_{i,2}(\B_{i,2})$ is a stochastic estimation of $\pi_{i,2}$ based on the mini-batch $\B_{i,2}$ and $D_f(x,y)\coloneqq f(y) - f(x) - \inner{\nabla f(x)}{y-x}$ is the Bregman divergence. We define that $\bD_{f^*}(\pi_1,\pi_1') \coloneqq \sum_{\z_i\in\D} D_{f_i^*}(\pi_{i,1},\pi_{i,1}')$ for any $\pi_1,\pi_1'\in\Pi_1$. Note that \eqref{eq:pd1} is equivalent to $\pi_{i,2}^{t+1} = \nabla g_i(\w^t)$. Besides, for $\z_i\in\B_1^t$ and $\pi_{i,1}^t=\nabla f_i(u_i^t)$ we have
\begin{align*}
	\pi_{i,1}^{t+1}  = \arg\min_{\pi_{i,1}} &-\pi_{i,1}\L_{i,2}(\w^t,\pi_{i,2}^t(\B_{i,2}^t)) + f_i^*(\pi_{i,1})  + \tau D_{f_i^*}(\pi_{i,1}^t,\pi_{i,1})\\
 =\arg\min_{\pi_{i,1}} &-\pi_{i,1}\frac{g_i(\w^t;\B_{i,2}^t)+\tau f_i(u_i^t)}{1+\tau} + f_i^*(\pi_{i,1})  .
\end{align*}
The last equation above is due to $g_i(\w^t;\B_{i,2}^t) = \L_{i,2}(\w^t,\pi_{i,2}^t(\B_{i,2}^t))$ and $f_i(u_i^t) = \pi_{i,1}^t u_i^t - f_i^*(\pi_{i,1}^t)$. Then, we can conclude that $\pi_{i,1}^{t+1} = \nabla f_i(u_i^{t+1})$ if we define $u_i^{t+1} = (1-\gamma) f_i(u_i^t) + \gamma g_i(\w^t;\B_{i,2}^t)$ and $\gamma = \frac{1}{1+\tau}$.

On this class of convex problems, we can establish an improved rate for SOX in the order of $O(1/\epsilon^2)$, which is optimal in terms of $\epsilon$ (but might not be optimal in terms of $n$). The analysis is inspired by~\citet{zhang2020optimal}, which provide the optimal complexity for the traditional SCO problems. We extend their analysis to handle the selective sampling/update to accommodate the FCCO problem. Moreover, our analysis also gets rid of one drawback of \citet{zhang2020optimal} that needs two independent batches to estimate $g_i(\w)$ and $\nabla g_i(\w)$, which is achieved by cancelling the highlighted terms in Lemma~\ref{lem:bound_Q1} and Lemma~\ref{lem:bound_Q0}. 

\begin{theorem}\label{thm:cvx_sox-simp}
	Assume $f_i$ is monotone, convex, smooth and Lipschitz-continuous while $g_i$ is convex and Lipschitz-continuous. SOX with $ \eta=O(\min(\min\{B_1,B_2\}\epsilon, \frac{B_1\epsilon}{n})), \gamma=O(B_2\epsilon), \beta=1$, $\bar\w^T=\sum_t\w^t/T$, $\bar \pi_{1}^T = \sum_t\pi_{1}^t/T$, $\bar \pi_{2}^T = \sum_t\pi_{2}^t/T$ ensures that $\max_{\w,\pi_1,\pi_2}\E[\L(\bar\w^T,\pi_1,\pi_2) - \L(\w,\bar \pi_1^T, \bar \pi_2^T)]\leq \epsilon$ after $O(\max(\frac{n}{B_1\epsilon^2}, \frac{n}{B_1B_2\epsilon^2}, \frac{1}{\min(B_1,B_2)\epsilon^2}))$ iterations. 
\end{theorem}
\begin{remark}
The complexity of $O(1/\epsilon^2)$ matches the best for convex problems without additional assumptions~\citet{agarwal2009information}. However,  the convergence bound on the primal-dual gap $\max_{\w,\pi_1,\pi_2}\E[\L(\bar\w^T,\pi_1,\pi_2) - \L(\w,\bar \pi_1^T, \bar \pi_2^T)]$ is weaker than that of the duality gap $\E[\max_{\w,\pi_1,\pi_2} \L(\bar\w^T,\pi_1,\pi_2) - \L(\w,\bar \pi_1^T, \bar \pi_2^T)]$. Similar results were proved in~\cite{DBLP:conf/icml/SongWD21} for stochastic algorithms with stochastic dual coordinate updates.  Please refer to Appendix~\ref{sec:mistake} for more discussions. 
\end{remark}
Next, we show an improved rate for variance convergence when the objective function is strongly convex.

\section{Improved Rate of Convergence for Strongly Convex Problems}\label{sec:scvx}

In this section, we consider a class of regularized FCCO
\begin{align*}
\min_{\w\in\Omega} F(\w),\quad F(\w) \coloneqq F(\w) + R(\w),
\end{align*}
where $F(\w)\coloneqq \frac{1}{n}\sum_{\z_i\in\D} f_i(g(\w;\z_i,\S_i))$, $f_i:\R\rightarrow \R$ and the domain $\Omega$ is closed, non-empty, and convex. For simplicity, we denote $g(\w; \z_i, \S_i)$ by $g_i(\w)$ and its stochastic estimator $g(\w; \z_i,\xi_i)$ by $g_i(\w;\xi_i)$ in this section. We additionally make the assumption below.

\begin{asm}\label{asm:R}
$R$ is $\mu$ strongly-convex, $L_R$-smooth and $C_R$-Lipschitz continuous on $\Omega$. 
\end{asm}

By the Fenchel conjugation, this problem can be reformulated as a saddle point problem.
\begin{align*}
	\min_{\w\in\Omega}\max_{\pi\in\Pi} \L(\w,\pi),\quad \L(\w,\pi) = \Phi(\w,\pi) - \mathbf{f}^*(\pi) + R(\w),
\end{align*}
where $f_i^*$ is the convex conjugate of $f_i$, $\mathbf{f}^*(\pi) = \frac{1}{n}\sum_{\z_i\in\D} f_i^*(\pi_i)$, $\Phi(\w,\pi) =  \frac{1}{n}\sum_{\z_i\in \D} \pi_i g_i(\w)$.

\begin{theorem}\label{thm:scvx_sox}
Under Assumptions~\ref{asm:var}, \ref{asm:R}, SOX with $\eta_t=O(\mu(t+1)), \gamma_t =\frac{1}{1+\tau_t}, \tau_t=O(\frac{B_1}{n}(t+1)), \beta_t =\beta =1$ ensures that $\max\{\E[F(\bar{\w}^T) - F(\w^*)], \E[\|\w^T - \w^*\|_2^2]\}\leq \epsilon$ after $\tilde{O}(\frac{n}{B_1\mu^2 \epsilon})$ iterations. Moreover, $\max_{\w,\pi} \E[\L(\bar{\w}^T,\pi) - \L(\w,\bar{\pi}^T)] \leq \epsilon$ after $\tilde{O}\left(\frac{n}{B_1 \mu \epsilon}\right)$ iterations, where $\bar\w^T=\sum_{t=0}^{T-1}\w^t/T$, $\bar \pi^T = \sum_{t=0}^{T-1}\pi^t/T$. 
\end{theorem}
\section{Experiments}\label{sec:exp}
In this section, we provide some experimental results to verify some aspects of our theory and compare SOX  with other baselines for three applications: deep average precision (AP) maximization, $p$-norm push optimization with concentration at the top, and neighborhood component analysis (NCA).

\subsection{Deep AP Maximization}\label{sec:ap} 
AP maximization in the form of FCCO has been considered in~\citet{qi2021stochastic,wang2021momentum}. For a binary classification problem, let $\S_+$, $\S_-$ denote the set of positive and negative examples, respectively, $\S=\S_+\cup\S_-$ denote the set of all examples.  A smooth surrogate objective for maximizing AP can be formulated as:
\begin{align}\label{eq:ap}
	F(\w) = - \frac{1}{|\S_+|}\sum_{\x_i\in\S_+}\frac{\sum_{\x\in\S_+}\ell(h_\w(\x) - h_\w(\x_i))}{\sum_{\x\in\S}\ell(h_\w(\x) - h_\w(\x_i))},
\end{align}
where  $\ell(\cdot)$ is a surrogate function that penalizes large input. It is a special case of FCCR by defining $g_i(\w)=[\sum_{\x\in\S_+}\ell(h_\w(\x) - h_\w(\x_i)), \sum_{\x\in\S}\ell(h_\w(\x) - h_\w(\x_i))]$ and $f(g_i(\w)) = - \frac{[g_i(\w)]_1}{[g_i(\w)]_2}$.
\begin{figure*}
	\subfigure[Varying $B_1$]{
		\centering
		\includegraphics[width=0.3\linewidth]{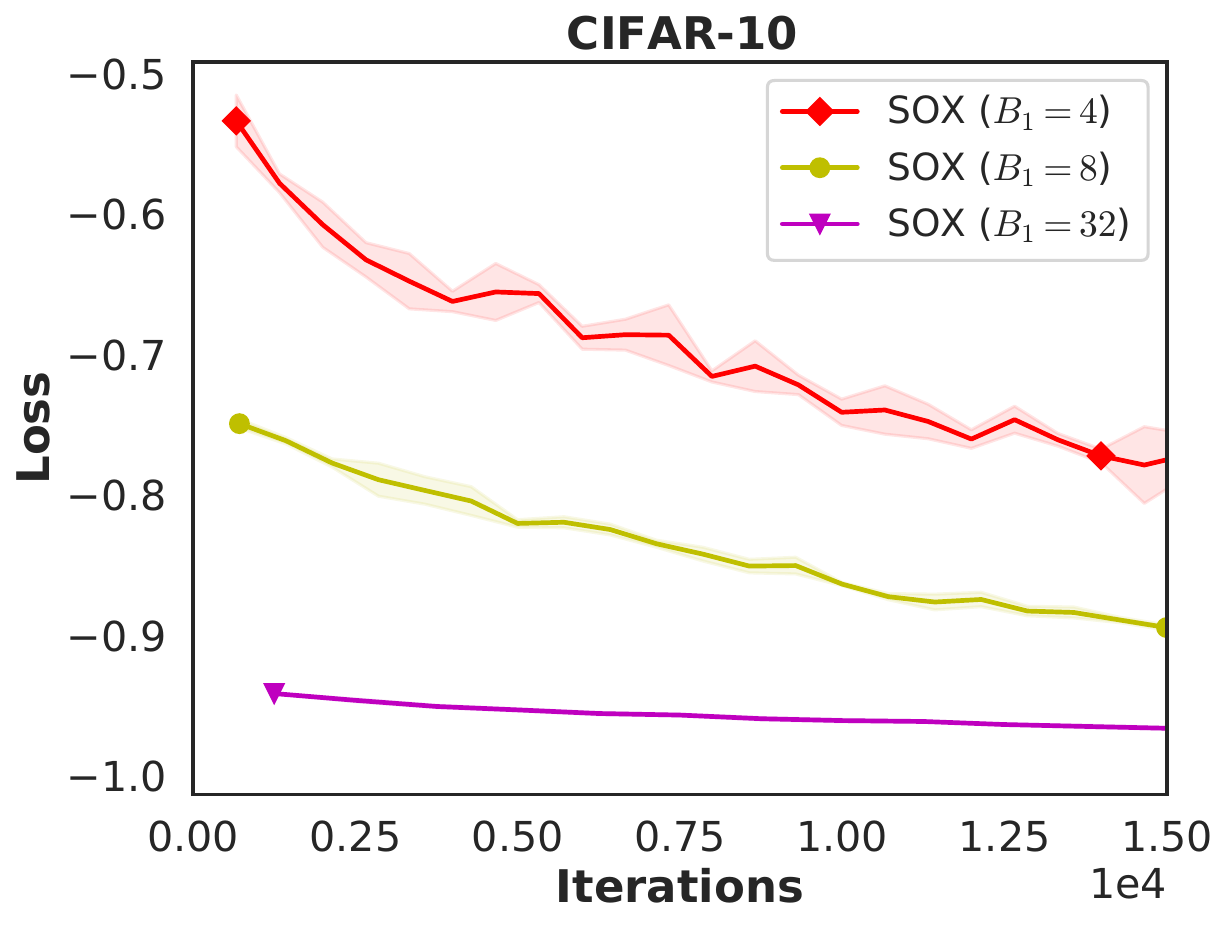}
	}
	\hfill
	\subfigure[Varying $B$]{	\centering
		\includegraphics[width=0.3\linewidth]{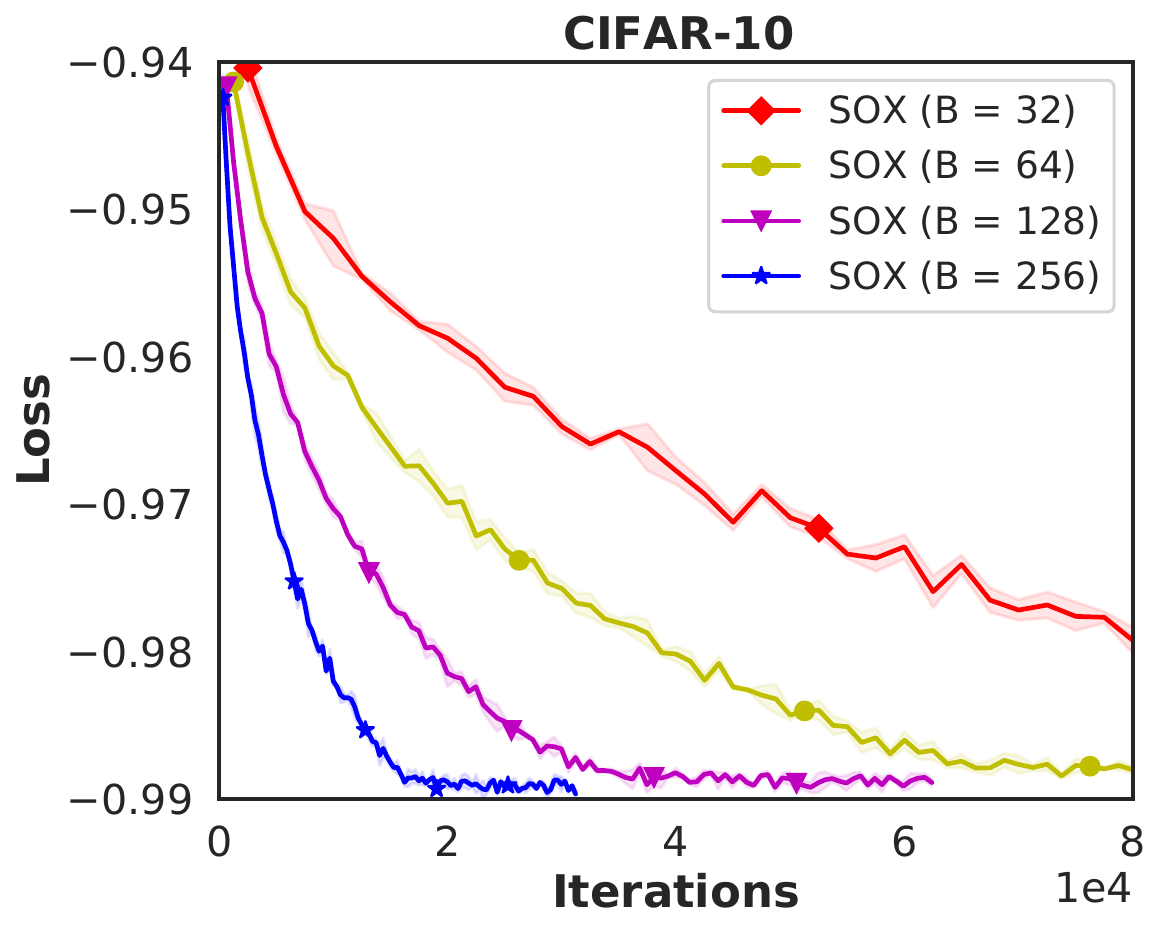}
	}
	\hfill
	\subfigure[Varying $\gamma$]{	\centering
		\includegraphics[width=0.3\linewidth]{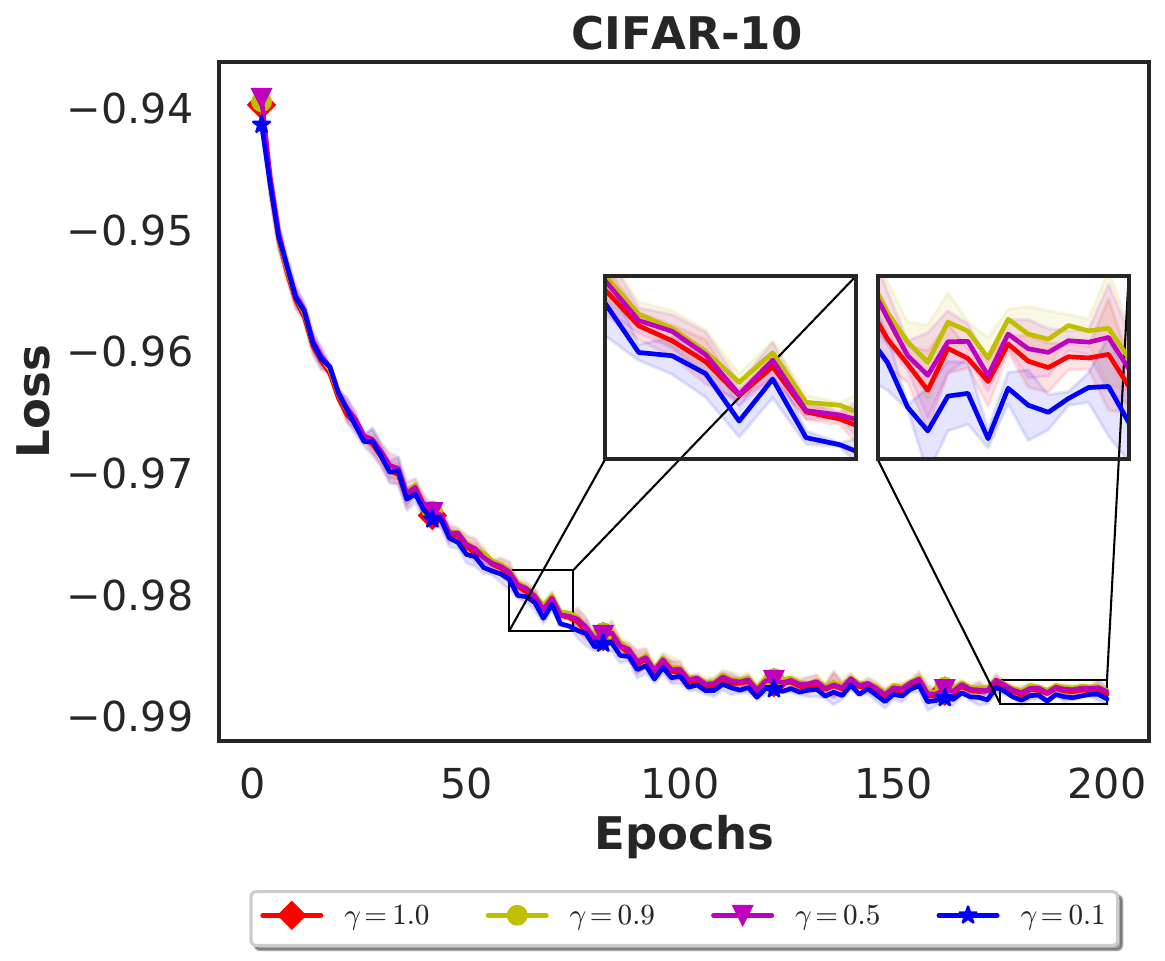}
	}
	\hfill
	\subfigure[vs. Baselines]{	\centering
		\includegraphics[width=0.3\linewidth]{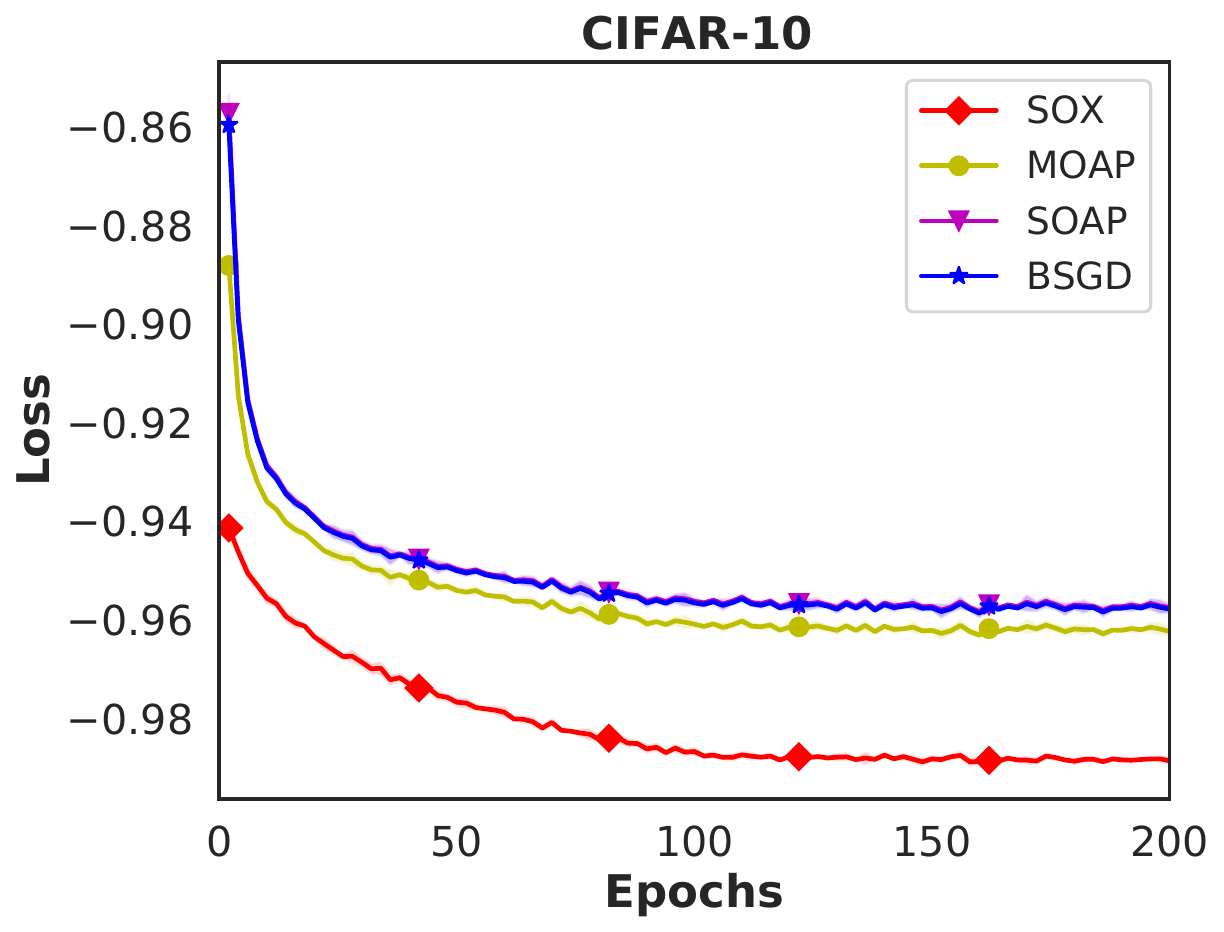}
	}
	\hfill
		\subfigure[Varying $B_1$]{
		\centering
		\includegraphics[width=0.3\linewidth]{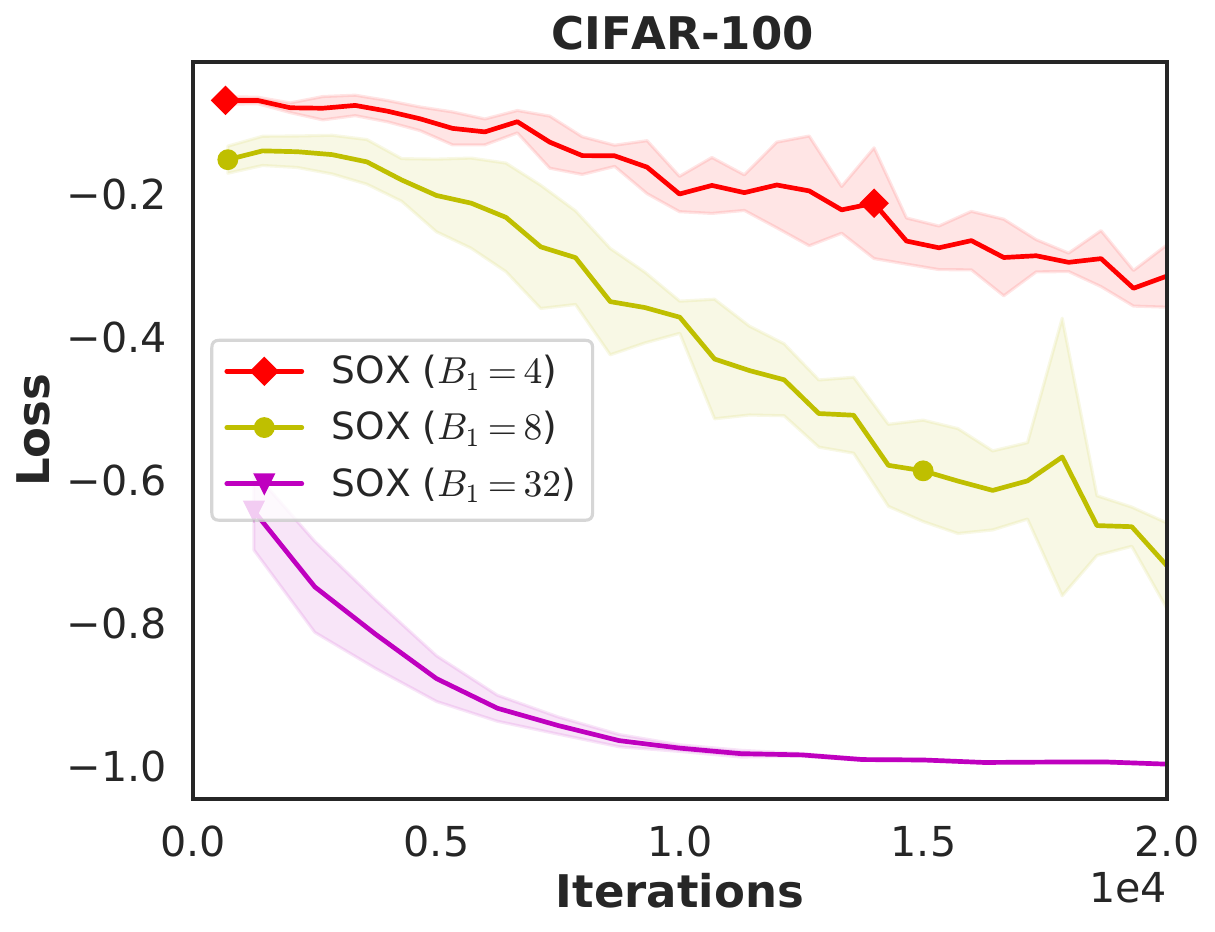}
	}
	\hfill
	\subfigure[Varying $B$]{	\centering
		\includegraphics[width=0.3\linewidth]{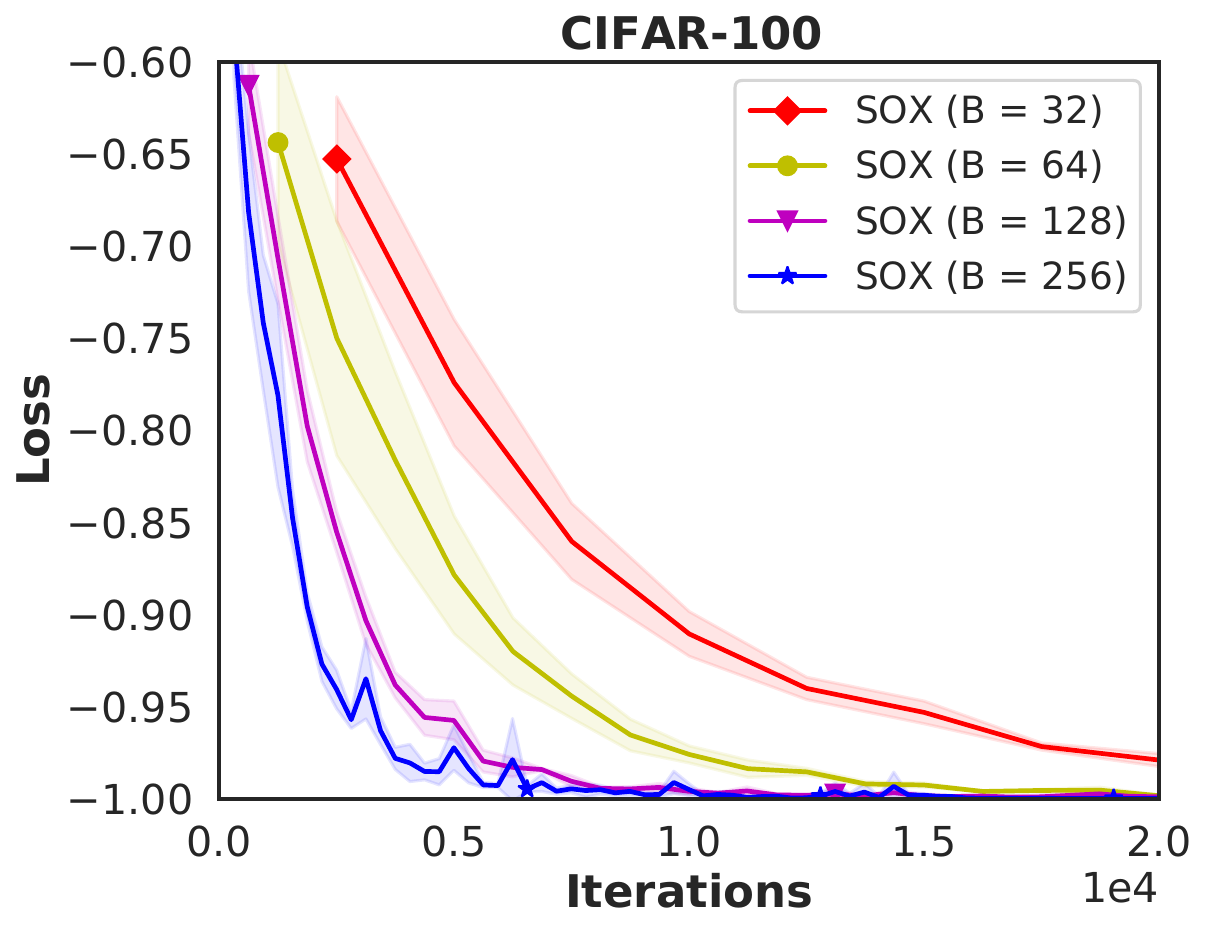}
	}
	\hfill
	\subfigure[Varying $\gamma$]{	\centering
		\includegraphics[width=0.3\linewidth]{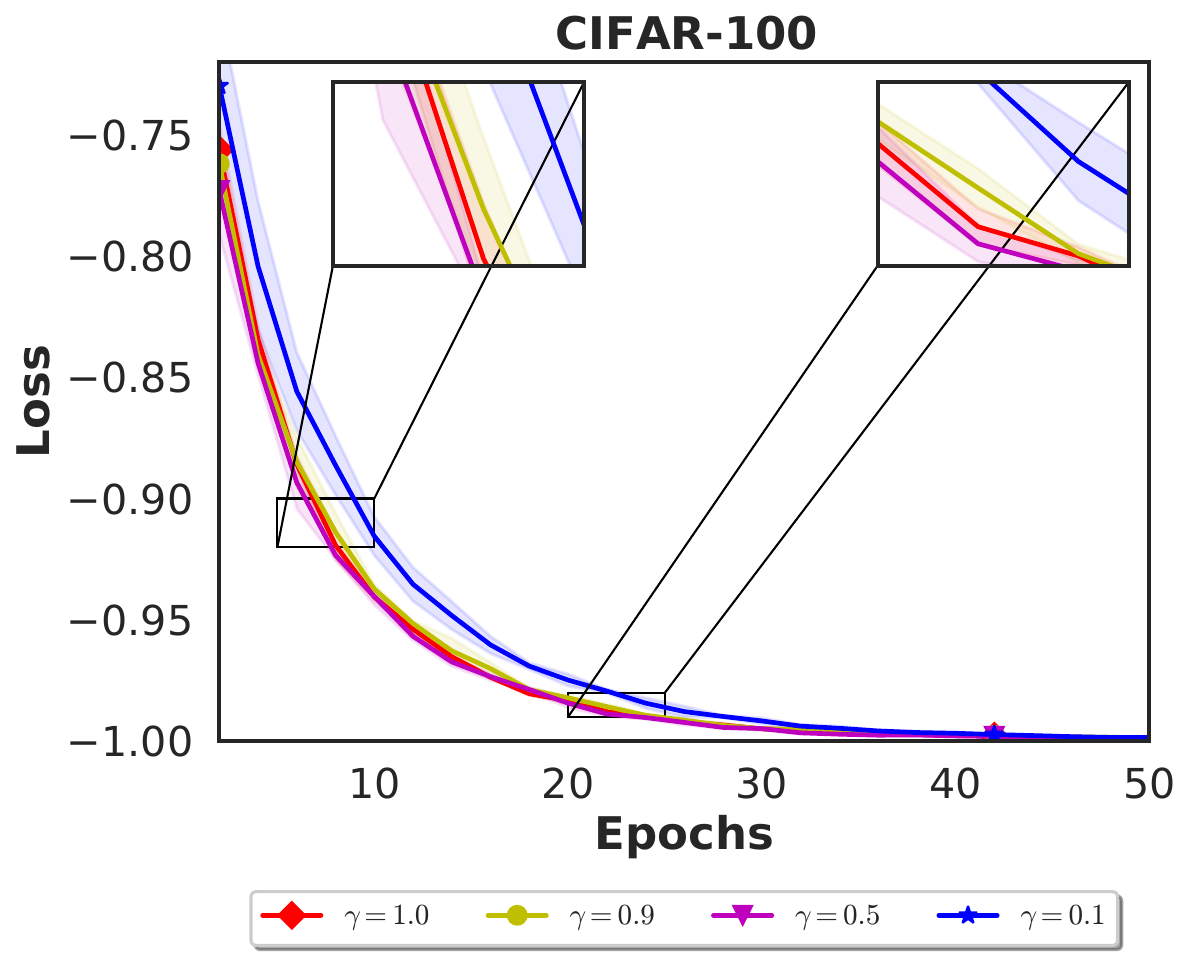}
	}
	\hfill
	\subfigure[vs. Baselines]{	\centering
		\includegraphics[width=0.3\linewidth]{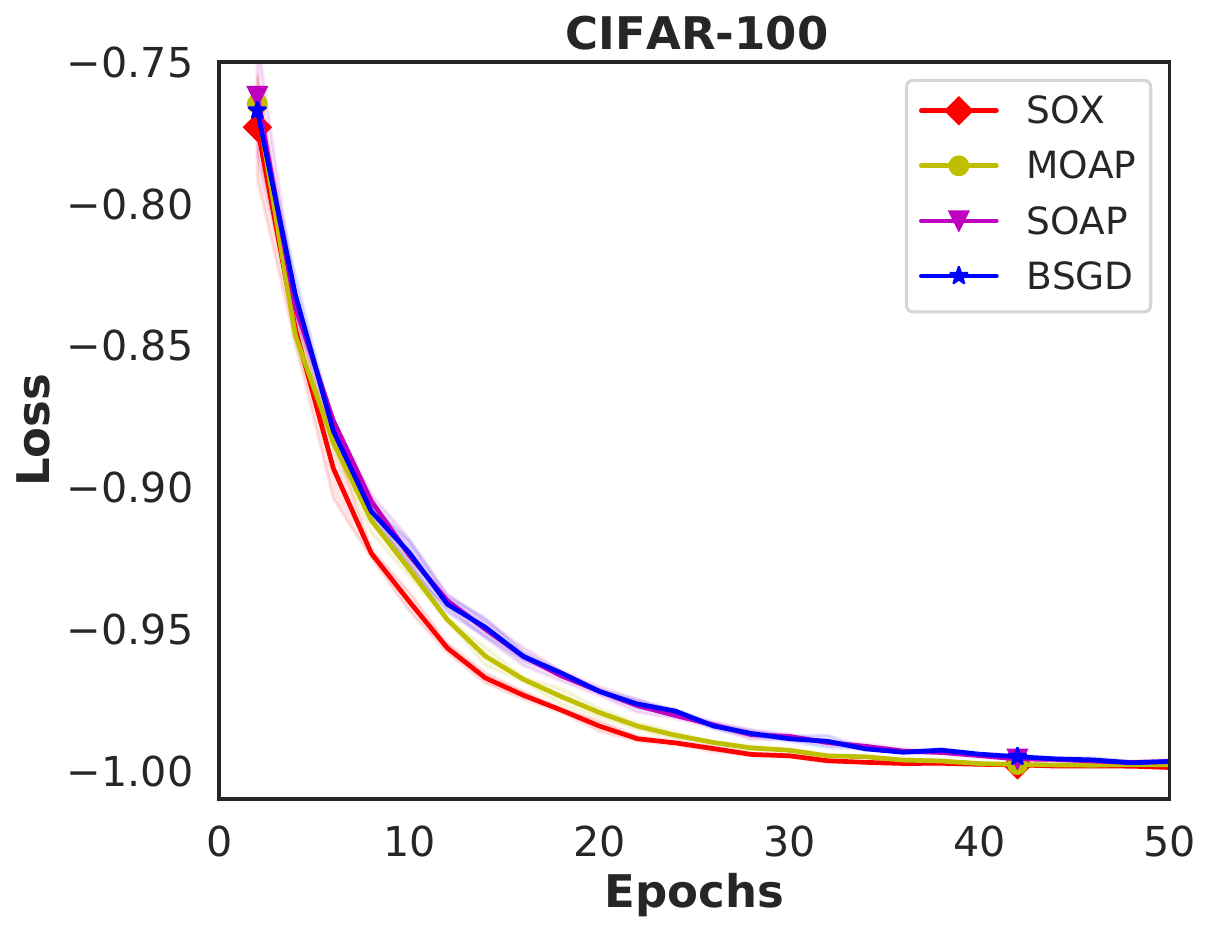}
	}
	\caption{Results of AP maximization on the CIFAR-10 and CIFAR-100 data sets.}
	\label{fig:AP_tr_loss}
\end{figure*}

{\bf Setting.} We conduct experiments on two image datasets, namely CIFAR-10, CIFAR-100. We use the dataloader provided in the released code of~\citet{qi2021stochastic}, which constructs the imbalanced versions of binarized CIFAR-10 and CIFAR-100. We consider two tasks: training ResNet18 on the CIFAR-10 data set and training ResNet34 on the CIFAR-100 data set. We follow the same procedure as in~\citet{qi2021stochastic} that first pre-trains the network by optimizing a cross-entropy loss and then fine-tunes all layers with the randomly initialized classification layer. We also use the same squared hinge loss as in~\citet{qi2021stochastic}. We aim to answer the following four questions related to our theory: {\bf Q1}: Given a batch size $B$, what is the best value for $B_1, B_2$, i.e., the sizes of $\B_1^t$ and $\B_2^t$? {\bf Q2}: Is there parallel speed-up by increasing the total batch size $B=B_1+B_2$? {\bf Q3}: What is the  best value of $\gamma$?   {\bf Q4}: Does SOX converge faster than SOAP (SGD-style) and MOAP? In all experiments, we tune the initial learning rate in a range $10^{-4:1:-1}$ to achieve the best validation error, and decrease the learning rate at $50\%$ and $75\%$ of total epochs. The experiments are performed on a node of a cluster with single GeForce RTX 2080 Ti GPU. We tune the value of $\gamma$ and fix $\beta=0.1$ (same as the default value 0.9 of gradient momentum).

$\bullet$ To answer Q1, we fix the total batch size $B$ as $64$ and vary $B_1$ in the range $\{4, 8, 16, 32\}$. The curves of training losses are shown in Figure~\ref{fig:AP_tr_loss}(a) and (e) on the two datasets. We can see that when $B_1=32=B/2$ SOX has the fastest convergence in terms of number of iterations. This is consistent with our convergence theory. 

$\bullet$ To answer Q2, we fix $B_1=B_2$ and vary $B$ in the range $\{32, 64, 128, 256\}$. The curves of training losses are shown in Figure~\ref{fig:AP_tr_loss}(b) and (f) on the two datasets. We can see that the iteration complexity of SOX decreases as $B$ increases, which  is also consistent with our convergence theory. 

$\bullet$ To answer Q3, we fix $B_1=B_2=B/2=32$ and run SOX with different values of $\gamma$. We can see that $\gamma=1$ does not give the best result, which means the na\"ive mini-batch estimation of $g_i(\w)$ is worse than the moving average estimator with a proper value of $\gamma$. Moreover, we also observe that the best value of $\gamma$ depends on the task: $\gamma=0.1, 0.5$  give the fastest convergence on training ResNet18 with CIFAR-10 and ResNet34 with CIFAR-100, respectively. 

$\bullet$ The Figure~\ref{fig:AP_tr_loss} (d) and (h) answer Q4, which indicates that SOX converges faster than MOAP, which is faster than SOAP (SGD-style) and BSGD. 
\begin{table}[htp]
	\centering 	\caption{Test AP comparison among SOX and the baselines on the AP maximization task. }
	\label{tab:ap_test}	\scalebox{0.9}{		\begin{tabular}{ccccc}
			\toprule[.1em]
			\multicolumn{5}{c}{Dataset: CIFAR-10}\\
			\midrule
			Metrics & MOAP & BSGD & SOAP & SOX \\\midrule[.05em]
			Test AP ($\uparrow$) & 0.763 $\pm$ 0.001 &0.762 $\pm$ 0.001 &0.762 $\pm$ 0.001 & {\bf 0.765 $\pm$ 0.001} \\\cmidrule(lr){1-1}
			\#Epoch ($\downarrow$)&13.0 $\pm$ 4.3 & 15.7 $\pm$ 1.9 &15.7 $\pm$ 1.9 &{\bf 7.0 $\pm$ 3.3} \\
			\midrule
			\multicolumn{5}{c}{Dataset: CIFAR-100}\\
			\midrule
			Metrics & MOAP & BSGD & SOAP & SOX \\\midrule[.05em]
			Test AP ($\uparrow$) & 0.584 $\pm$ 0.010 &0.582 $\pm$ 0.005&0.575 $\pm$ 0.017& {\bf 0.597 $\pm$ 0.012} \\\cmidrule(lr){1-1}
			\#Epoch ($\downarrow$)&17.0 $\pm$ 1.6&{\bf 3.7 $\pm$ 0.9}&11.7 $\pm$ 6.6 & 5.0 $\pm$ 2.8 \\
			\bottomrule
	\end{tabular}}
\end{table}
The curves of average precision on the training data can be found in Figure~\ref{fig:AP_tr_AP} of the Appendix. We also report the test AP of SOX with baselines on CIFAR-10 and CIFAR-100 datasets in Table~\ref{tab:ap_test}. Note that the CIFAR-10 and CIFAR-100 test datasets are \emph{balanced} while our training datasets are imbalanced. Thus, there might be a distribution shift between the training and test datasets. To prevent overfitting, algorithms are early stopped when the validation loss reaches the minimum. The results indicate that SOX converges to a better solution using an overall fewer number of epochs.

\subsection{$p$-norm Push with Concentration at the Top}

In the bipartite ranking problem,  the $p$-norm push objective~\citep{rudin2009p}  can be defined as 
\begin{align*}
F(\w) = \frac{1}{|\S_-|}\sum_{\z_i\in\S_-}\left(\frac{1}{|\S_+|}\sum_{\z_j\in\S_+}\ell(h_\w(\z_j) - h_\w(\z_i))\right)^p,    
\end{align*}
where $p>1$ and $\ell(\cdot)$ is similar as above.  We can cast this function into FCCR by defining $\D=\S_+$, $\S_i=\S_-$,  $g_i(\w)= \frac{1}{|\S_+|}\sum_{\z_j\in\S_+}\ell(h_\w(\z_j) - h_\w(\z_i))$ that couples each positive example $\z_i$ with all negative samples, $f(g)= g^p$. Note that $f$ is monotonically increasing and convex while $g_i$ is convex given that $\ell$ is convex. 
\citet{rudin2009p} only provide a boosting-style $p$-norm push algorithm (BS-PnP), which is not scalable because it processes all $|\S_+|$ positive and $|\S_-|$ negative instances at each iteration. 
\begin{table}[htp]
	\centering 	\caption{Comparison among SOX and the baselines BS-PnP, BSGD for optimizing $p$-norm Push for learning a linear model. }
	\label{tab:pnp_compare_short}	\scalebox{0.9}{	
		\begin{tabular}{cccc}
			\toprule[.1em]
			\multicolumn{4}{c}{covtype} \\
			\midrule[0.01em]
			Algorithms & BS-PnP & BSGD & SOX \\\midrule[.05em]
			Test Loss ($\downarrow$) & 0.778 &0.625 $\pm$ 0.018 &{\bf 0.516 $\pm$ 0.003}\\\cmidrule(lr){1-1}
			Time (s) ($\downarrow$)&6043.90 & {\bf 4.20 $\pm$ 0.08}&4.62 $\pm$ 0.10 \\	\midrule[0.1em]
			\multicolumn{4}{c}{ijcnn1} \\
			\midrule[0.01em]
			Algorithms & BS-PnP & BSGD & SOX \\\midrule[.05em]
			Test Loss ($\downarrow$) & 0.268 &0.202 $\pm$ 0.001 &{\bf 0.128 $\pm$ 0.002}\\\cmidrule(lr){1-1}
			Time (s) ($\downarrow$)&648.06 & {\bf 4.02 $\pm$ 0.04}&4.15 $\pm$ 0.06 \\
			\bottomrule
	\end{tabular}}
\end{table}
We compare SOX with the BS-PnP, and the baselines BSGD~\citep{hu2020biased}. Besides,  SOAP~\citep{qi2021stochastic} and MOAP~\citep{wang2021momentum}, which were originally designed for the AP maximization, can also be applied to the $p$-norm push problem \footnote{Due to limited space, the comparison with SOAP and MOAP can be found in  Table~\ref{tab:pnp_compare} of the Appendix.}. Following \citet{rudin2009p}, we choose $\ell(\cdot)$ to be the exponential function. We conduct our experiment on two LibSVM datasets: covtype and ijcnn1. For both datasets, we randomly choose 90\% of the data for training and the rest of data is for testing. For this experiment, we learn a linear ranking function $h_\w(\x) = \inner{\w}{\x}$ and $p=4$. For each algorithm, we run it with 5 different random seeds and report the average test loss with standard deviation. Besides, we also report the running time. For the stochastic algorithms (BSGD, SOAP, MOAP, SOX), we choose $B=64$ and $B_1=B_2$. The algorithms are implemented with Python and run on a server with 12-core Intel(R) Xeon(R) CPU E5-2697 v2 @ 2.70GHz. 

As shown in Table~\ref{tab:pnp_compare_short}, the BS-PnP algorithm is indeed not scalable and takes much longer time than the stochastic algorithms. Moreover, our SOX is consistently better than BSGD in terms of test loss.

\subsection{Neighborhood Component Analysis}\label{sec:nca}
\begin{figure*}
	\begin{minipage}{0.30\linewidth}
		\includegraphics[width=\linewidth]{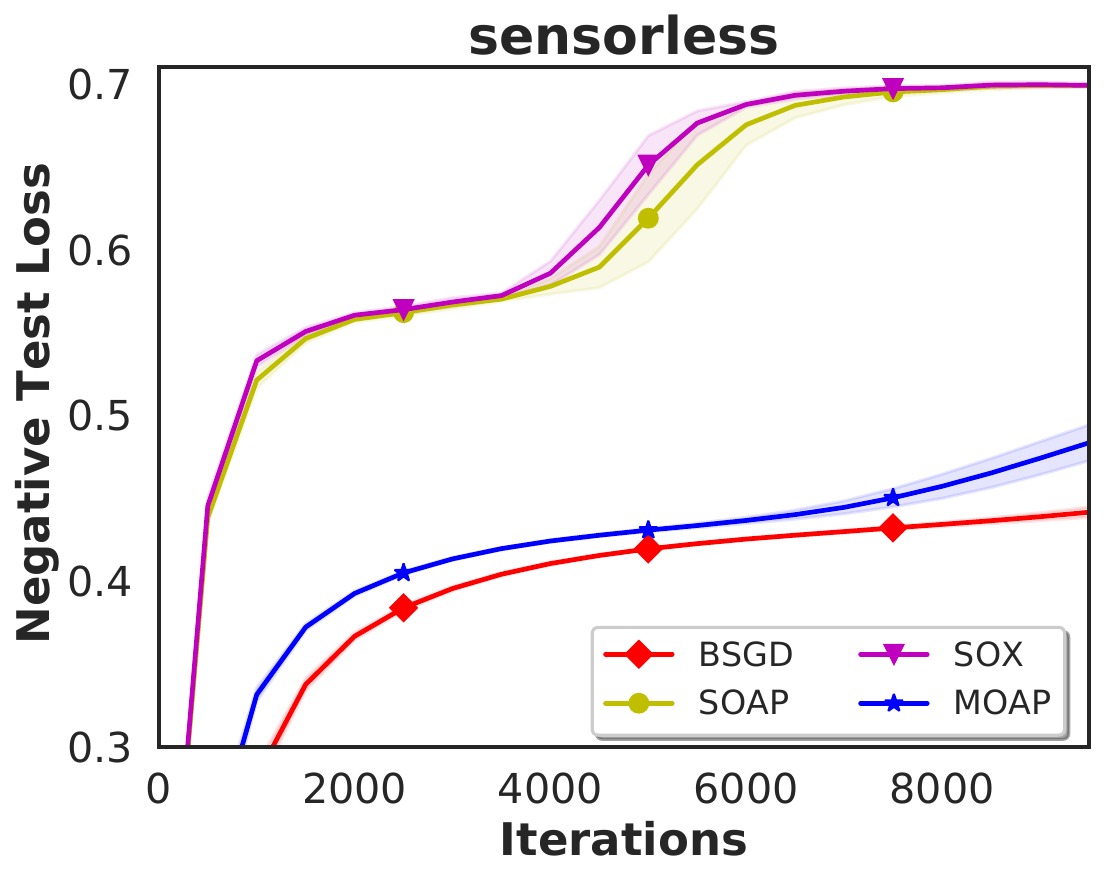}
	\end{minipage}
	\hfill
	\begin{minipage}{0.30\linewidth}
		\includegraphics[width=\linewidth]{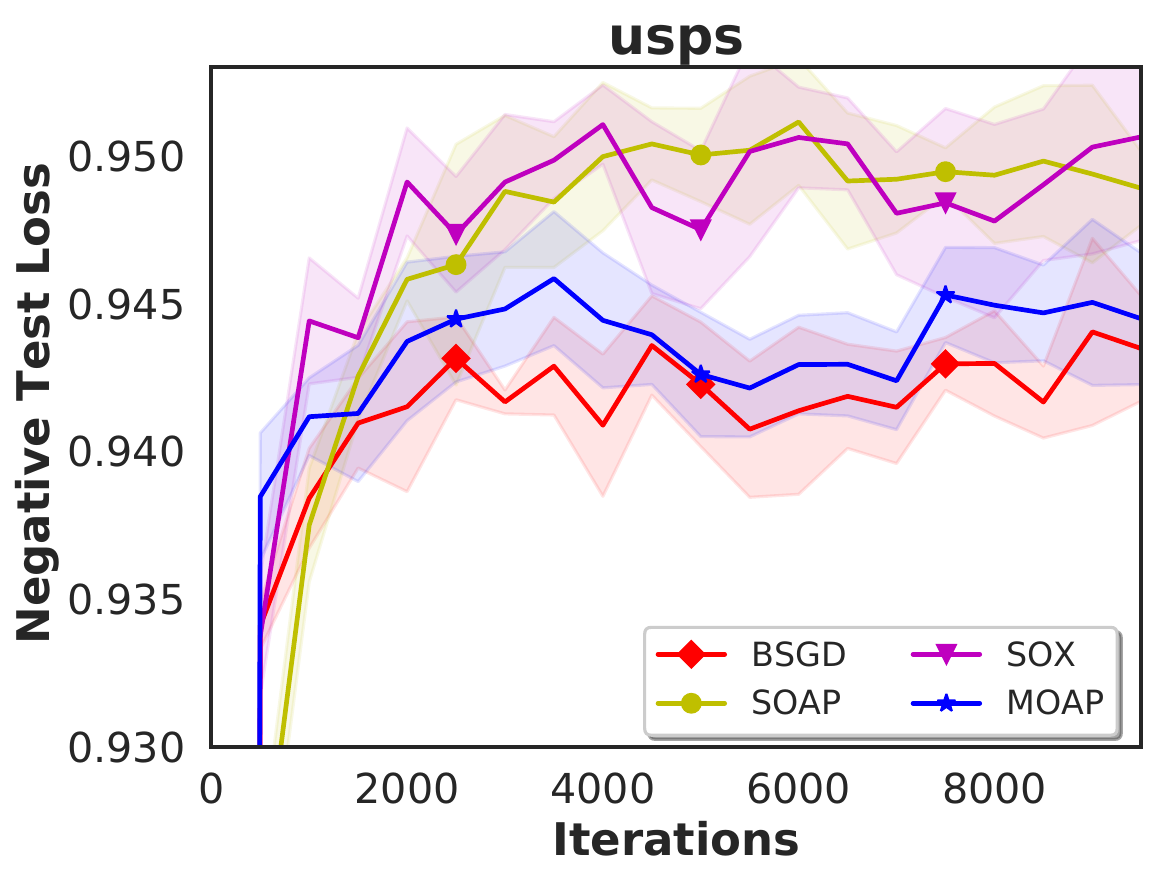}
	\end{minipage}
	\hfill
	\begin{minipage}{0.30\linewidth}
		\includegraphics[width=\linewidth]{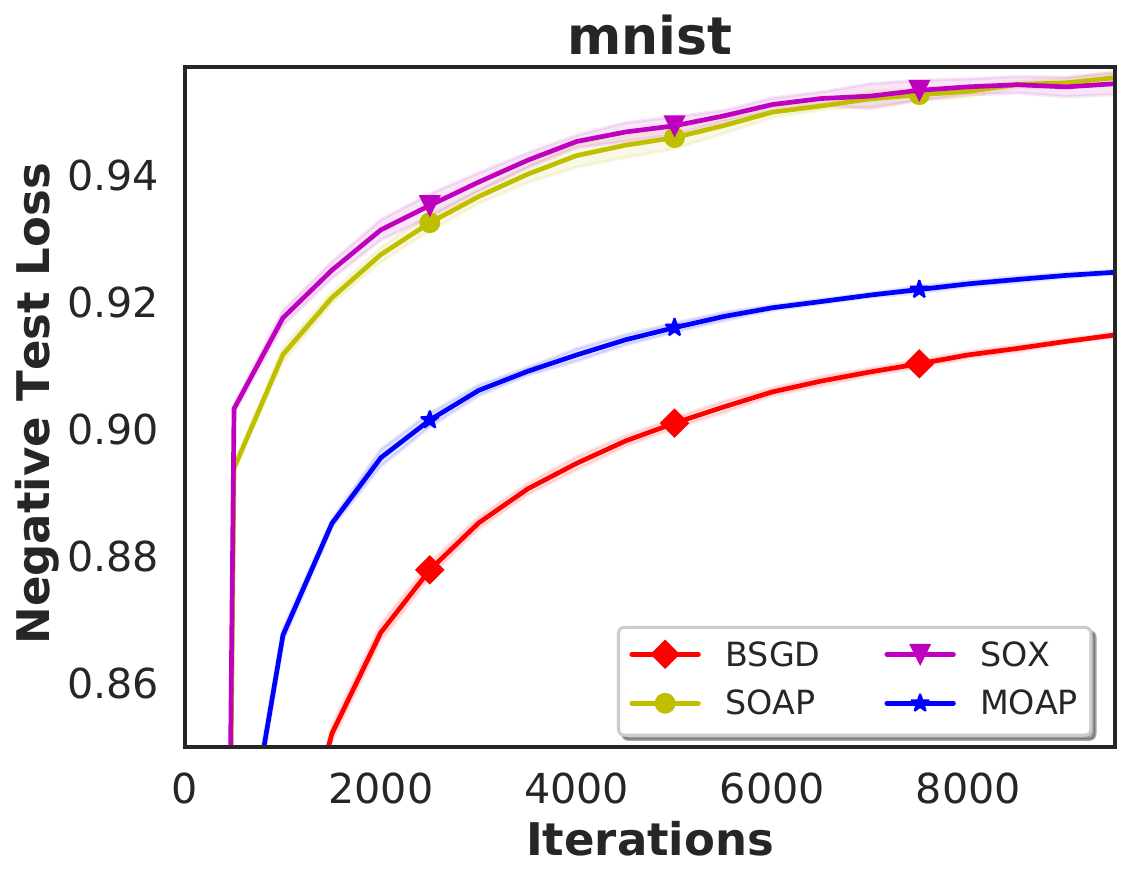}
	\end{minipage}
	\hfill
	\caption{Results of neighborhood component analysis on three datasets.}
	\label{fig:NCA_te_loss}
\end{figure*}

Neighborhood Component Analysis (NCA) was proposed in~\citet{goldberger2004neighbourhood} for learning a Mahalanobis
distance measure. Given a set of data points $\D=\{\x_1,\ldots, \x_n\}$, where each data $\x_i$ has a class label $y_i$. The objective of NCA is defined as 
\begin{align}\label{eq:nca}
	F(A)=-\sum_{\x_i\in\D} \frac{\sum_{\x\in\C_i}\exp(-\|A\x_i - A\x\|^2)}{\sum_{\x\in\S_i}\exp(-\|A\x_i - A\x\|^2)},
\end{align}
where $\C_i=\{\x_j\in\D: y_j=y_i\}$ and  $\S_i =\D\setminus\{\x_i\}$. We can map the above objective as an FCCR by defining  $g_i(A) = [\sum_{\x\in\C_i}\exp(-\|A\x_i - A\x\|^2), \sum_{\x\in\S_i}\exp(-\|A\x_i - A\x\|^2)]$  and $f(g_i(A))=-\frac{[g_i(A)]_1}{[g_i(A)]_2}$. 
The problem \eqref{eq:nca} can be solved by the gradient descent method. However, the exact gradient computation could be expensive or even infeasible when $|\D|$ is large. A widely used stochastic algorithm is to sample a mini-batch $\B\subseteq \D$ and replace $\C_i$ and $\S_i$ by $\C_i\cap \B$ and $\S_i\cap \B$, respectively, which is equivalent to the BSGD algorithm~\citep{hu2020biased}. Besides, SOAP~\citep{qi2021stochastic}, MOAP~\citep{wang2021momentum} and our SOX algorithm are also applicable to \eqref{eq:nca}.

The experiment is performed on three datasets: sensorless, usps, and mnist from the LibSVM~\citep{chang2011libsvm}. For each dataset, we randomly choose 90\% of the data for training and the rest as test data. Each algorithm is executed for 5 runs with different random seeds. We report the average test loss with standard deviation. For all algorithms, we choose batch size to be 64. As shown in Figure~\ref{fig:NCA_te_loss}, our SOX method outperforms previous methods on those datasets.

\section{More Applications of SOX}\label{sec:app} In this section, we present more applications of the proposed algorithm in machine learning, and highlight the potential of the proposed algorithm in addressing their computational challenges. Providing experimental results of these applications is beyond the scope of this paper.  


{\bf Listwise Ranking Objectives/Measures.} In learning to rank (LTR), we are given a set of queries $\Q=\{\q_1, \ldots, \q_n\}$. For each query, a set of items with relevance scores are provided $\S_q=\{(\x^q_1, y^q_1), \ldots, (\x^q_{n_q}, y^q_{n_q})\}$, where $\x^q_i$ denotes the input data, and  $y^q_i\in\R^+$ denotes its a relevance score with $y^q_i=0$ meaning irrelevant. For LTR, there are many listwise objectives and measures that can be formulated as FCCR, e.g., ListNet~\citep{ListNet}, ListMLE~\citep{ListMLE} , NDCG~\citep{wang2013theoretical}. Due to the limited space, we only consider that of ListNet. The objective function of ListNet can be  defined by a cross-entropy loss between two probabilities of list of scores:
\begin{align*}
	&F(\w) = -\sum_{q}\sum_{\x^q_i\in\S_q}P(y^q_i)\log \frac{\exp(h_\w(\x^q_i; \q)}{\sum_{\x\in\S_q}\exp(h_\w(\x; \q))},
\end{align*}
where $h_\w(\x^q_i; \q)$ denotes the prediction score of the item $\x^q_i$ with respect to the query $\q$, $P(y^q_i)$ denotes a probability for a relevance score $y^q_i$ (e.g., $P(y^q_i)\propto y^q_i$). We can map the above function into FCCR, where $g(\w; \x^q_i, \S_q) = \frac{1}{|\S_q|}\sum_{\x\in\S_q}\exp(h_\w(\x; \q)-h_\w(\x^q_i; \q))$ and $f(g) = \log (g)$, $\D=\{(\q, \x^q_i): P(y^q_i)>0\}$.  The original paper of ListNet uses a gradient method for optimizing the above objective, which has a complexity of $O(|\Q||\S_q|)$ and  is inefficient when $\S_q$ contains a large number of items.  

{\bf Deep Survival Analysis (DSA).} The survival analysis in medicine is to explore and understand the relationships between patients’ covariates (e.g., clinical and genetic features) and the effectiveness of various treatment options. Using the Cox model for modeling the hazard function, the negative log-likelihood can be written as~\citep{deepsurv}:
\begin{align*}
	F(\w) = \frac{1}{n}\sum_{i: E_i=1} \log\left( \sum_{j\in\S(T_i)}\exp(h_\w(\x_j)-h_\w(\x_i)\right),
\end{align*}
where $\x_i$ denote the input feature of a patient,  $h_\w(\x_i)$ denotes the risk value predicted by the network,  $E_i=1$ denotes an observable event of interest (e.g., death), $T_i$ denotes the time interval  between the time in which the baseline data was collected and the time of the event occurring, and $\S(t)=\{i: T_i\geq t\}$ denotes the set of patients still at risk of failure at time $t$. This is similar to the objective of ListMLE. The proposed algorithm is appropriate when both $\{i:E_i=1\}$ and $\S(T_i)$ are large. 

{\bf Deep Latent Variable Models (DLVM).}  Latent variable models refer to a family of generative models that use latent variables to model the observed data, where we consider the supervised learning setting. 
In particular, given a set of observed data $\D=\{(\x_1, y_1), \ldots, (\x_n, y_n)\}$, we model the probability of $\Pr(y|\x)$ by introducing a discrete latent variable $\z$, i.e.,  $\Pr(y|\x)=\sum_{\z\in\Z}\Pr(y|\x, \z)\Pr(\z|\x)$, where $\Z$ denotes the support set of the latent variable $\z$ and both $\Pr(y|\x, \z)$ and $\Pr(\z|\x)$ could be parameterized by a deep neural network. Then by minimizing negative log-likelihood of observed data, we have the objective function $
	F(\w) = -\sum_{(\x_i, y_i)\in\D} \log \sum_{\z\in\Z} \Pr(y_i|\z, \x_i)\Pr(\z|\x_i)$.
When $\Z$ is a large set, evaluating the inner sum is expensive.  While the above problem is traditionally solved by EM-type algorithms, a stochastic algorithm based on backpropogation is used more often in modern deep learning. We consider an application in NLP for retrieve-and-predict language model pre-training~\citep{DBLP:conf/icml/GuuLTPC20}. In particular, $\x_i$ denotes an masked input sentence, $y$ denotes masked tokens, $\z$ denotes a document from a large corpus $\Z$ (e.g., wikipedia). In~\citet{DBLP:conf/icml/GuuLTPC20}, $\Pr(\z|\x_i)= \frac{\exp(E(\x)^{\top}E(\z)) }{\sum_{\z'\in\Z}\exp(E(\x)^{\top}E(\z'))}$, where $E(\cdot)$ is a document embedding network,  and $\Pr(y|\x, \z)$ is computed by a masked language model that a joint embedding $\x, \z$ is used to make the prediction. Hence, we can write $F(\w)$ as
\begin{align*} 
	F(\w) & = -\sum_{i=1}^n \log \sum_{\z\in\Z} \Pr(y_i|\z, \x_i)\exp(E(\x_i)^{\top}E(\z)))  +\sum_{i=1}^n\log(\sum_{\z'\in\Z}\exp(E(\x_i)^{\top}E(\z'))).
\end{align*}
Note that both terms in the above is a special case of FCCR. The proposed algorithm gives an efficient way to solve this problem when $\Z$ is very large. \citet{DBLP:conf/icml/GuuLTPC20} address the challenge by approximating the inner summation by summing over the top $k$ documents with highest probability under $\Pr(\z| \x)$, which is retrieved by using maximum inner product search with a running time and storage space that scale sub-linearly with the number of documents. In contrast, SOX has a complexity independent of the number of documents per-iteration, which  depends on the batch size. 

{\bf Softmax Functions.} One might notice that in the considered problems ListNet, ListMLE, NCA, DSA, DLVM, a common function that causes the difficulty in optimization is the softmax function in the form $\frac{\exp(h(\x_i))}{\sum_{\x\in\X}\exp(h(\x))}$ for a target item $\x_i$ out of a large number  items in $\X$. This also occurs in NLP pre-training methods that predicts masked tokens out of  billions/trillions of tokens~\citep{borgeaud2022improving}. Taking the logarithmic of the softmax function gives the coupled compositional form $\log\sum_{\x'\in\X}\exp(h(\x)-h(\x_i))$, and summing over all items gives the considered FCCR.

\section*{Acknowledgements}

We thank anonymous reviewers and Yao Yao (UIowa) for spotting several mistakes in the proof and Gang Li (UIowa) for discussing the experiments on $p$-norm push optimization. This work is partially supported by NSF Grant 2110545, NSF Career Award 1844403.

\newpage
\appendix

\section{Omitted Experimental Results}

\subsection{AP Maximization}
We provide the curves of training loss in Figure~\ref{fig:AP_tr_loss}. Here we also present the curves of training average precision.
\begin{figure*}[h]
	\subfigure[Varying $B_1$]{
		\centering
		\includegraphics[width=0.3\linewidth]{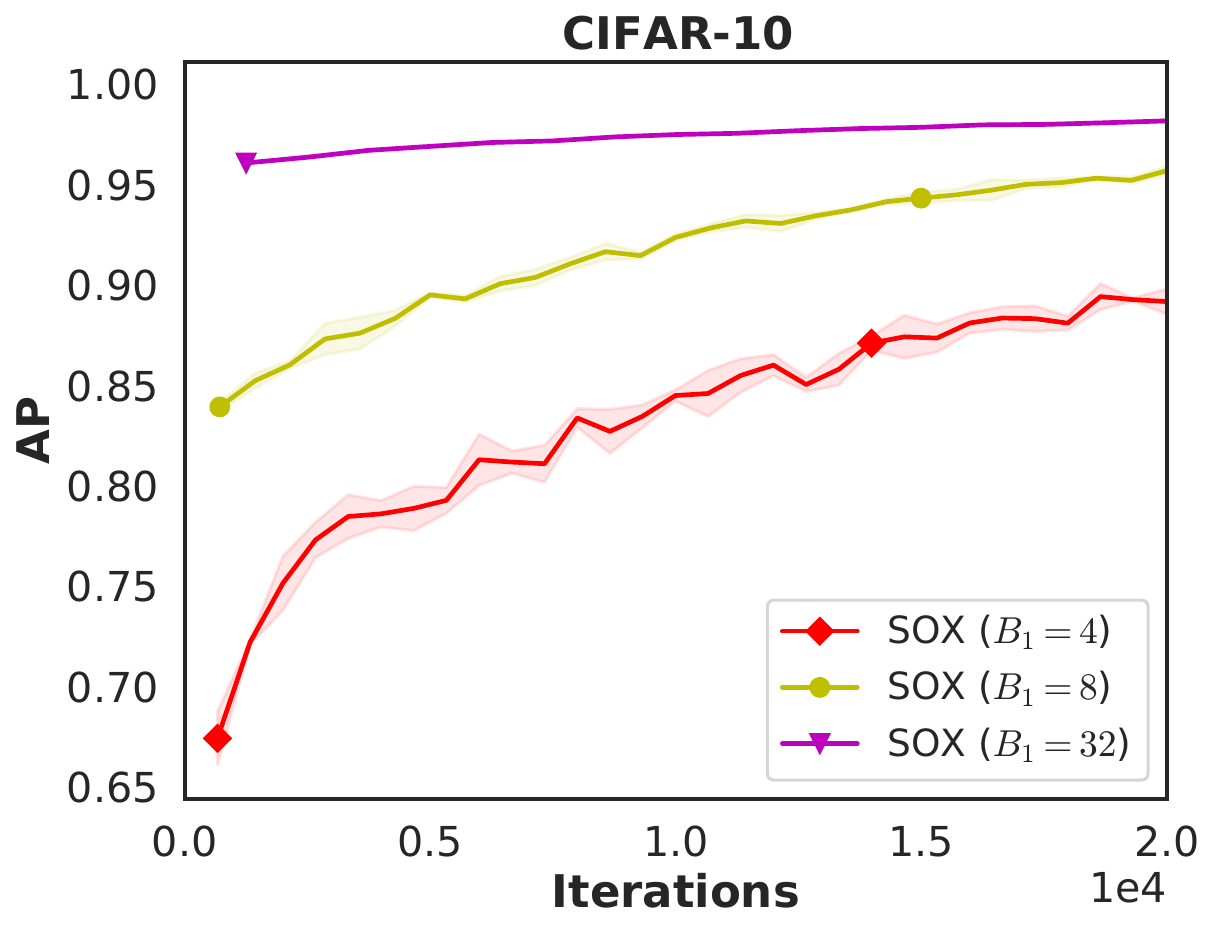}
	}
	\hfill
	\subfigure[Varying $B$]{	\centering
		\includegraphics[width=0.3\linewidth]{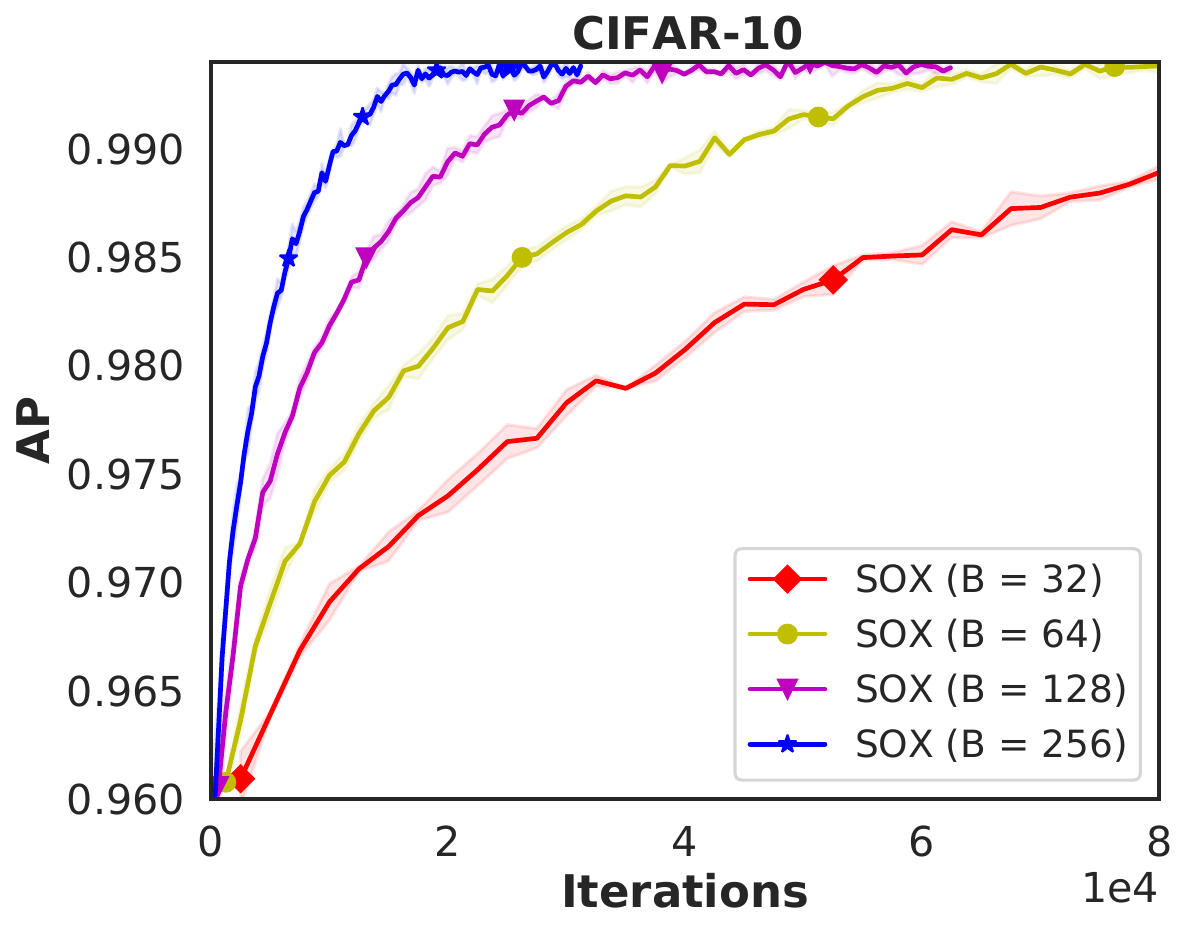}
	}
	\hfill
	\subfigure[Varying $\gamma$]{	\centering
		\includegraphics[width=0.3\linewidth]{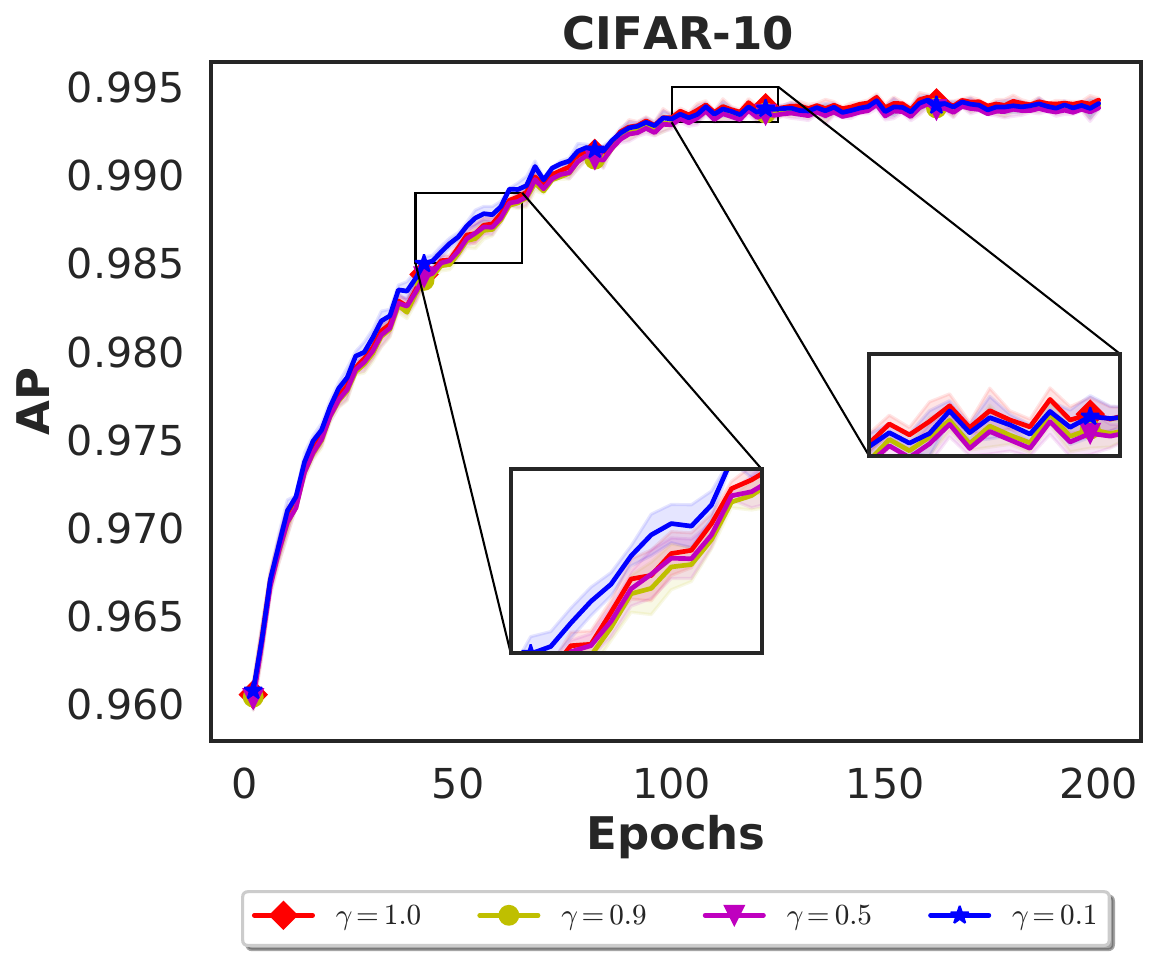}
	}
	\hfill
	\subfigure[vs. Baselines]{	\centering
		\includegraphics[width=0.3\linewidth]{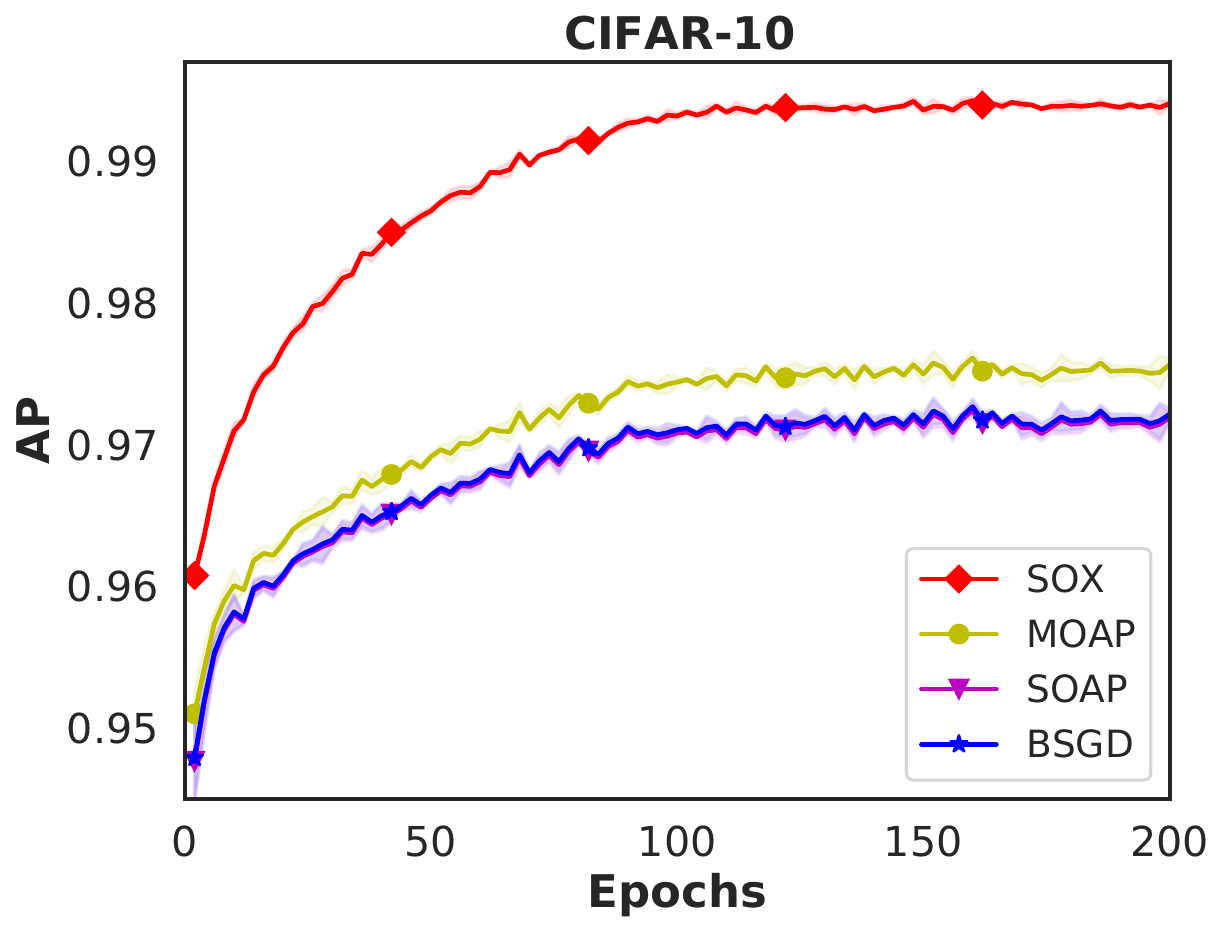}
	}
	\hfill
	\subfigure[Varying $B_1$]{
		\centering
		\includegraphics[width=0.3\linewidth]{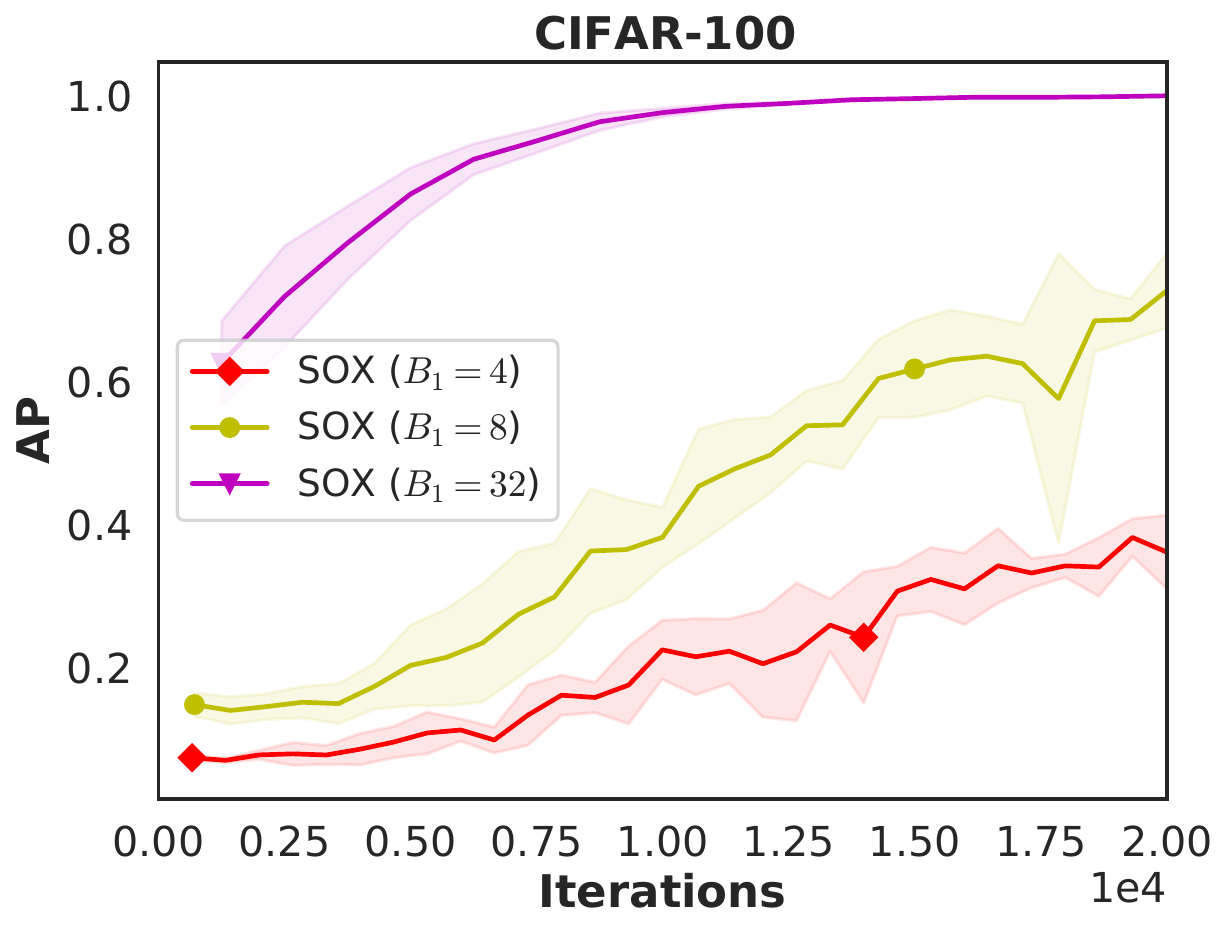}
	}
	\hfill
	\subfigure[Varying $B$]{	\centering
		\includegraphics[width=0.3\linewidth]{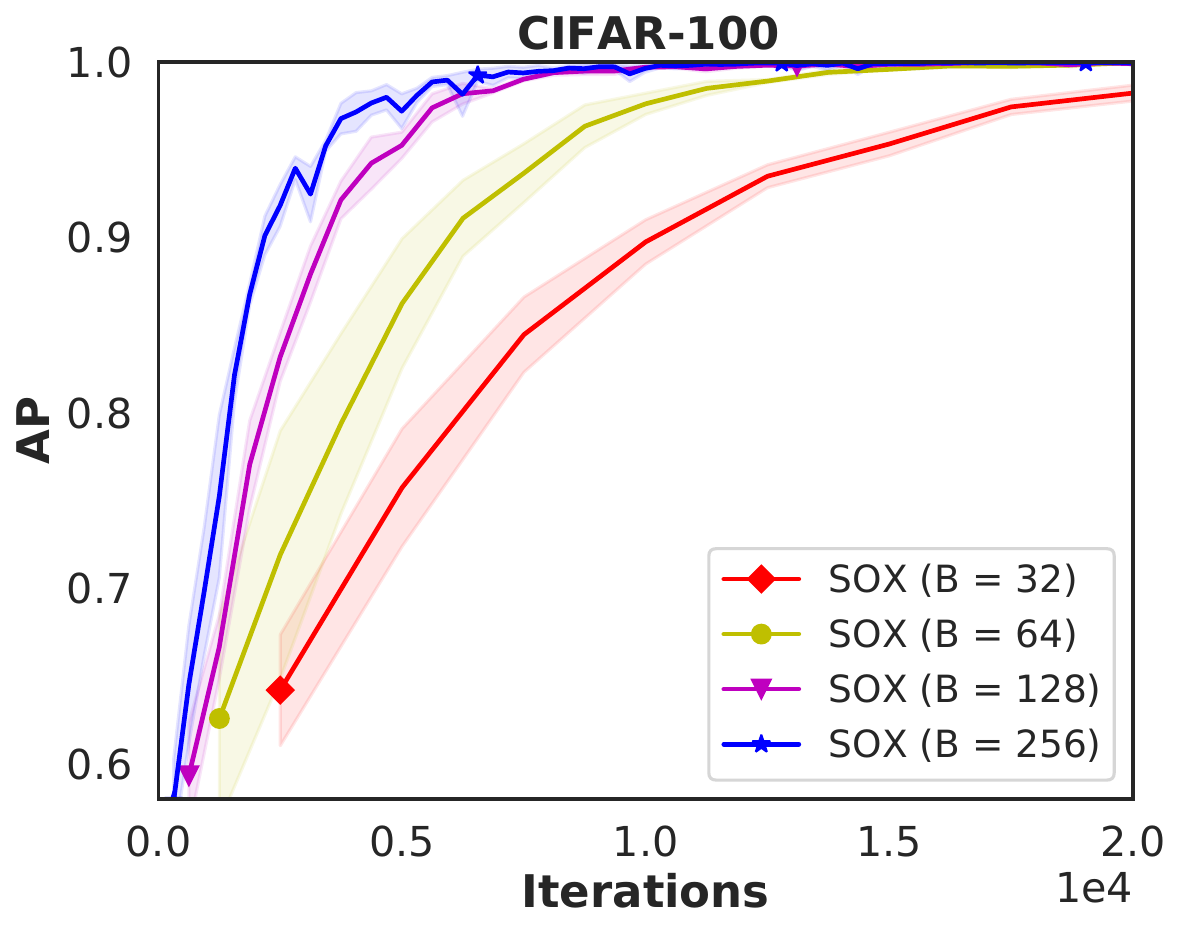}
	}
	\hfill
	\subfigure[Varying $\gamma$]{	\centering
		\includegraphics[width=0.3\linewidth]{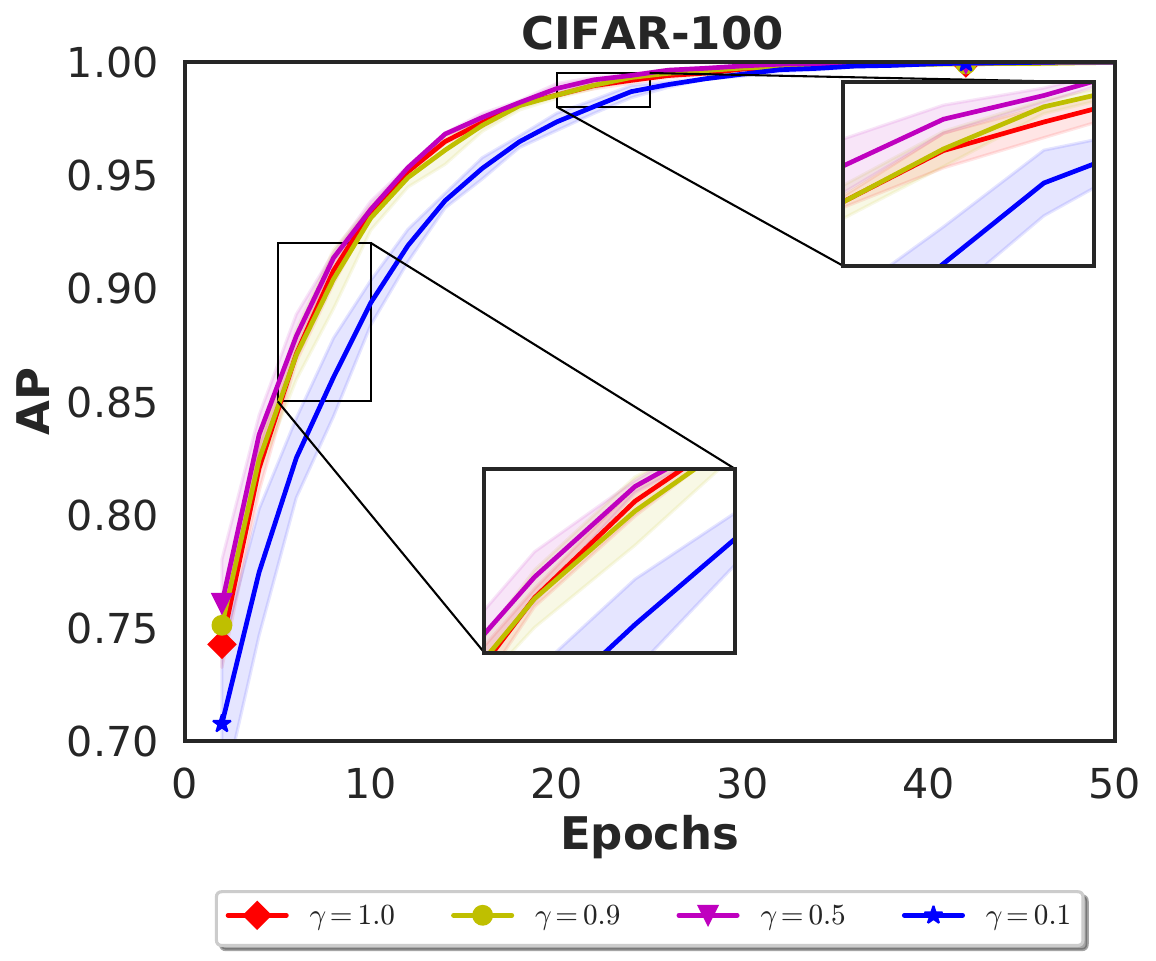}
	}
	\hfill
	\subfigure[vs. Baselines]{	\centering
		\includegraphics[width=0.3\linewidth]{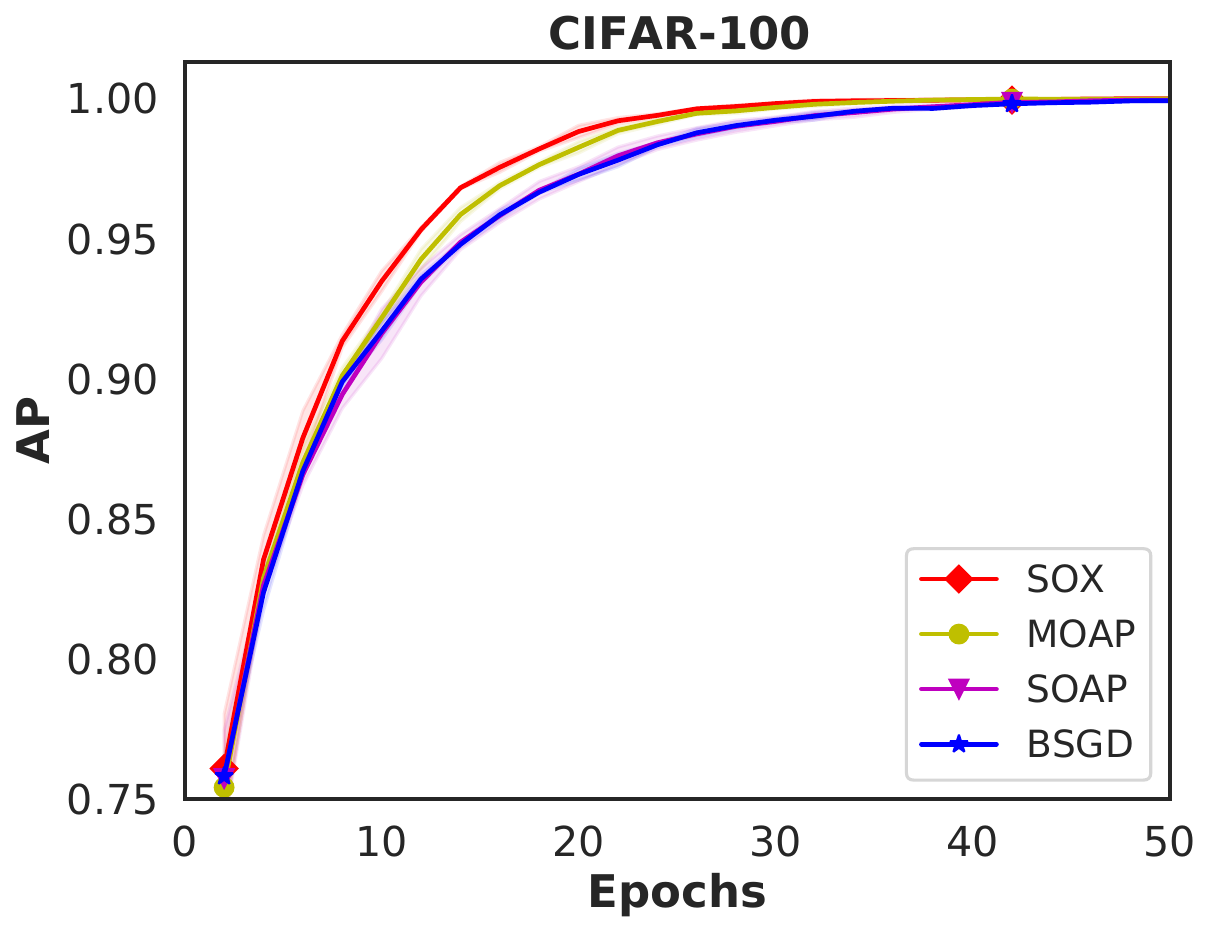}
	}
	\caption{Training average precision curves of AP maximization on the CIFAR-10 and CIFAR-100 data sets.}
	\label{fig:AP_tr_AP}
\end{figure*}

\subsection{Minimizing $p$-norm Push}
Table~\ref{tab:pnp_compare} and Figure~\ref{fig:pnorm-linear} (c)\&(f) show that SOX consistently outperforms BS-PnP/BSGD/MOAP in terms of $p$-norm push loss on the test data. SOX has better performance than SOAP on the covtype data while match its performance on ijcnn1. 
\begin{table}[htp]
	\centering 	\caption{Comparison among SOX and the baselines BS-PnP, BSGD for optimizing $p$-norm Push for learning a linear model. }
	\label{tab:pnp_compare}
	\scalebox{0.9}{	
		\begin{tabular}{cccccc}
			\toprule[.1em]
			\multicolumn{6}{c}{covtype} \\
			\midrule[0.01em]
			Algorithms & BS-PnP & BSGD & SOAP& MOAP& SOX \\\midrule[.05em]
			Test Loss ($\downarrow$) & 0.778 &0.625 $\pm$ 0.018 &0.523 $\pm$ 0.004&0.559 $\pm$ 0.011&{\bf 0.516 $\pm$ 0.003}\\\cmidrule(lr){1-1}
			Time (s) ($\downarrow$)&6043.90 & {\bf 4.20 $\pm$ 0.08}&4.32 $\pm$ 0.15 &4.89 $\pm$ 0.06&4.62 $\pm$ 0.10 \\	\midrule[0.1em]
			\multicolumn{6}{c}{ijcnn1} \\
			\midrule[0.01em]
			Algorithms & BS-PnP & BSGD & SOAP& MOAP& SOX \\\midrule[.05em]
			Test Loss ($\downarrow$) & 0.268 &0.202 $\pm$ 0.001 &{\bf 0.128 $\pm$ 0.002}&0.147 $\pm$ 0.001 &{\bf 0.128 $\pm$ 0.002}\\\cmidrule(lr){1-1}
			Time (s) ($\downarrow$)&648.06 & {\bf 4.02 $\pm$ 0.04}&4.04 $\pm$ 0.11 &4.42 $\pm$ 0.05&4.15 $\pm$ 0.06 \\
			\bottomrule
	\end{tabular}}
\end{table}

\begin{figure*}[h]
	\subfigure[Varying $B$]{
		\includegraphics[width=0.31\linewidth]{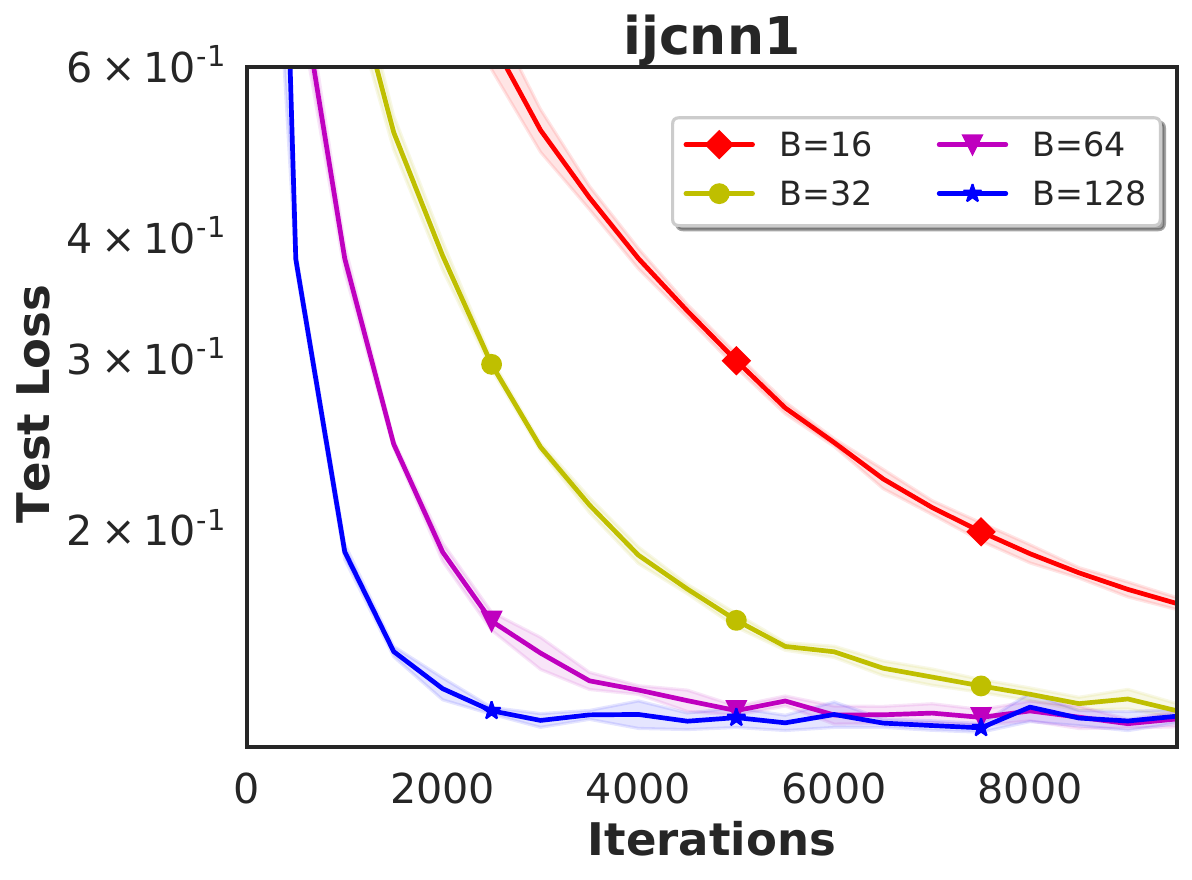}
	}\hfill
	\subfigure[Varying $\gamma$]{
		\includegraphics[width=0.31\linewidth]{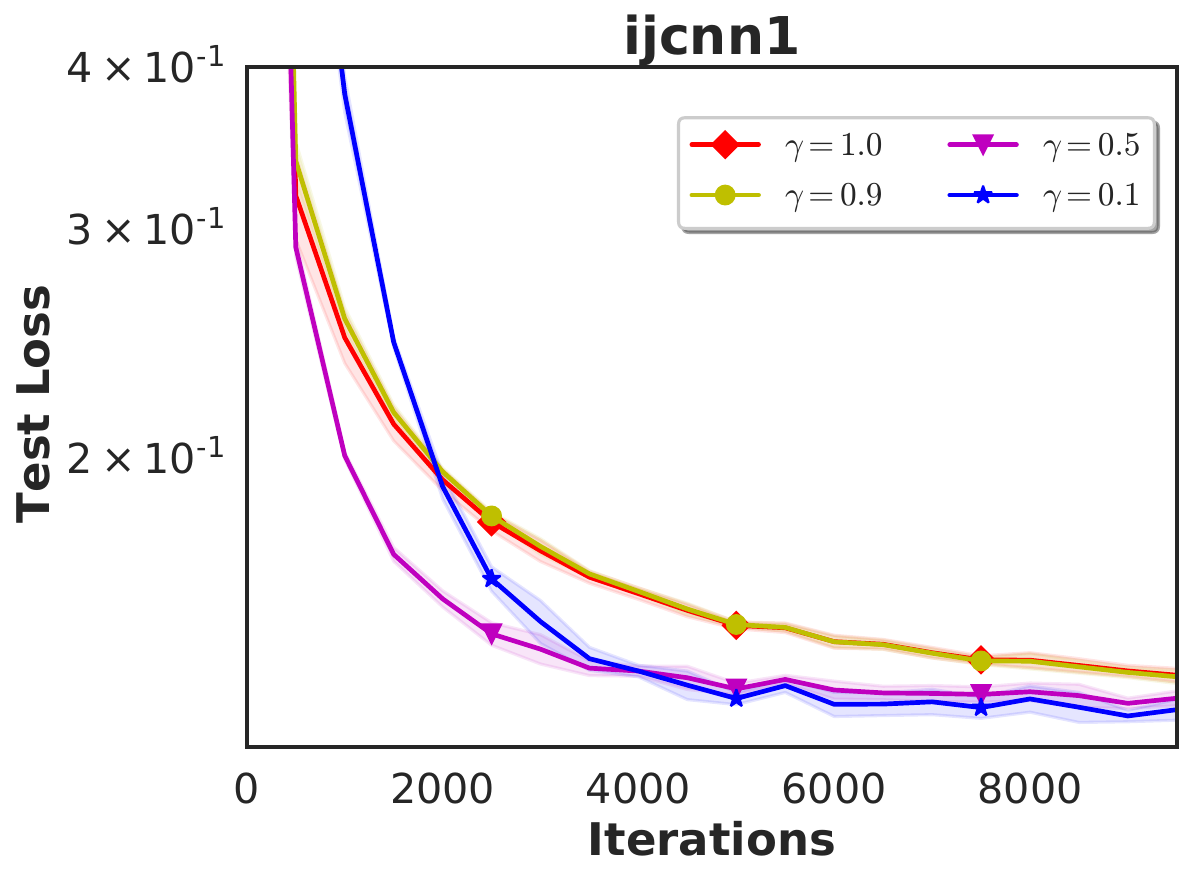}
	}\hfill
	\subfigure[v.s. Baselines]{
		\includegraphics[width=0.31\linewidth]{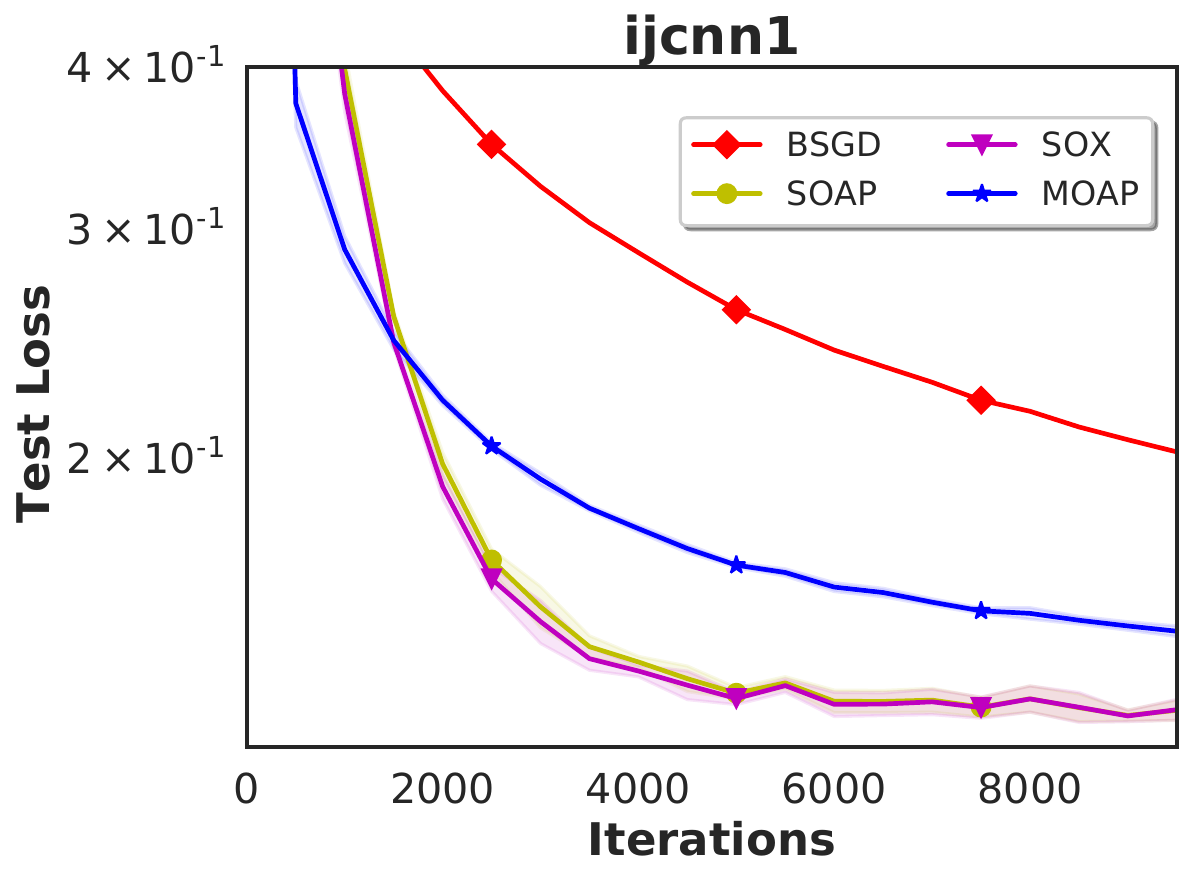}
	}\hfill
	\subfigure[Varying $B$]{
		\includegraphics[width=0.31\linewidth]{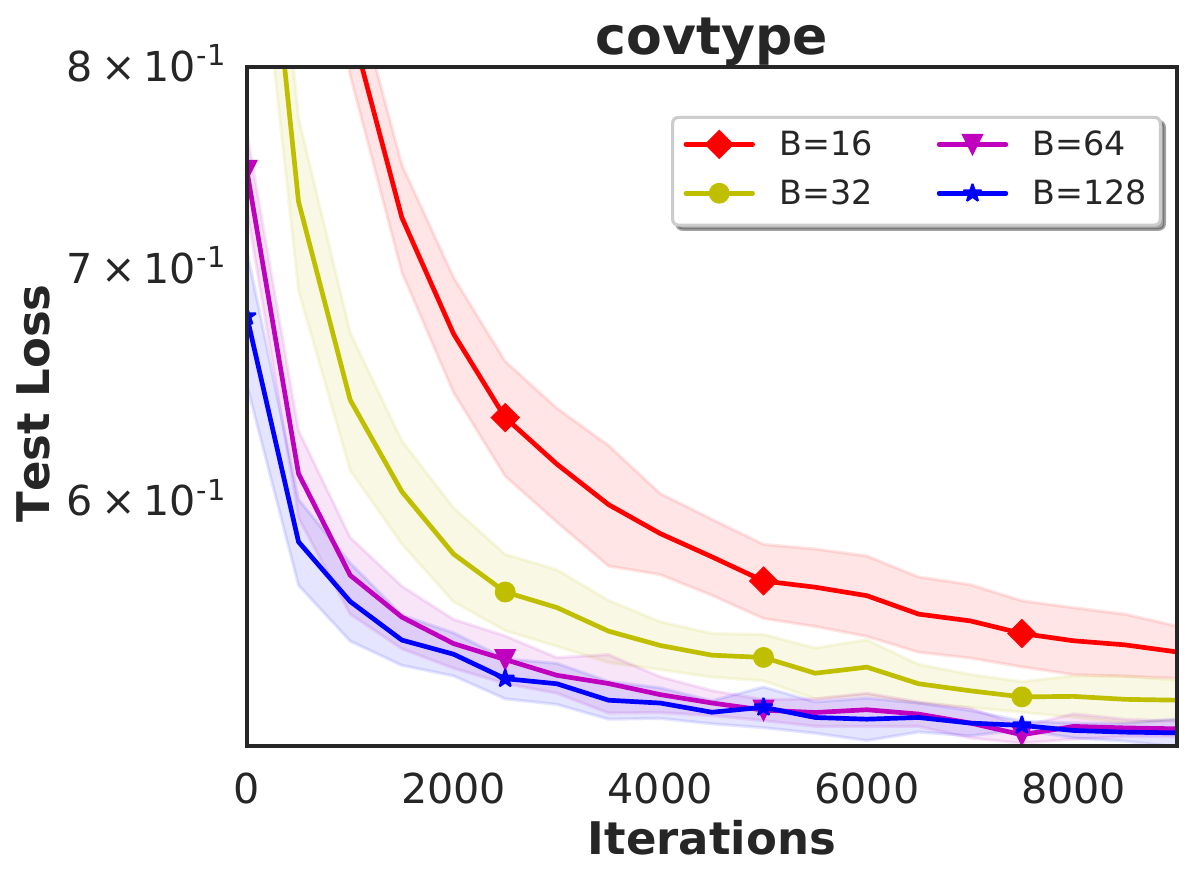}
	}\hfill
	\subfigure[Varying $\gamma$]{
		\includegraphics[width=0.31\linewidth]{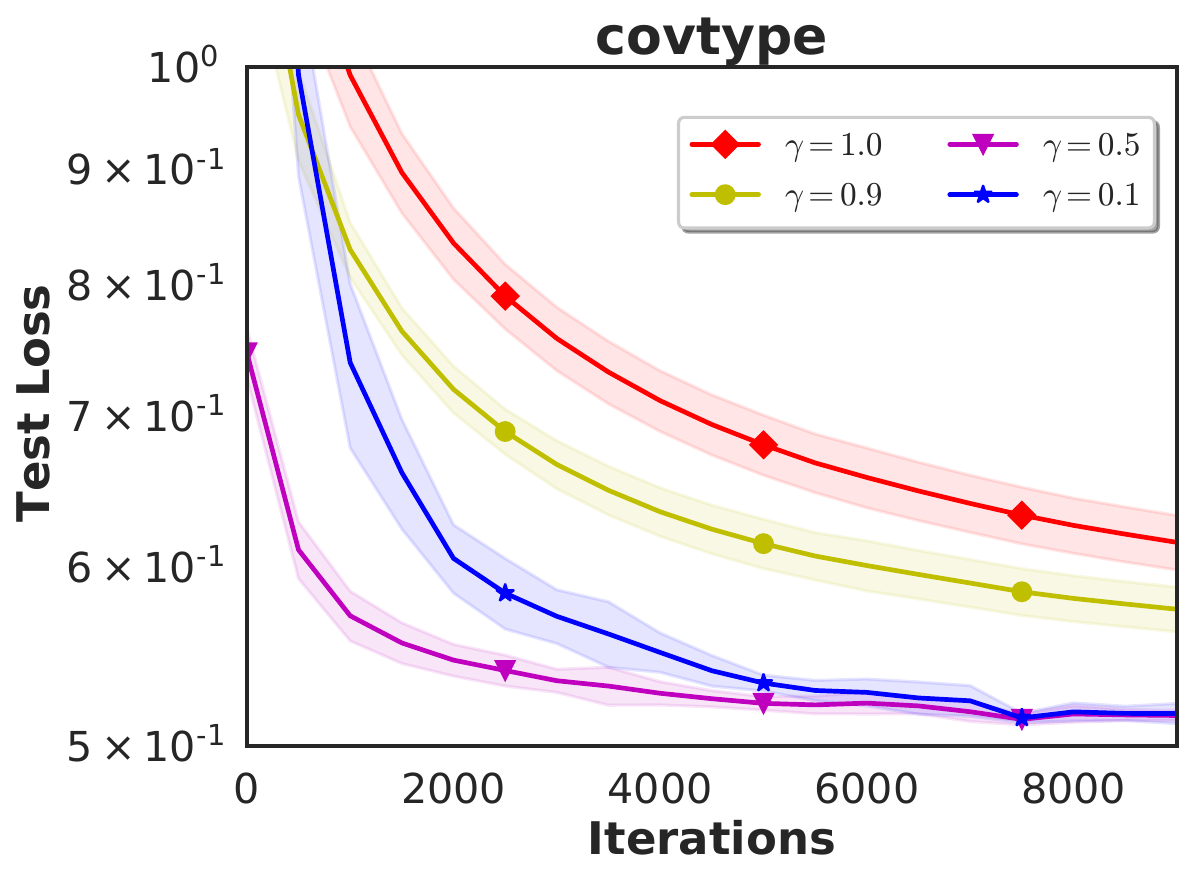}
	}\hfill
	\subfigure[v.s. Baselines]{
		\includegraphics[width=0.31\linewidth]{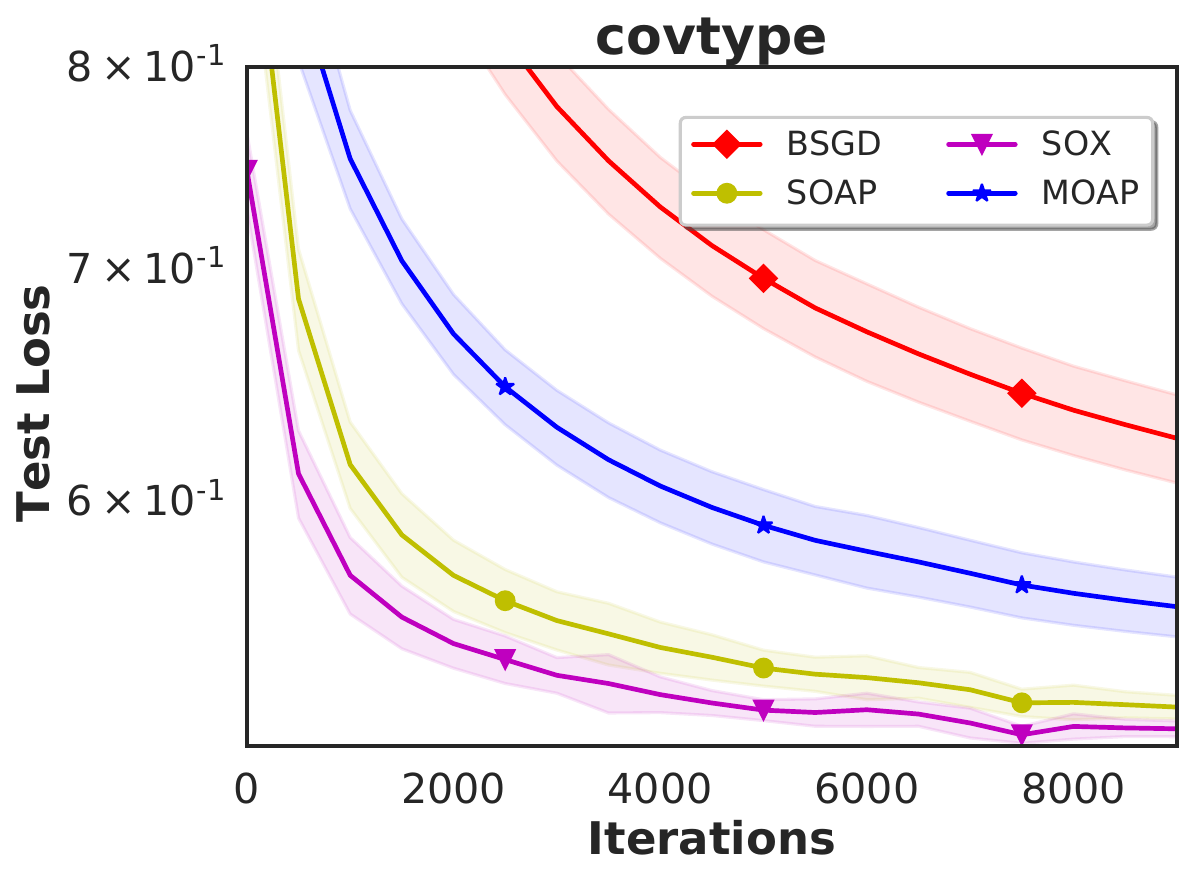}
	}\hfill
	\caption{Test loss curves of the $p$-norm push optimization task.}
	\label{fig:pnorm-linear}
\end{figure*}
Besides, we also empirically verify other aspects of the theory. In Figure~\ref{fig:pnorm-linear} (a)\&(d), we show that the iteration complexity of SOX decreases when $B$ increases; In Figure~\ref{fig:pnorm-linear} (b)\&(e), we show the effect of moving-average based estimation: $\gamma=1$ is not the best choice.

\section{Proof of Theorem~\ref{thm:sox}}
\begin{lemma}[Lemma 2 in \citealt{li2021page}]
	Consider a sequence $\w^{t+1} = \w^t - \eta \v^t$ and the $L_F$-smooth function $F$ and the step size $\eta L_F\leq 1/2$. 
	\begin{align}\label{eq:nonconvex_starter}
		F(\w^{t+1}) & \leq F(\w^t) + \frac{\eta}{2}\Delta^t - \frac{\eta}{2}\Norm{\nabla F(\w^t)}^2 - \frac{\eta}{4}\Norm{\v^t}^2,
	\end{align}	
	where $\Delta^t \coloneqq \Norm{\v^t - \nabla F(\w^t)}^2$.
	\label{lem:nonconvex_starter}
\end{lemma}

We build a recursion for the gradient variance $\Delta^t \coloneqq \norm{\v^t - \nabla F(\w^t)}^2$ by proving the following lemma. \begin{lemma}\label{lem:grad_recursion}
	If $\beta\leq \frac{2}{7}$, the gradient variance $\Delta^t \coloneqq \norm{\v^t - \nabla F(\w^t)}^2$ can be bounded as
	\begin{align}\label{eq:grad_recursion}
\E\left[\Delta^{t+1}\right] & \leq (1-\beta)\E\left[\Delta^t\right] + \frac{2L_F^2\eta^2\E\left[\norm{\v^t}^2\right]}{\beta} + \colorbox{gray!40}{$\frac{3L_f^2C_1^2}{n}\E\left[\sum_{\z_i\in\B_1^t}\norm{u_i^{t+1}-u_i^t}^2\right] $} \\\nonumber
& \quad\quad +  \frac{2\beta^2 C_f^2(\zeta^2+C_g^2)}{\min\{B_1,B_2\}}+ 5\beta L_f^2C_1^2 \E\left[\Xi_{t+1}\right],
	\end{align}
	where $\Xi_{t} = \frac{1}{n}\norm{\u^t -\bg(\w^t;\S)}^2$, $\u^t = [u_1^t,\dotsc,u_n^t]^\top$, $\bg(\w^t;\S) = [g(\w^t;\z_1,\S_1),\dotsc, g(\w^t;\z_n,\S_n)]^\top$.
\end{lemma}
\begin{proof} The proof technique follows similarly as that in~\citep{Ghadimi2020AST}. 
	We define that $\Delta^t \coloneqq \Norm{\v^t - \nabla F(\w^t)}^2$ and $G(\w^{t+1}) = \frac{1}{B_1}\sum_{\z_i\in\B_1^t}\nabla g(\w^{t+1};\z_i,\B_{i,2}^t) \nabla f(u^t_i)$. Based on the rule of update $\v^{t+1} = (1-\beta)\v^t + \beta G(\w^{t+1})$, we have
	\begin{align*}
		\Delta^{t+1}&= \Norm{\v^{t+1} - \nabla F(\w^{t+1})}^2 = \Norm{(1-\beta)\v^t + \beta \frac{1}{B_1}\sum_{\z_i\in\B_1^{t+1}} \nabla g(\w^{t+1}; \z_i, \B_{i,2}^{t+1}) \nabla f(u^t_i) - \nabla F(\w^{t+1})}^2\\
		& = \Norm{\tcircle{1} + \tcircle{2} + \tcircle{3}  + \tcircle{4}}^2,
	\end{align*}
	where $\tcircle{1}$, $\tcircle{2}$, $\tcircle{3}$, $\tcircle{4}$ are defined as
	\begin{align*}
		& \tcircle{1} = (1-\beta)(\v^t - \nabla F(\w^t)),\quad \tcircle{2} = (1-\beta)(\nabla F(\w^t) - \nabla F(\w^{t+1})),\\
		& \tcircle{3} = \beta \frac{1}{B_1}\sum_{\z_i\in\B_1^{t+1}}\left(\nabla g(\w^{t+1};\z_i, \B_{i,2}^{t+1})\nabla f(u^t_i) - \nabla g(\w^{t+1};\z_i, \B_{i,2}^{t+1}) \nabla f(g(\w^{t+1};\z_i,\S_i))\right),\\
		& \tcircle{4} = \beta \left(\frac{1}{B_1}\sum_{\z_i\in\B_1^{t+1}}  \nabla g(\w^{t+1};\z_i,\B_{i,2}^{t+1})\nabla f(g(\w^{t+1};\z_i,\S_i))-\nabla F(\w^{t+1})\right).
	\end{align*}
	Note that $\E_t\left[\inner{\tcircle{1}}{\tcircle{4}}\right] = \E_t\left[\inner{\tcircle{2}}{\tcircle{4}}\right] = 0$. Then, the Young's inequality for products implies that
	\begin{align*}
		& \E_t\left[\Norm{\tcircle{1} + \tcircle{2} + \tcircle{3}  + \tcircle{4}}^2\right] \\
		&= \Norm{\tcircle{1}}^2 + \Norm{\tcircle{2}}^2 + \E_t\left[\Norm{\tcircle{3}}^2\right] + \E_t\left[\Norm{\tcircle{4}}^2\right] + 2\inner{\tcircle{1}}{\tcircle{2}} + 2\E_t\left[\inner{\tcircle{1}}{\tcircle{3}}\right] + 2\E_t\left[\inner{\tcircle{2}}{\tcircle{3}}\right] +  2\E_t\left[\inner{\tcircle{3}}{\tcircle{4}}\right]\\
		& \leq (1+\beta)\Norm{\tcircle{1}}^2 + 2\left(1+\frac{1}{\beta}\right)\Norm{\tcircle{2}}^2 + \frac{2+3\beta}{\beta}\E_t\left[\Norm{\tcircle{3}}^2\right] + 2\E_t\left[\Norm{\tcircle{4}}^2\right].
	\end{align*}
	Besides, we have
	\begin{align*}
		& (1+\beta)\Norm{\tcircle{1}}^2 = (1+\beta)(1-\beta)^2\Norm{\v^t - \nabla F(\w^t)}^2 \leq (1-\beta)\Norm{\v^t - \nabla F(\w^t)}^2,\\
		& 2\left(1+\frac{1}{\beta}\right)\Norm{\tcircle{2}}^2 = 2\left(1+\frac{1}{\beta}\right)(1-\beta)^2\Norm{\nabla F(\w^t) - \nabla F(\w^{t+1})}^2 \leq \frac{2L_F^2 \eta^2}{\beta}\Norm{\v^t}^2,\\
		& \frac{2+3\beta}{\beta}\Norm{\tcircle{3}}^2 = \frac{2+3\beta}{\beta}\frac{\beta^2}{B_1}\sum_{\z_i\in\B_1^{t+1}}\Norm{\nabla g(\w^{t+1};\z_i,\B_{i,2}^{t+1})}^2\Norm{\nabla f(u^t_i) - \nabla f(g(\w^{t+1};\z_i,\S_i))}^2\\
		&\quad\quad\quad\quad\quad\quad \leq \frac{(2+3\beta)\beta L_f^2}{B_1}\sum_{\z_i\in\B_1^{t+1}} \Norm{\nabla g(\w^{t+1};\z_i,\B_{i,2}^{t+1})}^2\Norm{u^t_i - g(\w^{t+1};\z_i,\S_i)}^2.
	\end{align*}
	Consider that $\w^{t+1}$ and $u_i^t$ do not depend on either $\B_1^{t+1}$ or $\B_{i,2}^{t+1}$.
	\begin{align*}
		& (2+3\beta)\beta L_f^2 \E_t\left[\frac{1}{B_1}\sum_{\z_i\in\B_1^{t+1}} \Norm{\nabla g(\w^{t+1};\z_i,\B_{i,2}^{t+1})}^2\Norm{u^t_i - g(\w^{t+1};\z_i,\S_i)}^2\right]\\ & = (2+3\beta)\beta L_f^2 \E_t\left[\frac{1}{B_1}\sum_{\z_i\in\B_1^{t+1}} \E_t\left[\Norm{\nabla g(\w^{t+1};\z_i,\B_{i,2}^{t+1})}^2\mid \z_i\in\B_1^{t+1}\right]\Norm{u^t_i - g(\w^{t+1};\z_i,\S_i)}^2\right]\\
		& \leq (2+3\beta)\beta L_f^2 C_1^2 \E_t\left[\frac{1}{B_1}\sum_{\z_i\in\B_1^{t+1}}\Norm{u^t_i - g(\w^{t+1};\z_i,\S_i)}^2 \right] \\
		&  \leq \frac{(2+3\beta)\beta(1+\delta) L_f^2 C_1^2}{n}\sum_{\z_i\in\D} \E_t\left[\Norm{u^{t+1}_i - g(\w^{t+1};\z_i,\S_i)}^2\right] \\
		&\quad\quad +\frac{(2+3\beta)\beta(1+1/\delta) L_f^2 C_1^2}{n}\E_t\left[\sum_{\z_i\in\D} \Norm{u^{t+1}_i -u^t_i}^2 \right],
	\end{align*}
	where $C_1^2\coloneqq C_g^2 + \zeta^2/B$ and $\delta>0$ is a constant to be determined later. Note that we have $u_i^{t+1} = u_i^t$ for all $i\notin \B_1^t$. 
	\begin{align*}
		& (2+3\beta)\beta L_f^2 \E\left[\frac{1}{B_1}\sum_{\z_i\in\B_1^t} \Norm{\nabla g(\w^{t+1};\z_i,\B_{i,2}^t)}^2\Norm{u^t_i - g(\w^{t+1};\z_i,\S_i)}^2\right]\\ 
		& \leq (2+3\beta)\beta\left((1+\delta) L_f^2 C_1^2 \E\left[ \frac{1}{n}\sum_{\z_i\in\D}\Norm{u^{t+1}_i - g(\w^{t+1};\z_i,\S_i)}^2\right] \right.\\
		&\quad\quad\quad\quad\quad\quad\quad\quad \left. + (1+1/\delta) L_f^2 C_1^2\E\left[\frac{1}{n}\sum_{\z_i\in\B_1^t} \Norm{u^{t+1}_i -u^t_i}^2 \right]\right).
	\end{align*}
	If $\beta\leq \frac{2}{7}$ and $\delta = \frac{3\beta}{2}$, we have $(2+3\beta)\beta(1+\delta)\leq 5\beta$ and $(2+3\beta)\beta(1+1/\delta)\leq 3$. 
	\begin{align*}
		\E\left[\frac{2+3\beta}{\beta}\Norm{\tcircle{3}}^2\right] & \leq 5\beta L_f^2C_1^2 \E\left[\Xi_{t+1}\right] + \frac{3L_f^2 C_1^2}{n}\E\left[\sum_{\z_i\in\B_1^t}\Norm{u_i^{t+1} - u_i^t}^2\right].
	\end{align*}
	Next, we upper bound the term $\E_t\left[\Norm{\tcircle{4}}^2\right]$.
	\begin{align*}
		& \E_t\left[\Norm{\tcircle{4}}^2\right] \\
		& = \beta^2\E_t\left[\Norm{\frac{1}{B_1}\sum_{\z_i\in\B_1^t}\nabla g(\w^{t+1};\z_i,\B_{i,2}^t)\nabla f(g(\w^{t+1};\z_i,\S_i)) - \frac{1}{n}\sum_{\z_i\in\D} \nabla g(\w^{t+1};\z_i,\S_i)\nabla f(g(\w^{t+1};\z_i,\S_i))}^2\right]\\
		& = \beta^2 \E_t\left[\Norm{\frac{1}{B_1}\sum_{\z_i\in\B_1^t}\nabla g(\w^{t+1};\z_i,\B_{i,2}^t) \nabla f(g(\w^{t+1};\z_i,\S_i)) - \frac{1}{B_1}\sum_{\z_i\in\B_1^t}\nabla g(\w^{t+1};\z_i,\S_i) \nabla f(g(\w^{t+1};\z_i,\S_i)) }^2\right]\\
		& +\beta^2 \E_t\left[\Norm{\frac{1}{B_1}\sum_{\z_i\in\B_1^t}\nabla g(\w^{t+1};\z_i,\S_i) \nabla f(g(\w^{t+1};\z_i,\S_i)) - \frac{1}{n}\sum_{\z_i\in\D}\nabla g(\w^{t+1};\z_i,\S_i) \nabla f(g(\w^{t+1};\z_i,\S_i)) }^2\right]\\
		& \leq 
		\frac{\beta^2C_f^2(\zeta^2 + C_g^2)}{\min\{B_1, B_2\}}.
	\end{align*}
\end{proof}

\subsection{Proof of Lemma~\ref{lem:fval_recursion}}
Based on Algorithm~\ref{alg:sox}, the update rule of $u_i$ is 
\begin{align*}
	u^{t+1}_i = \begin{cases} (1-\gamma) u^t_i + \gamma g(\w_{t+1};\z_i,\B_{i,2}^{t+1})  & \z_i\in\B_1^{t+1}\\ u^t_i & \z_i\notin \B_1^{t+1}.\end{cases}
\end{align*}
We can re-write it into the equivalent expression below.
\begin{align*}
	u^{t+1}_i = \begin{cases}  u^t_i - \gamma\left(u^t_i - g(\w^{t+1};\z_i,\B_{i,2}^{t+1})\right)& \z_i\in\B_1^{t+1}\\ u^t_i & \z_i\notin \B_1^{t+1}.\end{cases}
\end{align*}
Let us define $\phi_t(\u) =\frac{1}{2}\Norm{\u - \bg(\w^t)}^2= \frac{1}{2}\sum_{\z_i\in\D}\Norm{u_i - g(\w^t;\z_i,\S_i)}^2$, which is a 1-strongly convex function. Then, the update rule \eqref{eq:bcd} can be viewed as one step of the stochastic block coordinate descent algorithm (Algorithm 2 in \citealt{dang2015stochastic}) for minimizing $\phi_{t+1}(\u)$, where the Bregman divergence is associated with the quadratic function. We follow the analysis of \citet{dang2015stochastic}.
\begin{align*}
	\phi_{t+1}(\u^{t+1}) & = \frac{1}{2}\Norm{\u^{t+1} - \bg(\w^{t+1})}^2 = \frac{1}{2}\Norm{\u^t - \bg(\w^{t+1})}^2 + \inner{\u^t - \bg(\w^{t+1})}{\u^{t+1} - \u^t} + \frac{1}{2}\Norm{\u^{t+1} - \u^t}^2\\
	&= \frac{1}{2}\Norm{\u^t - \bg(\w^{t+1})}^2 +  \sum_{\z_i\in\B_1^t}\inner{u^t_i - g(\w^{t+1};\z_i,\B_{i,2}^{t+1})}{u^{t+1}_i - u^t_i} + \frac{1}{2}\sum_{\z_i\in\B_1^{t+1}}\Norm{u^{t+1}_i -u^t_i}^2\\
	&\quad\quad\quad + \sum_{\z_i\in\B_1^t} \inner{g(\w_{t+1};\z_i,\B_{i,2}^{t+1}) - g(\w_{t+1};\z_i,\S_i)}{u^{t+1}_i - u^t_i}.
\end{align*}
Note that $u^t_i - g(\w^{t+1};\z_i,\B^{t+1}_{2,i}) = (u^t_i - u^{t+1}_i)/\gamma$ and $2\inner{b-a}{a-c}\leq \Norm{b-c}^2 - \Norm{a-b}^2 - \Norm{a-c}^2$.
\begin{align*}
	& \sum_{\z_i\in\B_1^{t+1}}\inner{u^t_i - g(\w^{t+1};\z_i,\B^{t+1}_{2,i})}{u^{t+1}_i - u^t_i} \\
	& = \sum_{\z_i\in\B_1^{t+1}}\inner{u^t_i - g(\w^{t+1};\z_i,\B^{t+1}_{2,i})}{g(\w^{t+1};\z_i,\S_i) - u^t_i} + 	\sum_{\z_i\in\B_1^{t+1}} \inner{u^t_i - g(\w^{t+1};\z_i,\B^{t+1}_{2,i})}{u^{t+1}_i - g(\w^{t+1};\z_i,\S_i)}\\
	& = \sum_{\z_i\in\B_1^{t+1}}\inner{u^t_i - g(\w^{t+1};\z_i,\B^{t+1}_i)}{g(\w^{t+1};\z_i,\S_i) - u^t_i} + 	\frac{1}{\gamma}\sum_{\z_i\in\B_1^{t+1}} \inner{u^t_i - u^{t+1}_i}{u^{t+1}_i - g(\w^{t+1};\z_i,\S_i)}\\
	& \leq \sum_{\z_i\in\B_1^{t+1}}\inner{u^t_i - g(\w^{t+1};\z_i,\B^{t+1}_{2,i})}{g(\w^{t+1};\z_i,\S_i) - u^t_i} \\
	& \quad\quad\quad + \frac{1}{2\gamma} \sum_{\z_i\in\B_1^{t+1}}\left(\Norm{u^t_i - g(\w_{t+1};\z_i,\S_i)}^2 - \Norm{u^{t+1}_i - u^t_i}^2 - \Norm{u^{t+1}_i - g(\w^{t+1};\z_i,\S_i)}^2\right).
\end{align*}
If $\gamma <\frac{1}{5}$, we have
\begin{align*}
	& -\frac{1}{2}\left(\frac{1}{\gamma} -1 -  \frac{\gamma+1}{4\gamma}\right)\sum_{\z_i\in\B_1^{t+1}}\Norm{u^{t+1}_i - u^t_i}^2 +   \sum_{\z_i\in\B_1^{t+1}} \inner{g(\w^{t+1};\z_i,\B^{t+1}_{2,i}) - g(\w^{t+1};\z_i,\S_i)}{u^{t+1}_i - u^t_i} \\
	& \leq - \frac{1}{4\gamma} \sum_{\z_i\in\B_1^{t+1}}\Norm{u_i^{t+1} - u_i^t}^2 + \gamma\sum_{\z_i\in\B_1^{t+1}} \Norm{g(\w^{t+1};\z_i,\B^{t+1}_{2,i}) - g(\w^{t+1};\z_i,\S_i)}^2 + \frac{1}{4\gamma} \sum_{\z_i\in\B_1^{t+1}}\Norm{u_i^{t+1} - u_i^t}^2\\
	& = \gamma \sum_{\z_i\in\B_1^{t+1}} \Norm{g(\w^{t+1};\z_i,\B^{t+1}_{2,i}) - g(\w^{t+1};\z_i,\S_i)}^2.
\end{align*}
Then, we have
\begin{align*}
&  \frac{1}{2}\Norm{\u^{t+1} - \bg(\w^{t+1})}^2 \leq \frac{1}{2}\Norm{\u^t -  \bg(\w^{t+1})}^2 + \frac{1}{2\gamma}\sum_{\z_i\in\B_1^{t+1}} \Norm{u^t_i - g(\w^{t+1};\z_i,\S_i)}^2  \\
	& \quad\quad\quad - \frac{1}{2\gamma} \sum_{\z_i\in\B_1^{t+1}} \Norm{u^{t+1}_i - g(\w^{t+1};\z_i,\S_i)}^2  + \gamma \sum_{\z_i\in\B^{t+1}_1}\Norm{g(\w^{t+1};\z_i, \B^{t+1}_{2,i}) - g(\w^{t+1};\z_i,\S_i)}^2 \\
	&  \quad\quad\quad - \frac{(\gamma+1)}{8\gamma}\sum_{\z_i\in\B_1^{t+1}}\Norm{u_i^{t+1} - u_i^t}^2 +  \sum_{\z_i\in\B_1^{t+1}}\inner{u^t_i - g(\w^{t+1};\z_i,\B^{t+1}_{2,i})}{g(\w^{t+1};\z_i,\S_i) - u^t_i}.
\end{align*}
Note that $\frac{1}{2\gamma}\sum_{i\notin\B_1^t} \Norm{u^t_i - g(\w^{t+1};\z_i,\S_i)}^2 = \frac{1}{2\gamma} \sum_{i\notin\B_1^t} \Norm{u^{t+1}_i - g(\w^{t+1};\z_i,\S_i)}^2 $ based on Algorithm~\ref{alg:sox}, which implies that
\begin{align*}
	\frac{1}{2\gamma}\sum_{\z_i\in\B_1^t} \left(\Norm{u^t_i - g(\w^{t+1};\z_i,\S_i)}^2 - \Norm{u^{t+1}_i - g(\w^{t+1};\z_i,\S_i)}^2\right) = \frac{1}{2\gamma_t}\left(\Norm{\u^t - \bg(\w^t)}^2 - \Norm{\u^{t+1} - \bg(\w^t)}^2\right).
\end{align*}
Besides, we also have $\E\left[\sum_{\z_i\in\B_1^{t+1}}\Norm{g(\w^{t+1};\z_i,\B^{t+1}_{2,i}) - g(\w^{t+1};\z_i,\S_i)}^2\right] \leq \frac{B_1 \sigma^2}{B_2}$ and 
\begin{align*}
& \E_t\left[\sum_{\z_i\in\B_1^{t+1}}\inner{u^t_i - g(\w^{t+1};\z_i,\B^{t+1}_{2,i})}{g(\w^{t+1};\z_i,\S_i) - u^t_i}\right]
\\
& = \frac{B_1}{n}\sum_{\z_i\in\D}\inner{u^t_i - g(\w^{t+1};\z_i,\S_i)}{g(\w^{t+1};\z_i,\S_i)-u^t_i} = -\frac{B_1}{n} \Norm{u^t_i - g(\w^{t+1};\z_i,\S_i)}^2.
\end{align*}
Then, we can obtain 
\begin{align*}
& \frac{\gamma + 1}{2}\E\left[\Norm{\u^{t+1} - \bg(\w^{t+1})}^2\right] \\
& \leq \frac{\gamma\left(1-\frac{B_1}{n}\right) + 1}{2}\E\left[\Norm{\u^t - \bg(\w^{t+1})}^2\right] + \frac{\gamma^2 B_1 \sigma^2}{B_2} - \frac{(\gamma+1)}{8}\sum_{\z_i\in\B_1^t}\Norm{u_i^{t+1} - u_i^t}^2.	
\end{align*}
Further divide $\frac{\gamma + 1}{2}$ and take full expectation on both sides
\begin{align*}
	\E\left[\Norm{\u^{t+1} - \bg(\w^{t+1})}^2\right] & \leq \frac{\gamma\left(1-\frac{B_1}{n}\right) + 1}{\gamma + 1} \E\left[\Norm{\u^t - \bg(\w^{t+1})}^2\right] + \frac{2}{1+\gamma}\frac{\gamma^2 \sigma^2B_1}{B_2} \\
	&\quad\quad - \frac{1}{4}\E\left[\sum_{\z_i\in\B_1^t}\Norm{u_i^{t+1} - u_i^t}^2\right].
\end{align*}
Note that $\frac{\gamma\left(1-\frac{B_1}{n}\right) + 1}{\gamma + 1} = 1- \frac{\gamma B_1}{\left(\gamma + 1\right)n}\leq 1-\frac{\gamma B_1}{2n}$ and $\frac{1}{1+\gamma}\leq 1$ for $\gamma \in (0,1]$. Besides, we have $\Norm{\u^t - \bg(\w^{t+1})}^2 \leq (1+\frac{\gamma B_1}{4n})\Norm{\u^t - \bg(\w^t)}^2 + (1+\frac{4n}{\gamma B_1})\Norm{\bg(\w^{t+1}) - \bg(\w^t)}^2$ due to Young's inequality and $\Norm{\bg(\w^t) - \bg(\w^{t-1})}^2 \leq n C_g^2 \Norm{\w^{t+1} - \w^t}^2 = n \eta^2 C_g^2 \Norm{\v^t}^2$.
\begin{align*}
	& \E\left[\Xi_{t+1}\right] = \E\left[\frac{1}{n}\Norm{\u^{t+1} - \bg(\w^{t+1})}^2\right] \\
	&\leq \left(1-\frac{\gamma B_1}{2n}\right)\E\left[\frac{1}{n}\Norm{\u^t - \bg(\w^{t+1})}^2\right] + \frac{2\gamma^2 \sigma^2 B_1}{n B_2} - \frac{1}{4n}\E\left[\sum_{\z_i\in\B_1^t}\Norm{u_i^{t+1} - u_i^t}^2\right]\\
	& \leq \left(1-\frac{\gamma B_1}{4n}\right)\E\left[\frac{1}{n}\Norm{\u^t - \bg(\w^t)}^2\right] + \frac{5nC_g^2\E\left[\Norm{\w^{t+1}- \w^t}^2\right]}{\gamma B_1} + \frac{2\gamma^2\sigma^2 B_1}{nB_2} - \frac{1}{4n}\E\left[\sum_{\z_i\in\B_1^t}\Norm{u_i^{t+1} - u_i^t}^2\right]\\
	& \leq \left(1-\frac{\gamma B_1}{4n}\right)\E\left[\Xi_t\right] + \frac{5n\eta^2C_g^2\E\left[\Norm{\v^t}^2\right]}{\gamma B_1} + \frac{2\gamma^2\sigma^2B_1}{nB_2} -\frac{1}{4n}\E\left[\sum_{\z_i\in\B_1^t}\Norm{u_i^{t+1} - u_i^t}^2\right].
\end{align*}

\subsection{Proof of Theorem~\ref{thm:sox}}

Based on Lemma~\ref{lem:nonconvex_starter}, Lemma~\ref{lem:grad_recursion}, and Lemma~\ref{lem:fval_recursion}, we have
\begin{align}\label{eq:sox_first_eq}
	& \E\left[F(\w^{t+1}) - F^*\right] \leq \E\left[F(\w^t)- F^*\right] + \frac{\eta}{2}\E\left[\Delta_t\right] - \frac{\eta}{2}\E\left[\Norm{\nabla F(\w^t)}^2\right] - \frac{\eta}{4}\E\left[\Norm{\v^t}^2\right],\\\label{eq:sox_second_eq}
	&\E\left[\Delta_{t+1}\right] \leq (1-\beta) \E\left[\Delta_t\right] + \frac{2L_F^2\eta^2}{\beta}\E\left[\Norm{\v^t}^2\right] + 5\beta L_f^2C_1^2\E\left[\Xi_{t+1}\right] + \frac{2\beta^2C_f^2(\zeta^2 + C_g^2)}{\min\{B_1,B_2\}} \\\nonumber
	&\quad\quad\quad\quad\quad\quad + \frac{3L_f^2 C_1^2}{n}\E\left[\sum_{\z_i\in\B_1^t}\Norm{u_i^{t+1} - u_i^t}^2\right],\\\label{eq:sox_third_eq}
	&	\E\left[\Xi_{t+1}\right] \leq \left(1-\frac{\gamma B_1}{4n}\right)\E\left[\Xi_t\right] + \frac{5n\eta^2 C_g^2\E\left[\Norm{\v^t}^2\right]}{\gamma B_1} + \frac{2\gamma^2\sigma^2B_1}{nB_2} - \frac{1}{4n}\E\left[\sum_{\z_i\in\B_1^t}\Norm{u_i^{t+1} - u_i^t}^2\right].
\end{align} 
Summing \eqref{eq:sox_first_eq}, $\frac{\eta}{\beta}\times$\eqref{eq:sox_second_eq}, and $\frac{20L_f^2C_1^2n\eta}{\gamma B_1}\times$\eqref{eq:sox_third_eq} leads to
\begin{align*}
	& \E\left[F(\w^{t+1})- F^* + \frac{\eta}{\beta}\Delta_{t+1} + \frac{20L_f^2C_1^2n\eta}{\gamma B_1}\left(1-\frac{\gamma B_1}{4n}\right)\Xi_{t+1}\right]\\
	& \leq  \E\left[F(\w^t)- F^* + \frac{\eta}{\beta}\left(1-\frac{\beta }{2}\right)\Delta_t + \frac{20L_f^2C_1^2n\eta}{\gamma B_1}\left(1-\frac{\gamma B_1}{4n}\right)\Xi_t\right] \\
	&  \quad\quad\quad - L_f^2C_1^2\eta\left(\frac{5n}{\gamma B_1}-\frac{3}{\beta}\right)\E\left[\frac{1}{n}\sum_{\z_i\in\B_1^t}\Norm{u_i^{t+1} - u_i^t}^2\right] -\frac{\eta}{2}\E\left[\Norm{\nabla F(\w^t)}^2\right]\\
	& \quad\quad\quad  - \eta\left(\frac{1}{4} - \frac{2L_F^2\eta^2}{\beta^2} - \frac{100L_f^2 n^2C_1^2\eta^2 C_g^2}{\gamma^2 B_1^2}\right) \E\left[\Norm{\v^t}^2\right] + \frac{2\beta\eta C_f^2(\zeta^2 + C_g^2)}{\min\{B_2,B_2\}} + \frac{40\eta\gamma L_f^2 C_1^2\sigma^2}{B_2}.
\end{align*}
If $\gamma \leq \frac{5n}{3B_1}\beta$, we have $\frac{5n}{\gamma B_1}-\frac{3}{\beta}\geq 0$. Set $\beta=\min\{\frac{\min\{B_1,B_2\}\epsilon^2}{12C_f^2(\zeta^2+C_g^2)}, \frac{2}{7}\}$, $\gamma = \min\left\{\frac{B_2\epsilon^2}{240L_f^2C_1^2\sigma^2}, \frac{1}{5}, \frac{5n}{3B_1}\beta\right\}$, and $\eta = \min\left\{\frac{\beta}{4 L_F}, \frac{\gamma B_1}{30L_f n C_1 C_g}\right\}$. Define the Lyapunov function as $\Phi_t\coloneqq F(\w^t) - F^* + \frac{\eta}{\beta}\Delta_t + \frac{20L_f^2 C_1^2}{B_1}\frac{\eta}{\gamma}\left(1-\frac{\gamma B_1}{4n}\right)\Xi_t$. If we initialize $\u^1$ and $\v^1$ as $\u^1_i = g_i(\w^1;\B_{i,2}^1)$ for $\z_i\in\B_1^t$ and $\v^1 = 0$, we have $\E\left[\Delta_1\right]\leq C_f^2C_g^2$ and $\E\left[\Xi_1\right]\leq \frac{\sigma^2}{B_2}$. Then,
\begin{align}\label{eq:sox_grad_norm_bound}
	\frac{1}{T}\sum_{t=1}^T \E\left[\Norm{\nabla F(\w^t)}^2\right] \leq \frac{2\Lambda_\Phi^1}{\eta T} + \frac{4\beta C_f^2(\zeta^2 + C_g^2)}{\min\{B_1,B_2\}} + \frac{80 \gamma L_f^2 C_1^2\sigma^2}{B_2},
\end{align}
where we define $ \Lambda_\Phi^1 \coloneqq  \Delta_F + \frac{1}{4L_F}C_f^2 C_g^2 + \frac{2 L_f C_1\sigma^2}{3 C_g B_2}\geq \E\left[\Phi_1\right]$.  After
\begin{align*}
 \resizebox{\hsize}{!}{$
	T = \frac{6\Lambda_\Phi^1}{\epsilon^2}\max\left\{\frac{48C_f^2(\zeta^2+C_g^2)L_F}{\min\{B_1,B_2\}\epsilon^2}, 14L_F, \frac{7200n L_f^3C_1^3C_g\sigma^2}{B_1B_2\epsilon^2}, \frac{150 L_f n C_1 C_g}{B_1}, 63L_fC_1 C_g, \frac{216L_f C_1 C_g C_f^2(\zeta^2+C_g^2)}{\min\{B_1,B_2\}\epsilon^2}\right\}$}
\end{align*}
iterations, we have $\frac{1}{T}\sum_{t=1}^T \E\left[\Norm{\nabla F(\w^t)}^2\right] \leq \epsilon^2$.

\section{Proof of Theorem~\ref{thm:sox-boost-simp}}

\begin{lemma}\label{lem:per_stage_func}
	Under assumptions~\ref{asm:lip},~\ref{asm:var} and the $\mu$-PL of $F$, the $k$-th epoch of applying Algorithm~\ref{alg:sox-boost} leads to
	\begin{align}\label{eq:single_epoch_dnasa}
		\E\left[\Gamma_{k+1}\right] & \leq \frac{\E\left[\Gamma_k\right]}{\mu\eta_k T_k} + \frac{\E\left[\Delta_k\right]}{\mu \beta_k T_k} + \frac{20 n L_f^2 C_1^2\E\left[\Xi_k\right]}{B_1 \mu \gamma_k T_k} + \frac{\beta_k C_2^2}{\mu \min\left\{B_1, B_2\right\}} + \frac{\gamma_k C_3^2}{\mu B_2}, 
	\end{align}
	where $\Gamma_k\coloneqq F(\w^k) - F(\w^*)$, $C_1^2 \coloneqq C_g^2 +\zeta^2/B_2$, $C_2^2\coloneqq 2C_f^2(\zeta^2 + C_g^2)$, $C_3^2\coloneqq 40 L_f^2 C_1^2\sigma^2$.
\end{lemma}
\begin{proof}
	Since $F$ is $\mu$-PL, we have $F(\w) - F(\w^*) \leq \frac{1}{2\mu} \Norm{\nabla F(\w)}^2$. We define $\Gamma_k\coloneqq F(\w^k) - F(\w^*)$, $\Delta_k\coloneqq \Norm{\v^k - \nabla F(\w^k)}^2$ and $\Xi_k\coloneqq \frac{1}{n}\Norm{\u^k - \bg(\w^k)}^2$. Applying PL condition and Theorem~\ref{thm:sox} to one epoch of SOX-boost leads to $\E\left[\Gamma_{k+1}\right]  \leq \frac{1}{2\mu}\E\left[\Norm{\nabla F(\w^{k+1})}^2\right]$ and 
	\begin{align*}
		\resizebox{\hsize}{!}{$ \frac{1}{2\mu}\E\left[\Norm{\nabla F(\w^{k+1})}^2\right]\leq \frac{1}{2\mu T_k}\sum_{t=1}^{T_k}\E\left[\Norm{\nabla F(\w^t)}^2\right] \leq \frac{\E\left[\Gamma_k\right]}{\mu\eta_k T_k} + \frac{\E\left[\Delta_k\right]}{\mu \beta_k T_k} + \frac{20 n L_f^2 C_1^2\E\left[\Xi_k\right]}{B_1 \mu \gamma_k T_k} + \frac{\beta_k C_2^2}{\mu \min\left\{B_1, B_2\right\}} + \frac{\gamma_k C_3^2}{\mu B_2},$}
	\end{align*}
	where we define $C_2^2\coloneqq 2C_f^2(\zeta^2 + C_g^2)$, $C_3^2\coloneqq 40 L_f^2 C_1^2\sigma^2$. 
\end{proof}

\begin{lemma}\label{lem:epoch_var_decay}
	Under assumptions~\ref{asm:lip},~\ref{asm:var}, the $k$-th epoch of Algorithm~\ref{alg:sox-boost} leads to
	\begin{align*}
		& \E\left[\Delta_{k+1} + C_5\Xi_{k+1}\right]\leq  \frac{\E\left[6\Gamma_k + 10C_4\Delta_k + 7C_4C_5 \Xi_k\right]}{\eta_k T_k}+ \frac{10\eta_kC_2^2}{C_4\min\{B_1,B_2\}} + \frac{80nC_3^2\eta_k }{3B_1 B_2 C_4}, 
	\end{align*}
	where $\beta_k\leq \min\left\{\frac{3B_1}{50n}, \frac{2}{7}\right\}$, $\gamma_k = \frac{10n}{3B_1}\beta_k$, $\eta_k = \beta_kC_4$, $C_4\coloneqq \min\left\{1/4 L_F, 1/9 L_fC_1 C_g\right\}$, $C_5\coloneqq 12 L_f^2C_1^2$.
\end{lemma}

\begin{proof}
	Applying Lemma~\ref{lem:grad_recursion} to single iteration in any epoch of SOX-boost with $\beta_k\leq \frac{2}{7}$ leads to
	\begin{align*}
		\E\left[\Delta_{t+1}\right] &\leq (1-\beta_k) \E\left[\Delta_t\right] + \frac{4L_F^2\eta_k^2}{\beta} \E\left[\Delta_t\right] + \frac{4 L_F^2\eta_k^2}{\beta_k} \E\left[\Norm{\nabla F(\w^t)}^2\right]\\
		&\quad\quad + 5\beta_k L_f^2C_1^2\E\left[\Xi_{t+1}\right] + \frac{2\beta_k^2C_f^2(\zeta^2+ C_g^2)}{\min\{B_1,B_2\}} + \frac{3L_f^2 C_1^2}{n}\E\left[\sum_{\z_i\in\B_1^t}\Norm{u_i^{t+1} - u_i^t}^2\right].
	\end{align*}
	Applying Lemma~\ref{lem:fval_recursion} to one iteration in any epoch of SOX-boost with $\gamma_k\leq \frac{1}{5}$ leads to
	\begin{align*}
	\resizebox{\hsize}{!}{$
		\E\left[\Xi_{t+1}\right] \leq \left(1-\frac{\gamma_k B_1}{4n}\right)\E\left[\Xi_t\right] + \frac{10n\eta_k^2 C_g^2\E\left[\Delta_t\right]}{\gamma_k B_1} + \frac{10n\eta_k^2 C_g^2\E\left[\Norm{\nabla F(\w^t)}^2\right]}{\gamma_k B_1} + \frac{2\gamma_k^2\sigma^2 D}{n B_2} - \frac{1}{4n}\E\left[\sum_{\z_i\in\B_1^t}\Norm{u_i^{t+1} - u_i^t}^2\right].$}
	\end{align*}
	The following holds by summing up $\E\left[\Delta_{t+1}\right]$ and $\frac{40 L_f^2 C_1^2 n \beta_k}{\gamma_k B_1}\times \E\left[\Xi_{t+1}\right]$ and noticing $\left(1- \frac{\gamma_k B_1}{4n}\right)\leq \left(1- \frac{\gamma_k B_1}{8n}\right)^2$.
	\begin{align*}
	& \E\left[\Delta_{t+1} + \frac{40 L_f^2C_1^2 n\beta_k}{\gamma_k B_1}\Xi_{t+1}\right]  \\
	& \leq \E\left[\left(1-\beta_k + \frac{4 L_F^2\eta_k^2}{\beta_k} + \frac{400n^2\eta_k^2\beta_k C_g^2 L_f^2C_1^2}{\gamma_k^2B_1^2}\right)\Delta_t + \frac{40 L_f^2C_1^2 n\beta_k}{\gamma_k B_1}\left(1-\frac{\gamma_k B_1}{8n}\right)\Xi_t\right]\\
		& \quad\quad + \left(\frac{4 L_F^2\eta_k^2}{\beta_k} + \frac{400n^2\eta_k^2\beta_k C_g^2 L_f^2C_1^2}{\gamma_k^2 B_1^2}\right) \E\left[\Norm{\nabla F(\w^t)}^2\right] + \frac{2\beta_k^2 C_f^2(\zeta^2+C_g^2)}{\min\{B_1,B_2\}} \\
		&  \quad\quad + \frac{80 \gamma_k\beta L_f^2 C_1^2\sigma^2}{B_2} -L_f^2C_1^2\left(\frac{10n\beta_k}{\gamma_k B_1} - 3\right) \E\left[\frac{1}{n}\sum_{\z_i\in\B_1^t} \Norm{u_i^{t+1} - u_i^t}^2\right]
	\end{align*}
	If we set $\gamma_k\leq \frac{10n}{3B_1}\beta_k$, $\eta_k\leq \frac{\beta_k}{4 L_F}$ and $\eta_k\leq \frac{\gamma_k B_1}{30 L_f n C_1 C_g}$. 
	\begin{align*}
		1-\beta_k + \frac{4 L_F^2\eta_k^2}{\beta_k} + \frac{400n^2\eta_k^2\beta_k C_g^2 L_f^2C_1^2}{\gamma_k^2B_1^2} \leq 1 - \beta_k + \frac{\beta_k}{4} + \frac{4\beta_k}{9} \leq 1-\frac{\beta_k}{4}.  
	\end{align*}
	If $\beta_k\leq \min\left\{\frac{3B_1}{50n}, \frac{2}{7}\right\}$, we can set $\gamma_k = \frac{10n}{3B_1}\beta_k$ and $\eta_k = \beta_k C_4$, where $C_4\coloneqq \min\left\{1/4 L_F, 1/9 L_fC_1 C_g\right\}$. Then, $1-\frac{\gamma_k B_1}{8n}=1-\frac{5\beta_k}{12}\leq 1-\frac{\beta_k}{4}$. Besides, we define $C_5\coloneqq \frac{40 L_f^2C_1^2 n\beta_k}{\gamma_k B_1} = 12 L_f^2C_1^2$.
	\begin{align*}
		& \E\left[\Delta_{t+1} + C_5\Xi_{t+1}\right]\leq \left(1-\frac{\beta_k}{4}\right)\E\left[\Delta_t + C_5\Xi_t\right] + \frac{3\beta_k}{4} \E\left[\Norm{\nabla F(\w^t)}^2\right] +  \frac{2\beta_k^2 C_f^2(\zeta^2+C_g^2)}{\min\{B_1,B_2\}}  + \frac{2 \gamma_k\beta_k C_3^2}{B_2}.
	\end{align*}
	Telescoping the equation above from 1 to $T_k$ iterations in epoch $k$ leads to
	\begin{align*}
		\E\left[\Delta_{k+1} + C_5\Xi_{k+1}\right] & = \E\left[\frac{1}{T_k}\sum_{t=0}^{T_k}\Delta_{t+1} + C_5\frac{1}{T_k}\sum_{t=0}^{T_k}\Xi_{t+1}\right]\\
		& \leq \frac{4C_4\E\left[\Delta_k + C_5\Xi_k\right]}{\eta_k T_k} + 3\frac{1}{T_k}\sum_{t=1}^{T_k} \E\left[\Norm{\nabla F(\w^t)}^2\right] + \frac{4\beta_k C_2^2}{\min\{B_1,B_2\}} + \frac{20nC_3^2\eta_k}{3B_1 B_2C_4}.
	\end{align*}
	Applying \eqref{eq:sox_grad_norm_bound}, we can further derive that
	\begin{align*}
		& \E\left[\Delta_{k+1} + C_5\Xi_{k+1}\right]\leq  \frac{\E\left[6\Gamma_k + 10C_4\Delta_k + 7C_4C_5 \Xi_k\right]}{\eta_k T_k}+ \frac{10\eta_kC_2^2}{C_4\min\{B_1,B_2\}} + \frac{80nC_3^2\eta_k }{3B_1 B_2 C_4}.
	\end{align*} 
\end{proof}

\begin{lemma}\label{lem:epoch_sox_induction}
	If we set $ \eta_k = \min\left\{\frac{\mu C_4\min\{B_1,B_2\}\varepsilon_k}{6C_2^2}, \frac{\min\{B_1,B_2\}\varepsilon_k}{60C_2^2}, \frac{\mu C_4 B_1 B_2 \varepsilon_k}{20nC_3^2},\frac{B_1 B_2 C_4\varepsilon_k}{200nC_3^2}, \frac{3B_1 C_4}{50n}, \frac{2C_4}{7}\right\}$ and $T_k = \frac{\max\left\{\frac{12}{\mu},96C_4\right\}}{\eta_k}$, $\gamma_k = \frac{10n\eta_k}{3B_1 C_4}$, $\beta_k = \frac{\eta_k}{C_4}$, we can conclude that $\E\left[\Gamma_k\right]\leq \varepsilon_k$ and $\E\left[\Delta_k + C_5\Xi_k\right]\leq \frac{\varepsilon_k}{C_4}$, where $\varepsilon_1 = \max\left\{\Delta_F, C_4(C_f^2 C_g^2 + C_5\sigma^2/B_2)\right\}$ and $\varepsilon_k = \varepsilon_1/2^{k-1}$ for $k\geq 1$.
\end{lemma}
\begin{proof}
	We prove this lemma by induction. First, we define $\Delta_F \coloneqq F(\w^1) - F(\w^*)$, $\Gamma_k = F(\w^k) - F(\w^*)$, and $\varepsilon_1 = \max\left\{\Delta_F, C_4(C_f^2 C_g^2 + C_5\sigma^2/B_2)\right\}$, where $C_4\coloneqq \min\left\{1/4 L_F, 1/9L_f C_1 C_g\right\}$, $C_5\coloneqq 12 L_f^2C_1^2$. If we initialize $\u^1$ and $\v^1$ as $u^1_i = g_i(\w^1;\B_{i,2}^1)$ for $i\in\D_t$ and $\v^1 = 0$, we have $\E\left[\Gamma_1\right]\leq \varepsilon_1$ and $\E\left[\Delta_1+ C_5\Xi_1\right]\leq \frac{\varepsilon_1}{C_4}$. 
	
	Next, we consider the $k\geq 2$ case. Assume $\E\left[\Gamma_k\right]\leq \varepsilon_k$ and $\E\left[\Delta_k + C_5\Xi_k\right]\leq \frac{\varepsilon_k}{C_4}$. We define $\varepsilon_k = \varepsilon_1/2^{k-1}$ for $k\geq 2$. We choose $\beta_k\leq \min\left\{\frac{3D}{50n}, \frac{2}{7}\right\}$, $\gamma_k = \frac{10n}{3D}\beta_k$, $\eta_k = \beta_kC_4$. Based on Lemma~\ref{lem:per_stage_func}, we have
	\begin{align*}
		\E\left[\Gamma_{k+1}\right] &\leq \frac{\E\left[\Gamma_k + C_4(\Delta_k + C_5 \Xi_k)\right]}{\mu \eta_k T_k} + \frac{\eta_k C_2^2}{\mu C_4 \min\{B_1,B_2\}} + \frac{10 n C_3^2\eta_k}{3\mu B_1 B_2 C_4}.
	\end{align*}
	Besides, Lemma~\ref{lem:epoch_var_decay} implies that
	\begin{align*}
		& \E\left[\Delta_{k+1} + C_5\Xi_{k+1}\right]\leq  \frac{\E\left[6\Gamma_k + 10C_4(\Delta_k + C_5\Xi_k)\right]}{\eta_k T_k}+ \frac{10\eta_k C_2^2}{C_4\min\{B_1,B_2\}} + \frac{80n\eta_kC_3^2}{3B_1 B_2 C_4}.
	\end{align*} 
	The following choices of $\eta_k$ and $T_k$ makes $\E\left[\Gamma_{k+1}\right]\leq \varepsilon_{k+1} = \frac{\varepsilon_k}{2}$ and $\E\left[\Delta_{k+1} + C_5\Xi_{k+1}\right] \leq \frac{\varepsilon_{k+1}}{C_4} = \frac{\varepsilon_k}{2C_4}$. We define $C_6\coloneqq \min\{\lambda C_4,\frac{1}{10}\}$.
	\begin{align*}
		& \eta_k =\min\left\{\frac{\mu C_4\min\{B_1,B_2\}\varepsilon_k}{6C_2^2}, \frac{\min\{B_1,B_2\}\varepsilon_k}{60C_2^2}, \frac{\mu C_4 B_1 B_2 \varepsilon_k}{20nC_3^2},\frac{B_1 B_2 C_4\varepsilon_k}{200n C_3^2}, \frac{3B_1 C_4}{50n}, \frac{2C_4}{7}\right\},\\
		& T= \frac{\max\left\{\frac{12}{\mu}, 96 C_4\right\}}{\eta_k}.\\
	\end{align*}
\end{proof}	
\begin{theorem}[Detailed Version of Theorem~\ref{thm:sox-boost-simp}]\label{thm:sox-boost}
	Under assumptions~\ref{asm:lip}, \ref{asm:var} and the $\mu$-PL of $F$, SOX-boost (Algorithm~\ref{alg:sox-boost}) can find an $\w$ satisfying that $\E\left[F(\w) - F(\w^*)\right]\leq 2\epsilon$ after $$T = C_{1/\mu} \max\left\{\frac{6C_2^2}{\mu C_4\min\{B_1,B_2\}\epsilon}, \frac{60C_2^2}{\min\{B_1,B_2\}\epsilon}, \frac{20nC_3^2}{\mu C_4 B_1 B_2\epsilon}, \frac{200nC_3^2}{B_1 B_2C_4\epsilon}, \frac{50n\log(\varepsilon_1/\epsilon)}{3B_1C_4}, \frac{7\log(\varepsilon_1/\epsilon)}{2C_4}\right\}$$ iterations, where $C_{1/\mu} = \max\left\{\frac{12}{\mu}, 96C_4\right\}$.
\end{theorem}
\begin{proof}
	According to Lemma~\ref{lem:epoch_sox_induction}, the total number of iterations to achieve target accuracy $\E\left[\Gamma_k\right]\leq \epsilon$ can be represented as:
	\begin{align*}
		& T = \sum_{k=1}^{\log(\varepsilon_1/\epsilon)} T_k \\
		& = C_{1/\mu} \max\left\{\frac{6C_2^2}{\mu C_4\min\{B_1,B_2\}\epsilon}, \frac{60C_2^2}{\min\{B_1,B_2\}\epsilon}, \frac{20nC_3^2}{\mu C_4 B_1 B_2\epsilon}, \frac{200nC_3^2}{B_1 B_2C_4\epsilon}, \frac{50n\log(\varepsilon_1/\epsilon)}{3B_1C_4}, \frac{7\log(\varepsilon_1/\epsilon)}{2C_4}\right\},
	\end{align*}
	where $C_{1/\mu} = \max\left\{\frac{12}{\mu}, 96C_4\right\}$.  
\end{proof}

\begin{proof}[Proof of Corollary~\ref{cor:cvx}]
	Suppose that $\w^*$ is a minimum of $F$ and $\hat{\w}^*$ is the minimum of the strongly convexified $\hat{F}$. If $\E\left[\hat{F}(\w) - \hat{F}(\hat{\w}^*)\right]\leq \epsilon$, we have 
	\begin{align*}
		\E\left[F(\w)\right] \leq \E\left[\hat{F}(\w)\right] \leq \hat{F}(\hat{\w}^*)  + \epsilon \leq \hat{F}(\w^*)  + \epsilon = F(\w^*) + \frac{\lambda}{2}\Norm{\w^*}^2 + \epsilon.
	\end{align*}
	Thus, if the minimum $\w^*$ of $F$ is in a bounded domain $\Norm{\w^*}\leq C_*$ and we choose $\lambda = \frac{2\epsilon}{C_*^2}$, we also have $\E\left[F(\w) - F(\w^*)\right]\leq 2\epsilon$.
\end{proof}

\section{Proof of Theorem~\ref{thm:cvx_sox-simp}}

First, we state some technical lemmas.

\subsection{Technical Lemmas}

For any iteration $t$, we can define $\y^t = (\w^t;\pi_1^{t+1};\pi_2^{t+1})$ and $\pi_1^{t+1} = [\pi_{1,1}^{t+1}, \dotsc, \pi_{n,1}^{t+1}]^\top$, $\pi_2^{t+1} = [\pi_{1,2}^{t+1}, \dotsc, \pi_{n,2}^{t+1}]^\top$, where $\pi_{i,1}^{t+1}=\nabla f(u^{t+1}_i)$ and $\pi_{i,2}^{t+1}=\nabla g_i(\w_t)$.  
We can define the gap $Q(\y^t,\y)$ as
\begin{align*}
	Q(\y^t,\y) = \L(\w^t,\pi_1,\pi_2) - \L(\w,\pi_1^{t+1},\pi_2^{t+1}) = \frac{1}{n}\sum_{\z_i\in\D} \left(\underbrace{\L_{i,1}(\w^t,\pi_{i,1},\pi_{i,2})-\L_{i,1}(\w,\pi_{i,1}^{t+1},\pi_{i,2}^{t+1})}_{\coloneqq Q_i(\y^t,\y)}\right).
\end{align*}
The following lemmas are needed.
\begin{lemma}
	For any $\pi_{i,1}$ such that $\pi_{i,1}=\nabla f_i(u_{i})$ for some bounded $u_i$, then there exists $C_{f_*}$ such that $|f_i^*(\pi_{i,1})|\leq C_{f_*}$. 
\end{lemma}
\begin{proof}
	Due to the definition of convex conjugate, we have $f_i(u_i) + f_i^*(\pi_{i,1})=\pi_{i,1}u_i$. Due to that $\pi_{i,1}$ and $u_i$ are bounded and $f_i(u_i)$ is bounded due to its Lipchitz continity. As a result, $f_i^*(\pi_{i,1})$ is bounded by some constant $C_{f_*}$. 
\end{proof}
\begin{lemma}\label{lem:juditsky_variant}
	For any $\pi_1\in\Pi$ and the sequences $\{\tilde{\pi}_1^t\}$ and $\hat{\pi}_1^t$ defined as $\tilde{\pi}_1^t = \arg\min_{\pi_1}  \inner{\hat{g}(\w^t;\tilde{\B}_2^t) - g(\w^t)}{\pi_1} + \tau' \bD_{f^*}(\hat{\pi}_1^t, \pi_1)$, 
	$\hat{\pi}_1^{t+1}  = \arg\min_{\pi_1}  \inner{\hat{g}(\w^t;\B_2^t) - g(\w^t)}{\pi_1} + \tau' \bD_{f^*}(\hat{\pi}_1^t, \pi_1)$, we have
	\begin{align*}
		\inner{\hat{g}(\w^t;\B_2^t) - g(\w^t)}{\pi_1 - \tilde{\pi}_1^t} & \geq \tau' (\bD_{f^*}(\hat{\pi}_1^{t+1},\pi_1) - \bD_{f^*}(\hat{\pi}_1^t,\pi_1)) - \frac{L_f}{2\tau'}\Norm{\hat{g}(\w^t;\B_2^t)  - \hat{g}(\w^t;\tilde{\B}_2^t)}^2.
	\end{align*}
\end{lemma}
\begin{proof}
	The proof of this lemma is almost the same to that of Lemma 4 in \citet{juditsky2011solving}. Due to the three-point inequality (Lemma 1 in \citealt{zhang2020optimal}), we have:
	\begin{align}\label{eq:3p_1}
		\bD_{f^*}(\hat{\pi}_1^t, \hat{\pi}_1^{t+1}) & \geq \bD_{f^*}(\tilde{\pi}_1^t, \hat{\pi}_1^{t+1}) - \frac{1}{\tau'}\inner{\hat{g}(\w^t;\tilde{\B}_2^t) - g(\w^t)}{\hat{\pi}_1^{t+1} - \tilde{\pi}_1^t} + \bD_{f^*}(\hat{\pi}_1^t,\tilde{\pi}_1^t),\\\label{eq:3p_2}
		\bD_{f^*}(\hat{\pi}_1^{t+1},\pi_1) &\leq \bD_{f^*}(\hat{\pi}_1^t,\pi_1) + \frac{1}{\tau'}\inner{\hat{g}(\w^t;\B_2^t) - g(\w^t)}{\pi_1 - \hat{\pi}_1^{t+1}} - \bD_{f^*}(\hat{\pi}_1^t,\hat{\pi}_1^{t+1}),
	\end{align}
	where $\pi_1$ could be any $\pi_1\in\Pi_1$. The last term on the R.H.S. of \eqref{eq:3p_2} can be upper bounded by \eqref{eq:3p_1}. 
	\begin{align*}
		\bD_{f^*}(\hat{\pi}_1^{t+1},\pi_1) &\leq \bD_{f^*}(\hat{\pi}_1^t,\pi_1) +\frac{1}{\tau'}\inner{\hat{g}(\w^t;\B_2^t) - g(\w^t)}{\pi_1 - \tilde{\pi}_1^t} \\
		& \quad\quad\quad + \frac{1}{\tau'}\inner{\hat{g}(\w^t;\B_2^t)  - \hat{g}(\w^t;\tilde{\B}_2^t) }{\tilde{\pi}_1^t - \hat{\pi}_1^{t+1}} - \bD_{f^*}(\hat{\pi}_1^t,\tilde{\pi}_1^t) - \bD_{f^*}(\tilde{\pi}_1^t,\hat{\pi}_1^{t+1})
	\end{align*}
	Considering the strong convexity, we have
	\begin{align*}
		\bD_{f^*}(\hat{\pi}_1^{t+1},\pi_1) &\leq \bD_{f^*}(\hat{\pi}_1^t,\pi_1) +\frac{1}{\tau'}\inner{\hat{g}(\w^t;\B_2^t) - g(\w^t)}{\pi_1 - \tilde{\pi}_1^t} \\
		& \quad\quad\quad + \frac{1}{\tau'}\inner{\hat{g}(\w^t;\B_2^t)  - \hat{g}(\w^t;\tilde{\B}_2^t) }{\tilde{\pi}_1^t - \hat{\pi}_1^{t+1}}-\frac{1}{2L_f}\Norm{\hat{\pi}_1^t-\tilde{\pi}_1^t}^2 - \frac{1}{2L_f}\Norm{\tilde{\pi}_1^t - \hat{\pi}_1^{t+1}}^2
	\end{align*}
	Based on the Young's inequality, we further have
	\begin{align*}
		\frac{1}{\tau'}\inner{\hat{g}(\w^t;\B_2^t)  - \hat{g}(\w^t;\tilde{\B}_2^t) }{\tilde{\pi}_1^t - \hat{\pi}_1^{t+1}} &\leq \frac{L_f}{2(\tau')^2}\Norm{\hat{g}(\w^t;\B_2^t)  - \hat{g}(\w^t;\tilde{\B}_2^t)}^2 +  \frac{1}{2L_f}\Norm{\tilde{\pi}_1^t - \hat{\pi}_1^{t+1}}^2.
	\end{align*}
	Re-arranging the terms leads to
	\begin{align*}
		\inner{\hat{g}(\w^t;\B_2^t) - g(\w^t)}{\pi_1 - \tilde{\pi}_1^t} & \geq \tau' (\bD_{f^*}(\hat{\pi}_1^{t+1},\pi_1) - \bD_{f^*}(\hat{\pi}_1^t,\pi_1)) - \frac{L_f}{2\tau'}\Norm{\hat{g}(\w^t;\B_2^t)  - \hat{g}(\w^t;\tilde{\B}_2^t)}^2.
	\end{align*}
\end{proof} 
We can decompose $Q_i(\y^t,\y)$ into three terms
\begin{align*}
	Q_{i,2}(\y^t,\y) & = \L_{i,1}(\w^t,\pi_{i,1},\pi_{i,2}) - \L_{i,1}(\w^t,\pi_{i,1},\pi_{i,2}^{t+1})  \\
	& = \pi_{i,1}\left(\inner{\pi_{i,2}}{\w^t} - g_i^*(\pi_{i,2}) - \inner{\pi_{i,2}^{t+1}}{\w^t} + g_i^*(\pi_{i,2}^{t+1})\right),\\
	Q_{i,1}(\y^t,\y) &= \L_{i,1}(\w^t,\pi_{i,1},\pi_{i,2}^{t+1}) - \L_{i,2}(\w^t,\pi_{i,1}^{t+1},\pi_{i,2}^{t+1})\\&= \pi_{i,1}\L_{i,2}(\w^t,\pi_{i,2}^{t+1}) - f^*(\pi_{i,1}) - \pi_{i,1}^{t+1}\L_{i,2}(\w^t,\pi_{i,2}^{t+1}) + f^*(\pi_{i,1}^{t+1}), \\
	Q_{i,0}(\y^t,\y) &= \L_{i,1}(\w^t,\pi_{i,1}^{t+1},\pi_{i,2}^{t+1}) - \L_{i,1}(\w;\pi_{i,1}^{t+1},\pi_{i,2}^{t+1}) = \pi_i^{t+1} \inner{\pi_{i,2}^{t+1}}{\w^t-\w}.
\end{align*}

We upper bound these terms one by one by the following lemmas.

\begin{lemma}\label{lem:bound_Q2}
	We have $Q_{i,2}(\y^t,\y) \leq 0$ for any $\z_i\in\D$ and any $t=0,\dotsc, T-1$.
\end{lemma}
\begin{proof}
	Since $f$ is Lipschitz-continuous, convex and monotonically increasing, we have $0\leq \pi_{i,1}\leq C_f$. Besides, $\inner{\pi_{i,2}}{\w_t} - g_i^*(\pi_{i,2})\leq \inner{\pi_{i,2}^{t+1}}{\w_t} - g_i^*(\pi_{i,2}^{t+1})$ due to  $\pi_{i,2}^{t+1} = \arg\max_{\pi_{i,2}}\pi_{i,2}^\top \w_t - g_i^*(\pi_{i,2})$. We can conclude that $Q_{i,2}(\y_t,\y) \leq 0$. 
\end{proof}

\begin{lemma}\label{lem:bound_Q1}
	For any valid $\y=(\w, \pi_1, \pi_2)$, the term $\sum_{t=0}^{T-1}\E\left[\frac{1}{n}\sum_{\z_i\in\D} Q_{i,1}(\y^t,\y)\right]$ can be upper bounded as
	\begin{align*}
		&\sum_{t=0}^{T-1}\E\left[\frac{1}{n}\sum_{\z_i\in\D} Q_{i,1}(\y^t,\y)\right]\\ &\leq \frac{2n(C_fC_gC_\Omega+C_{f^*})}{B_1} + \frac{n\eta C T}{B_1} + \frac{\tau n L_fC_g^2C_\Omega^2}{B_1} + \frac{L_f T\sigma^2}{\tau B_2} + \frac{B_1 T L_f\sigma^2}{n \tau B_2} \colorbox{gray!40}{$-  \frac{\tau}{2B_1}\E\left[\sum_{t=0}^{T-1}\bD_{f^*}(\pi_1^t,\pi_1^{t+1})\right]$}.
	\end{align*}	
\end{lemma}
\begin{proof}
		We define that $\bff^*(\pi_1) \coloneqq \sum_{\z_i\in\D} f_i^*(\pi_{i,1})$ for any $\pi_1\in\Pi_1$,  $$\phi(\w^t,\pi_1)\coloneqq -\sum_{\z_i\in\D}\pi_{i,1}\L_2^i(\w^t;\pi_{i,2}^{t+1}) = -\sum_{\z_i\in\D} \pi_{i,1} g_i(\w^t), \quad h(\w^t,\pi_1) \coloneqq \phi(\w^t,\pi_1) + \sum_{\z_i\in\D} f_i^*(\pi_{i,1}).$$ Due to the update of rule of $\pi_{i,1}^t$ ($\pi_{i,1}^{t+1}=\pi_{i,1}^t$ for $i\notin \B_1^t$) and the convexity, we have
	\begin{align*}
		h(\w^t,\pi_1^{t+1}) &= \phi(\w^t,\pi_1^t) + (\phi(\w^t,\pi_1^{t+1})-\phi(\w^t,\pi_1^t)) +  \bff^*(\pi_1^{t+1})\\
		& \leq \phi(\w^t,\pi_1^t) \underbrace{- \sum_{\z_i\in\B_1^t} g_i(\w^t;\B_{i,2}^t) (\pi_{i,1}^{t+1} - \pi_{i,1}^t) + \sum_{\z_i\in\B_1^t}f_i^*(\pi_{i,1}^{t+1})}_{\coloneqq \text{\ding{161}}} + \sum_{\z_i\notin\B_1^t}f_i^*(\pi_{i,1}^{t+1})\\
		& \quad\quad\quad + \sum_{\z_i\in\B_1^t} (g_i(\w^t;\B_{i,2}^t) - g_i(\w^t))(\pi_{i,1}^{t+1} - \pi_{i,1}^t)
	\end{align*}
	Applying the three-point inequality (e.g. Lemma 1 of \citealt{zhang2020optimal}) leads to
	\begin{align*}
		& -\pi_{i,1}^{t+1} g_i(\w^t;\B_{i,2}^t) +   f_i^*(\pi_{i,1}^{t+1}) + \tau D_{f_i^*}(\pi_{i,1}^{t+1}, \pi_{i,1}) + \tau D_{f_i^*}(\pi_{i,1}^t,\pi_{i,1}^{t+1})\\
		& \leq -\pi_{i,1}g_i(\w^t;\B_{i,2}^t) + f_i^*(\pi_{i,1}) + \tau D_{f_i^*}(\pi_{i,1}^t, \pi_{i,1}),\quad \z_i\in\B_1^t.
	\end{align*}
	Add $\pi_{i,1}^t g_i(\w^t;\B_{i,2}^t)$ on both sides and re-arrange the terms. For $\z_i\in\B_1^t$, we have
	\begin{align*}
		& -g_i(\w^t;\B_{i,2}^t)(\pi_{i,1}^{t+1} - \pi_{i,1}^t ) +  f_i^*(\pi_{i,1}^{t+1}) \\
		& \leq -g_{i_t}(\w^t;\B_{i,2}^t)(\pi_{i,1} - \pi_{i,1}^t ) + f_i^*(\pi_{i,1}) + \tau D_{f^*}(\pi_{i,1}^t, \pi_{i,1}) - \tau D_{f_i^*}(\pi_{i,1}^{t+1}, \pi_{i,1}) - \tau D_{f_i^*}(\pi_{i,1}^t,\pi_{i,1}^{t+1}).
	\end{align*}
	The \ding{161} term can be upper bounded by summing the inequality above over all $\z_i\in\B_1^t$. Besides, note that $\pi_{i,1}^{t+1} = \pi_{i,1}^t$ for $\z_i\notin\B_1^t$ such that $\sum_{\z_i\notin\B_1^t} f_i^*(\pi_{i,1}^{t+1}) = \sum_{\z_i\notin\B_1^t} f_i^*(\pi_{i,1}^t)$.
	\begin{align*}
		& h(\w^t,\pi_1^{t+1}) \leq 	 \phi(\w^t,\pi_1^t) - \sum_{\z_i\in\B_1^t} g_i(\w^t;\B_{i,2}^t)(\pi_{i,1} - \pi_{i,1}^t ) + \sum_{\z_i\in\B_1^t} f_i^*(\pi_{i,1}) + \sum_{\z_i\notin\B_1^t}f_i^*(\pi_{i,1}^t) \\
		&  \quad\quad\quad + \tau \sum_{\z_i\in\B_1^t} D_{f_i^*}(\pi_{i,1}^t, \pi_{i,1})  - \tau\sum_{\z_{i_t}\in\B_1^t} D_{f_i^*}(\pi_{i,1}^{t+1}, \pi_{i,1}) \\
		&  \quad\quad\quad- \tau \sum_{\z_i\in\B_1^t} D_{f_i^*}(\pi_{i,1}^t,\pi_{i,1}^{t+1})+ \sum_{\z_i\in\B_1^t} \underbrace{(g_i(\w^t;\B_{i,2}^t) - g_i(\w^t))(\pi_{i,1}^{t+1} - \pi_{i,1}^t)}_{\coloneqq   \circledast}.
	\end{align*}
	Based on the $\frac{1}{L_f}$-strong convexity of $f^*(\pi_{i,1})$, the $ \circledast$ term for $\z_i\in\B_1^t$ can be bounded as
	\begin{align*}
		(g_i(\w^t;\B_{i,2}^t) - g_i(\w^t))(\pi_{i,1}^{t+1} - \pi_{i,1}^t) & \leq \frac{L_f}{\tau}\Norm{g_i(\w^t;\B_{i,2}^t) - g_i(\w^t)}^2 + \frac{\tau}{4L_f}\Norm{\pi_{i,1}^{t+1} - \pi_{i,1}^t}^2\\
		& \leq \frac{L_f}{\tau}\Norm{g_i(\w^t;\B_{i,2}^t) - g_i(\w^t)}^2 + \frac{\tau}{2} D_{f_i^*}(\pi_{i,1}^t, \pi_{i,1}^{t+1}).
	\end{align*}
	Taking the upper bound of $ \circledast$ into consideration leads to
	\begin{align*}
		h(\w^t,\pi_1^{t+1}) &\leq  \phi(\w^t,\pi_1^t) \underbrace{- \sum_{\z_i\in\B_1^t} g_i(\w^t;\B_{i,2}^t)(\pi_{i,1} - \pi_{i,1}^t )}_{\coloneqq \text{\ding{164}}} + \sum_{\z_i\in\B_1^t} f_i^*(\pi_{i,1}) +\sum_{\z_i\notin\B_1^t} f_i^*(\pi_{i,1}^t) \\
		& - \frac{\tau}{2} \sum_{\z_i\in\B_1^t}D_{f_i^*}(\pi_{i,1}^t,\pi_{i,1}^{t+1}) + \underbrace{\tau \sum_{\z_i\in\B_1^t} D_{f_i^*}(\pi_{i,1}^t, \pi_{i,1}) - \tau\sum_{\z_i\in\B_1^t} D_{f_i^*}(\pi_{i,1}^{t+1}, \pi_{i,1})}_{\coloneqq \text{\ding{169}}} \\
		& + \frac{L_f}{\tau}\sum_{i\in\B_1^t} \Norm{g_i(\w^t;\B_{i,2}^t) - g_i(\w^t)}^2 .
	\end{align*}
	Define that $\hat{g}(\w^t;\B_2^t) =\sum_{\z_i\in\B_1^t} \frac{n}{B_1} g_i(\w^t;\B_{i,2}^t)\e_i$ and $g(\w^t) = \sum_{\z_i\in\D} g_i(\w^t)\e_i$, where $\e_i\in\R^n$ is the indicator vector that only the $i$-th element is 1 while the others are 0. Note that $\E\left[\hat{g}(\w^t;\B_2^t) \right] = g(\w^t)$. Then, \ding{164} can be decomposed as
	\begin{align*}
		\text{\ding{164}} &= -\frac{B_1}{n}\sum_{\z_i\in\B_1^t} \frac{n}{B_1}g_i(\w^t;\B_{i,2}^t)(\pi_{i,1} - \pi_{i,1}^t)\\
		& = -\frac{B_1}{n}\inner{\hat{g}(\w^t;\B_2^t)}{\pi_1 - \pi_1^t}\\
		& =  -\frac{B_1}{n}\inner{g(\w^t)}{\pi_1 - \pi_1^t} - \frac{B_1}{n}\inner{\hat{g}(\w^t;\B_2^t) - g(\w^t)}{\pi_1 - \pi_1^t}\\
		& = -\frac{B_1}{n}\inner{g(\w^t)}{\pi_1 - \pi_1^t} - \frac{B_1}{n}\underbrace{\inner{\hat{g}(\w^t;\B_2^t) - g(\w^t)}{\pi_1 - \tilde{\pi}_1^t}}_{\coloneqq \text{\ding{168}}} - \frac{B_1}{n}\underbrace{\inner{\hat{g}(\w^t;\B_2^t) - g(\w^t)}{\tilde{\pi}_1^t - \pi_1^t}}_{\coloneqq \text{\ding{166}}},
	\end{align*}
	where $\tilde{\pi}_1^t$ is defined as
	\begin{align*}
		\tilde{\pi}_1^t &= \arg\min_{\pi_1}  \inner{\hat{g}(\w^t;\tilde{\B}_2^t) - g(\w^t)}{\pi_1} + \tau' \bD_{f^*}(\hat{\pi}_1^t, \pi_1),\\
		\hat{\pi}_1^{t+1} & = \arg\min_{\pi_1}  \inner{\hat{g}(\w^t;\B_2^t) - g(\w^t)}{\pi_1} + \tau' \bD_{f^*}(\hat{\pi}_1^t, \pi_1), 
	\end{align*}
	where $\tau'>0$, $\tilde{\B}_2^t$ is a ``virtual batch'' (never sampled in the algorithm) that is independent of but has the same size as $\B_2^t$. Based on Lemma~\ref{lem:juditsky_variant}, \ding{168} can be lower bounded as
	\begin{align*}
		\inner{\hat{g}(\w^t;\B_2^t) - g(\w^t)}{\pi_1 - \tilde{\pi}_1^t} & \geq -\tau' \bD_{f^*}(\hat{\pi}_1^t, \pi_1) + \tau' \bD_{f^*}(\hat{\pi}_1^{t+1}, \pi_1) - \frac{L_f}{2\tau'}\Norm{\hat{g}(\w^t;\B_2^t) - \hat{g}(
			\w^t;\tilde{\B}_2^t)}^2. 
	\end{align*}
	Thus, taking the expectation of the equation above w.r.t. the randomness in iteration $t$ leads to
	\begin{align*}
		\E_t\left[\frac{B_1}{n}\inner{\hat{g}(\w^t;\B_2^t) - g(\w^t)}{\pi_1 - \tilde{\pi}_1^t} \right] & \geq \frac{\tau' B_1}{n}\left(-\E_t\left[\bD_{f^*}(\hat{\pi}_1^t, \pi_1)\right] + \E_t\left[\bD_{f^*}(\hat{\pi}_1^{t+1}, \pi_1)\right]\right) - \frac{B_1\sigma^2 L_f}{\tau' B_2}.
	\end{align*}
	Note that $\E_t\left[\text{\ding{166}}\right] = \E_t\left[\inner{\hat{g}(\w^t;\B_2^t) - g(\w^t)}{\tilde{\pi}_1^t - \pi_1^t}\right] = 0$ since both $\pi_1^t$ and $\tilde{\pi}_1^t$ are independent of $\B_2^t$. Besides, we have $\text{\ding{169}} = \sum_{\z_i\in\B_1^t} D_{f_i^*}(\pi_{i,1}^t, \pi_{i,1}) - \sum_{\z_i\in\B_1^t} D_{f_i^*}(\pi_{i,1}^{t+1}, \pi_{i,1}) = \sum_{\z_i\in\D} D_{f_i^*}(\pi_{i,1}^t, \pi_{i,1}) - \sum_{\z_i\in\D} D_{f_i^*}(\pi_{i,1}^{t+1}, \pi_{i,1}) = \bD_{f^*}(\pi_1^t,\pi_1) - \bD_{f^*}(\pi_1^{t+1}, \pi_1)$ and $\sum_{\z_i\in\B_1^t} D_{f_i^*}(\pi_{i,1}^t,\pi_{i,1}^{t+1}) = \sum_{\z_i\in\D} D_{f_i^*}(\pi_{i,1}^t,\pi_{i,1}^{t+1}) = \bD_{f^*}(\pi_1^t,\pi_1^{t+1})$ because $\pi_{i,1}^{t+1} = \pi_{i,1}^t$ for $i\notin\B_1^t$. 
	\begin{align}\label{eq:exp_Bt}
		& \E_t\left[h(\w^t,\pi_1^{t+1})\right] \\\nonumber
		&\leq  \phi(\w^t,\pi_1^t) + \left(1-\frac{B_1}{n}\right)\bff^*(\pi_1^t) + \frac{B_1 (\E_t\left[\phi(\w^t,\pi_1)\right] - \phi(\w^t,\pi_1^t))}{n} + \frac{B_1}{n} \E_t\left[\bff^*(\pi_1)\right]  \\\nonumber
		&\quad\quad  -\frac{\tau}{2}\E_t\left[\bD_{f^*}(\pi_1^t,\pi_1^{t+1})\right] + \frac{\tau' B_1}{n}(\E_t\left[\bD_{f^*}(\hat{\pi}_1^t, \pi_1)\right] - \E_t\left[\bD_{f^*}(\hat{\pi}_1^{t+1}, \pi_1)\right]) + \frac{nL_f\sigma^2}{\tau' B_2}  \\\nonumber
		&\quad\quad + \tau  \E_t\left[\bD_{f^*}(\pi_1^t,\pi_1)\right] - \tau \E_t\left[\bD_{f^*}(\pi_1^{t+1},\pi_1)\right] + \frac{L_f B_1\sigma^2}{\tau B_2}\\\nonumber
		& = \left(1-\frac{B_1}{n}\right) h(\w^t,\pi_1^t) + \frac{B_1}{n}\E_t\left[h(\w^t,\pi_1)\right] + \tau ( \E_t\left[\bD_{f^*}(\pi_1^t, \pi_1)\right]-  \E_t\left[\bD_{f^*}(\pi_1^{t+1}, \pi_1)\right])  \\\nonumber
		& \quad\quad - \frac{\tau}{2}\E_t\left[\bD_{f^*}(\pi_1^t,\pi_1^{t+1})\right] + \frac{\tau' B_1}{n}(\E_t\left[\bD_{f^*}(\hat{\pi}_1^t, \pi_1)\right] - \E_t\left[\bD_{f^*}(\hat{\pi}_1^{t+1}, \pi_1)\right]) \\\nonumber
		&\quad\quad + \frac{nL_f\sigma^2}{\tau' B_2} + \tau  \E_t\left[\bD_{f^*}(\pi_1^t,\pi_1)\right] - \tau \E_t\left[\bD_{f^*}(\pi_1^{t+1},\pi_1)\right] + \frac{L_f B_1\sigma^2}{\tau B_2}.
	\end{align}
	Subtract $\E_t\left[h(\w^t,\pi_1)\right]$ from both sides and use the tower property of conditional expectation.
	\begin{align*}
		& \E\left[h(\w^t,\pi_1^{t+1}) - h(\w^t,\pi_1)\right] \leq \left(1-\frac{B_1}{n}\right) \E\left[h(\w^t,\pi_1^t) -h(\w^t,\pi_1)\right] + \tau\E\left[ \bD_{f^*}(\pi_1^t, \pi_1)-  \bD_{f^*}(\pi_1^{t+1},\pi_1)\right]\\
		&\quad\quad\quad - \frac{\tau}{2}\E\left[\bD_{f^*}(\pi_1^t, \pi_1^{t+1})\right]+ \frac{L_f B_1\sigma^2}{\tau B_2} + \frac{L_f B_1 \sigma^2}{\tau' B_2} +  \frac{\tau' B_1}{n}\E\left[\bD_{f^*}(\hat{\pi}_1^t, \pi_1) - \bD_{f^*}(\hat{\pi}_1^{t+1},\pi_1)\right].
	\end{align*}
	Let $\Delta^t = h(\w^t,\pi_1^{t+1}) - h(\w^t,\pi_1)$. Thus,
	\begin{align*}
		& \E\left[h(\w^t,\pi_1^t) - h(\w^t,\pi_1) - \Delta^{t-1}\right] = \E\left[h(\w^t,\pi_1^t) - h(\w^t,\pi_1) - h(\w^{t-1},\pi_1^t) + h(\w^{t-1},\pi_1)\right]\\
		&= \E\left[\sum_{\z_i\in\D} (\pi_{i,1} - \pi_{i,1}^t) (g_i(\w^t)  - g_i(\w^{t-1}))\right] \leq n C_f C_g \eta \E\left[\Norm{\w^t-\w^{t-1}}\right]\leq n\eta C_f C_g  \sqrt{C_f^2(C_g^2 + \zeta^2/B_2)}.
	\end{align*}
	We define $C = C_f C_g  \sqrt{C_f^2(C_g^2 + \zeta^2/B_2)}$. Do the telescoping sum for $t=1,\dotsc, T$.
	\begin{align*}
		\E\left[\sum_{t=0}^{T-1} \frac{B_1}{n}\Delta^t\right] & \leq (\Delta^0-\Delta^T)  + n\eta C T + \tau \bD_{f^*}(\pi_1^0,\pi_1) -\frac{\tau}{2} \E\left[\sum_{t=0}^{T-1} \bD_{f^*}(\pi_1^t,\pi_1^{t+1})\right]\\
		& \quad\quad + \frac{L_f B_1\sigma^2 T}{\tau B_2} + \frac{B_1 TL_f\sigma^2}{\tau' B_2} +  \frac{\tau B_1}{n}\bD_{f^*}(\hat{\pi}_1^0,\pi_1) .
	\end{align*}
	Consider that $\frac{1}{n}\sum_{\z_i\in\D} Q_{i,1}(\y^t,\y) = \frac{1}{n} \left(h(\w^t,\pi_1^{t+1}) - h (\w^t,\pi_1)\right) = \frac{\Delta^t}{n}$.
	\begin{align*}
		\sum_{t=0}^{T-1}\E\left[\frac{1}{n}\sum_{\z_i\in\D} Q_{i,1}(\y^t,\y)\right] &\leq \frac{\Delta^0 - \Delta^T}{B_1} + \frac{n\eta C T}{B_1} + \frac{\tau \bD_{f^*}(\pi_1^0,\pi_1)}{B_1} + \frac{L_f T\sigma^2}{\tau B_2} + \frac{TL_f\sigma^2}{\tau' B_2} \\
		& \quad\quad + \frac{\tau' \bD_{f^*}(\hat{\pi}_1^0,\pi_1)}{n}- \frac{\tau}{2B_1}\E\left[\sum_{t=0}^{T-1}\bD_{f^*}(\pi_1^t,\pi_1^{t+1})\right].
	\end{align*}
	The numerator in the first term on the right hand side can be upper bounded as follows.
	\begin{align*}
		\Delta^0 - \Delta^T & = h(\w^0,\pi_1^1) - h(\w^0,\pi_1) - h(\w^T,\pi_1^{T+1})+h(\w^T,\pi_1) \\
		& = - \sum_{\z_i\in\D}(\pi_{i,1}^1-\pi_{i,1})g_i(\w^0) - \sum_{\z_i\in\D} (\pi_{i,1} - \pi_{i,1}^{T+1})g_i(\w^T) + \sum_{\z_i\in\D} f_i^*(\pi_{i,1}^1)  - \sum_{\z_i\in\D} f_i^*(\pi_{i,1}^{T+1}) \\
		& \leq 2nC_f C_g C_\Omega + 2nC_{f^*},
	\end{align*}
	On the other hand, if we set $\hat{\pi}_{i,1}^0 =\nabla f(u_i^0)$ and $u_i^0 = g_i(\w^0)$, we have $D_{f_i^*}(\pi_{i,1}^0,\pi_{i,1}) = D_{f_i^*}(\nabla f(u_i^0), \nabla f(g_i(\bar{\w}))) = D_{f_i}(u_i^0, g_i(\bar{\w})) \leq \frac{L_f}{2}\Norm{u_i^0 - g_i(\bar{\w})}^2\leq \frac{L_f C_g^2 C_\Omega^2}{2}$ such that $\bD_{f^*}(\pi_1^0,\pi_1) = \sum_{\z_i\in\D} D_{f_i^*}(\pi_{i,1}^0,\pi_{i,1})\leq \frac{n L_f C_g^2 C_\Omega^2}{2}$. The proof concludes by setting $\tau' = \frac{n\tau}{B_1}>1$.
\end{proof}

\begin{lemma}\label{lem:bound_Q0}
	We have
	\begin{align*}
\sum_{t=0}^{T-1}\E\left[\frac{1}{n}\sum_{\z_i\in \D} Q_{i,0}(\y^t,\y)\right] & \leq \frac{T C_f^2C_g^2}{\eta} + \frac{T C_f^2 (\zeta^2+C_g^2)}{\eta\min\{B_1,B_2\}} + \frac{\eta C_\Omega^2}{2} + \colorbox{gray!40}{$\frac{\tau}{2B_1}\E\left[\sum_{t=0}^{T-1}\bD_{f^*}(\pi_1^t,\pi_1^{t+1})\right]$} \\
&  \quad\quad + \frac{B_1 L_f C_g^2 C_\Omega^2 T}{n\tau}.
	\end{align*}
	
\end{lemma}

\begin{proof}
		Based on the definition of $Q_{i,0}(\y^t,\y)$, we can derive that
	\begin{align}\nonumber
& \sum_{t=0}^{T-1}\E\left[\frac{1}{n}\sum_{\z_i\in\D} Q_{i,0}(\y^t,\y)\right]
= \sum_{t=0}^{T-1}\E\left[\frac{1}{n}\sum_{\z_i\in\D}  \inner{\pi_{i,1}^{t+1}\pi_{i,2}^{t+1}}{\w^t - \w}\right]\\\nonumber
& = \sum_{t=0}^{T-1}\E\left[\frac{1}{n}\sum_{\z_i\in\D} \inner{\pi_{i,1}^t\pi_{i,2}^{t+1}}{\w^{t+1} - \w}\right] + \sum_{t=0}^{T-1}\E\left[\frac{1}{n}\sum_{\z_i\in\D} \inner{\pi_{i,1}^{t+1}\pi_{i,2}^{t+1}}{\w^t - \w^{t+1}}\right]\\\label{eq:Q0_decomp}
& \quad\quad  +\sum_{t=0}^{T-1}\E\left[\frac{1}{n}\sum_{\z_i\in\D} \inner{(\pi_{i,1}^{t+1} - \pi_{i,1}^t)\pi_{i,2}^{t+1}}{\w^{t+1} - \w}\right].
\end{align}
	The second term on the right hand side of \eqref{eq:Q0_decomp} can be upper bounded by $$\sum_{t=0}^{T-1}\E\left[\frac{1}{n}\sum_{\z_i\in\D}  \inner{\pi_{i,1}^{t+1}\pi_{i,2}^{t+1}}{\w^t - \w^{t+1}}\right]\leq \E\left[\sum_{t=0}^{T-1}\left( \frac{C_f^2 C_g^2}{\eta} +\frac{\eta}{4}\Norm{\w^{t+1}-\w^t}^2\right)\right].$$
	The first term on the right hand side of \eqref{eq:Q0_decomp} can be upper bounded by
	\begin{align*}
& \sum_{t=0}^{T-1}\E\left[\frac{1}{n}\sum_{\z_i\in\D}  \inner{\pi_{i,1}^t\pi_{i,2}^{t+1}}{\w^{t+1} - \w}\right] 
		=  \sum_{t=0}^{T-1}\E\left[ \inner{\frac{1}{n}\sum_{\z_i\in\D} \pi_{i,1}^t\pi_{i,2}^{t+1} - \frac{1}{B_1}\sum_{i\in \B_1^t}\pi_{i,1}^t\pi_{i,2}^{t+1}(\B_{i,2}^t)}{\w^{t+1} - \w}\right] \\
		& \quad\quad\quad\quad + \sum_{t=0}^{T-1}\E\left[ \inner{ \frac{1}{B_1}\sum_{i\in\B_1^t}\pi_{i,1}^t\pi_{i,2}^{t+1}(\B_{i,2}^t)}{\w^{t+1} - \w}\right] \\
		& = \sum_{t=0}^{T-1}\E\left[ \inner{\frac{1}{n}\sum_{\z_i\in\D} \pi_{i,1}^t\pi_{i,2}^{t+1} - \frac{1}{B_1}\sum_{i\in\B_1^t}\pi_{i,1}^t\pi_{i,2}^{t+1}(\B_{i,2}^t)}{\w^{t+1} - \w^t}\right] \\
		& \quad\quad\quad\quad + \sum_{t=0}^{T-1}\E\left[ \inner{ \frac{1}{B_1}\sum_{i\in\B_1^t}\pi_{i,1}^t\pi_{i,2}^{t+1}(\B_{i,2}^t)}{\w^{t+1} - \w}\right].
	\end{align*}
	The last equality above uses $\E\left[ \inner{\frac{1}{n}\sum_{\z_i\in\D} \pi_{i,1}^t\pi_{i,2}^{t+1} - \frac{1}{B_1}\sum_{i\in\B_1^t}\pi_{i,1}^t\pi_{i,2}^{t+1}(\B_{i,2}^t)}{\w^t - \w}\right] = 0$. Moreover,
	\begin{align}\label{eq:Q0_2nd}
		\E\left[ \inner{\frac{1}{n}\sum_{\z_i\in\D} \pi_{i,1}^t\pi_{i,2}^{t+1} - \frac{1}{B_1}\sum_{i\in\B_1^t}\pi_{i,1}^t\pi_{i,2}^{t+1}(\B_{i,2}^t)}{\w^{t+1} - \w^t}\right] \leq \frac{ C_f^2(\zeta^2+C_g^2)}{\eta \min\{B_1,B_2\}} +  \frac{\eta}{4}\Norm{\w^{t+1}-\w^t}^2.
	\end{align}
	According to the three-point inequality (Lemma 1 in \citealt{zhang2020optimal}), we have 
	\begin{align}\label{eq:Q0_3rd}
		\sum_{t=0}^{T-1}\E\left[ \inner{ \frac{1}{B_1}\sum_{i\in\B_1^t}\pi_{i,1}^t\pi_{i,2}^{t+1}(\B_{i,2}^t)}{\w^{t+1} - \w}\right] \leq \frac{\eta}{2}\Norm{\w^t-\w}^2 - \frac{\eta}{2}\Norm{\w^{t+1}-\w}^2 - \frac{\eta}{2}\Norm{\w^t-\w^{t+1}}^2.
	\end{align}
	Besides, the third term on the R.H.S. of \eqref{eq:Q0_decomp} can be upper bounded as follows based on the Young's inequality with a constant $\rho>0$.
	\begin{align*}
		& \sum_{t=0}^{T-1}\E\left[\frac{1}{n}\sum_{\z_i\in\D} \inner{(\pi_{i,1}^{t+1} - \pi_{i,1}^t)\pi_{i,2}^{t+1}}{\w^{t+1} - \w}\right] \leq \sum_{t=0}^{T-1}\E\left[\frac{C_g^2\rho}{2n}\sum_{\z_i\in\D} \Norm{\pi_{i,1}^{t+1} - \pi_{i,1}^t}^2 + \frac{\Norm{\w^{t+1}-\w}^2}{2\rho}\right]\\
		& \leq \sum_{t=0}^{T-1}\E\left[ \frac{L_f C_g^2\rho}{n} \bD_{f^*}(\pi_1^t,\pi_1^{t+1})+ \frac{\Norm{\w^{t+1}-\w}^2}{2\rho}\right],
	\end{align*}
	where the last inequality is due to $L_f$-smoothness of $f$. Choose $\rho = \frac{n\tau}{2B_1 L_f C_g^2}$. Besides, we also have $\Norm{\w^{t+1}-\w}^2\leq C_\Omega^2$. 
	\begin{align}\label{eq:Q0_4th}
		\sum_{t=0}^{T-1}\E\left[\frac{1}{n}\sum_{\z_i\in\D} \inner{(\pi_{i,1}^{t+1} - \pi_{i,1}^t)\pi_{i,2}^{t+1}}{\w^{t+1} - \w}\right] \leq \frac{\tau}{2B_1}\E\left[\sum_{t=0}^{T-1}\bD_{f^*}(\pi_1^t,\pi_1^{t+1})\right] + \frac{B_1 L_f C_g^2 C_\Omega^2 T}{n\tau}.
	\end{align}
	Plugging \eqref{eq:Q0_2nd},  \eqref{eq:Q0_3rd}, and \eqref{eq:Q0_4th} into \eqref{eq:Q0_decomp} leads to
	\begin{align*}
		\sum_{t=0}^{T-1}\E\left[\frac{1}{n}\sum_{\z_i\in\D} Q_{i,0}(\y^t,\y)\right] & \leq \frac{T C_f^2 C_g^2}{\eta} + \frac{T C_f^2 (\zeta^2+C_g^2)}{\eta\min\{B_1,B_2\}} + \frac{C_\Omega^2}{2\eta} \\
		& \quad\quad + \frac{\tau}{2B_1}\E\left[\sum_{t=0}^{T-1}\bD_{f^*}(\pi_1^t,\pi_1^{t+1})\right] + \frac{B_1 L_f C_g^2 C_\Omega^2 T}{n\tau}.
	\end{align*}
\end{proof}

\subsection{Proof of Theorem~\ref{thm:cvx_sox-simp}}
\begin{proof}
	Based on Lemma~\ref{lem:bound_Q2}, Lemma~\ref{lem:bound_Q1}, and Lemma~\ref{lem:bound_Q0}, we have
	\begin{align*}
		& \E\left[\frac{1}{T}\sum_{t=0}Q(\y^t,\y)\right] \\
		&\leq \frac{2n(C_fC_gC_\Omega+C_{f^*})}{B_1 T} + \frac{n C}{\eta B_1} + \frac{\tau n L_f C_g^2 C_\Omega^2}{B_1 T} + \frac{L_f\sigma^2}{\tau B_2} + \frac{B_1}{n\tau B_2} +  \frac{C_f^2 C_g^2}{\eta} + \frac{C_f^2(\zeta^2+C_g^2)}{\eta \min\{B_1,B_2\}} \\
		&\quad\quad + \frac{\eta C_\Omega^2}{2 T} + \frac{B _1 L_f C_g^2 C_\Omega^2}{n\tau},
	\end{align*} 
	where $F(\bar{\w}^T) = \L(\bar{\w}^T,\pi_1,\pi_2)$, $\bar{\w}^T = \frac{1}{T}\sum_{t=0}^{T-1}\w^t$ and $\L_{i,1}(\w,\pi_{i,1}^{t+1},\pi_{i,2}^{t+1})\leq f(g_i(\w))$ such that we have $\frac{1}{n}\sum_{\z_i\in\D} \L_{i,1}(\w,\pi_{i,1}^{t+1},\pi_{i,2}^{t+1})\leq \frac{1}{n}\sum_{\z_i\in\D} f(g_i(\w)) = F(\w)$. 
	\begin{align*}
		\E\left[\frac{1}{T}\sum_{t=0}Q(\y^t,\y)\right]  & = \E\left[\frac{1}{T}\sum_{t=0}(\L(\w^t,\pi_1,\pi_2) - \L(\w,\pi_1^{t+1},\pi_2^{t+1}))\right].  
	\end{align*}
Thus, to ensure that $\max_{\y} \E\left[\frac{1}{T}\sum_{t=0}Q(\y^t,\y)\right]\leq \epsilon$, we can make
	\begin{align*}
		& \eta = \max\left\{\frac{9nC}{B_1\epsilon}, \frac{9C_f^2C_g^2}{\epsilon}, \frac{9C_f^2(\zeta^2+C_g^2)}{\min\{B_1,B_2\}\epsilon}\right\},\quad \tau = \max\left\{\frac{9L_f \sigma^2}{B_2\epsilon}, \frac{9B_1}{n B_2\epsilon}, \frac{9B_1 L_f C_g^2 C_\Omega^2}{n\epsilon}\right\},\\
		&  \resizebox{\hsize}{!}{$T = \max\left\{\frac{18n(C_fC_gC_\Omega+C_{f^*})}{B_1 \epsilon}, \frac{81nL_f^2C_g^2C_\Omega^2\sigma^2}{B_1B_2\epsilon^2}, \frac{81L_f C_g^2C_\Omega^2}{B_2\epsilon^2}, \frac{81L_f^2 C_g^4C_\Omega^4}{\epsilon^2}, \frac{81 nC C_\Omega^2}{2B_1\epsilon^2}, \frac{81C_f^2C_g^2C_\Omega^2}{2\epsilon^2}, \frac{81C_f^2C_\Omega^2(\zeta^2+C_g^2)}{2\min\{B_1,B_2\}\epsilon^2}\right\}$}.
	\end{align*}
\end{proof}

\section{The Challenge of Deriving Strong Duality Gap}\label{sec:mistake}

In \eqref{eq:exp_Bt} of Lemma~\ref{lem:bound_Q1}, we used the equation for any $\pi_{i,1}$
\begin{align*}
 \boxed{\E\left[\sum_{\z_i\notin \B_1^t} f^*(\pi_{i,1})\right] = \left(1-\frac{B_1}{n}\right)\E\left[\sum_{\z_i\in \D}f^*(\pi_{i,1})\right],} 
\end{align*}
which meant to take expectation on the sampled mini-batch $\B_1^t$. To derive the upper bound of duality gap $\E[\max_{\w,\pi_1,\pi_2}(\L(\bar{\w}^T,\pi_1,\pi_2) - \L(\w, \bar{\pi}_1^T,\bar{\pi}_2^T))]$, we plugged in $\pi_{i,1} = \bar{\pi}_{i,1}^*$, $\bar{\pi}_{i,1}^* \coloneqq \arg\max_{\pi_{i,1}}\inner{\pi_{i,1}}{\L_{i,2}(\bar{\w}^T,\bar{\pi}_{i,2}^*)} - f^*(\pi_{i,1})$. However, $\bar{\pi}_{i,1}^*$ is not a valid value to plug in because $\bar{\pi}_{i,1}^*$ also depends on $\B_1^t$. Thus, the conclusion in Lemma 10 and Theorem 3 of our previous version are weakened. To be specific, we can only derive an upper bound for $\max_{\w,\pi_1,\pi_2}\E[\L(\bar{\w}^T,\pi_1,\pi_2) - \L(\w, \bar{\pi}_1^T,\bar{\pi}_2^T)]$, which is weaker than the duality gap (see Example 1 in \citet{alacaoglu2022complexity}). Unlike the duality gap, the upper bound on $\max_{\w,\pi_1,\pi_2}\E[\L(\bar{\w}^T,\pi_1,\pi_2) - \L(\w, \bar{\pi}_1^T,\bar{\pi}_2^T)]$ does not necessarily implies the bound for primal objective gap $\E[F(\bar{\w}^T) - F(\w^*)]$.

\section{Proof of Theorem~\ref{thm:scvx_sox}}

\begin{lemma}\label{lem:dual_contraction}
 Under Assumptions~\ref{asm:var},~\ref{asm:bounded_dom},~\ref{asm:R}, it is satisfied that 
 \begin{align*}
& \E[Q_1(\z^t,\z)] \leq \left(1 - \frac{B_1}{n}\right)\E[Q_1(\z^{t-1},\z)] +  \frac{C_1}{\eta_{t-1}} + \frac{L_f B_1 \sigma^2}{\tau_t n B_2}\\
& \quad\quad  + \left(\tau_t + 1 - \frac{B_1}{n}\right) \frac{1}{n}\E[\bD(\pi^t, \pi)]-  (\tau_t + 1)  \frac{1}{n}\E\left[\bD(\pi^{t+1}, \pi)\right]  - \colorbox{gray!40}{$\frac{\tau_t}{2}  \frac{1}{n}\E\left[\bD(\pi^t,\pi^{t+1}) \right]$},
 \end{align*}
 where $\z^t \coloneqq (\w^t,\pi^{t+1})$, $\z \coloneqq (\w,\pi)$ (independent of the randomness in the algorithm),  $Q_1(\z^t,\z) \coloneqq \L(\w^t, \pi) - \L(\w^t, \pi^{t+1})$, $\bD(\pi,\pi')\coloneqq \sum_{\z_i\in\D} D_{f_i^*}(\pi_i,\pi_i')$ for any $\pi,\pi'$, $C_1\coloneqq 2C_f C_g (C_f \sqrt{C_g^2+\zeta^2} + C_R)$.
 \end{lemma}
 \begin{proof}
 We define $h(\w,\pi) = -\sum_{\z_i\in\D} \pi_i g_i(\w) + \sum_{\z_i\in\D} f_i^*(\pi_i)$. Due to the update of rule of $\pi_i^t$ ($\pi_i^{t+1}=\pi_i^t$ for $i\notin \B_1^t$), we have
\begin{align}\nonumber
h(\w^t,\pi^{t+1}) & = -\sum_{\z_i\in\D} \pi_i^t g_i(\w^t) - \sum_{\z_i\in\D} (\pi_i^{t+1} - \pi_i^t) g_i(\w) + \sum_{\z_i\in\D} f_i^*(\pi_i^{t+1})\\\nonumber
& = -\sum_{\z_i\in\D} \pi_i^t g_i(\w^t) \underbrace{- \sum_{\z_i\in\D} (\pi_i^{t+1} - \pi_i^t) g_i(\w^t, \B_{i,2}^t) + \sum_{\z_i\in\D} f_i^*(\pi_i^{t+1})}_{\coloneqq \text{\ding{161}}} \\\label{eq:starter}
& \quad\quad + \sum_{i\in\B_1^t}(g_i(\w^t, \B_{i,2}^t) - g_i(\w^t))(\pi_i^{t+1} - \pi_i^t).
\end{align}
We define that $\bar{\pi}_i^{t+1}\coloneqq \arg\max_{\pi_i} \left\{\pi_i g_i(\w^t;\B_{i,2}^t) - f_i^*(\pi_i) - \tau_t D_{f_i^*}(\pi_i^t,\pi_i)\right\}$, $\forall \z_i\in\D$. Note that $\bar{\pi}_i^{t+1} = 
\pi_i^{t+1}$ when $\z_i\in\B_1^t$. Applying Lemma 3.8 in \citet{lan2020first} leads to
\begin{align*}
&  (\pi - \bar{\pi}_i^{t+1})g_i(\w^t;\B_{i,2}^t) +   f_i^*(\pi_i^{t+1}) - f_i^*(\pi_i) \\
& \leq \tau_t D_{f_i^*}(\pi_i^t, \pi_i)-  (\tau_t + 1) D_{f_i^*}(\bar{\pi}_i^{t+1}, \pi_i) - \tau_t D_{f_i^*}(\pi_i^t,\bar{\pi}_i^{t+1}).
\end{align*}
Add $\pi_i^t g_i(\w^t;\B_{i,2}^t)$ on both sides and re-arrange the terms. For $\z_i\in \D$, we have
\begin{align}\label{eq:3p_starter}
&  - (\bar{\pi}_i^{t+1} - \pi_i^t)g_i(\w^t;\B_{i,2}^t) +   f_i^*(\bar{\pi}_i^{t+1})  \\\nonumber
& \leq  - (\pi_i - \pi_i^t)g_i(\w^t;\B_{i,2}^t) + f_i^*(\pi_i)  + \tau_t D_{f_i^*}(\pi_i^t, \pi_i)-  (\tau_t + 1) D_{f_i^*}(\bar{\pi}_i^{t+1}, \pi_i) - \tau_t D_{f_i^*}(\pi_i^t,\bar{\pi}_i^{t+1}).
\end{align}
Since $\bar{\pi}_i^{t+1}$ is independent of $\B_1^t$, taking conditional expectation on $\B_1^t$ leads to
\begin{align*}
& \E[D_{f_i^*}(\pi_i^{t+1},\pi_i)\mid \F_t] = \frac{B_1}{n} \E[D_{f_i^*}(\bar{\pi}_i^{t+1},\pi_i)\mid \F_t] + \left(1 - \frac{B_1}{n}\right)D_{f_i^*}(\pi_i^t,\pi_i),\\
& \E[D_{f_i^*}(\pi_i^t,\pi_i^{t+1})\mid \F_t] = \frac{B_1}{n} \E[D_{f_i^*}(\pi_i^t,\bar{\pi}_i^{t+1})\mid \F_t],\\
& \E[\pi_i^{t+1}\mid \F_t] = \frac{B_1}{n} \E[\bar{\pi}_i^{t+1}\mid \F_t] + \left(1 - \frac{B_1}{n}\right) \pi_i^t,\\
& \E[f_i^*(\pi_i^{t+1})\mid \F_t] = \frac{B_1}{n}\E[ f_i^*(\bar{\pi}_i^{t+1})\mid \F_t] + \left(1 - \frac{B_1}{n}\right) f_i^*(\pi_i^t).
\end{align*}
We define $\bD(\pi,\pi')\coloneqq \sum_{\z_i\in\D} D_{f_i^*}(\pi_i,\pi_i')$. By multiply both sides of \eqref{eq:3p_starter} by $\frac{B_1}{n}$ and plug the above in, we can bound $\E[\text{\ding{161}}\mid \F_t]$ by  
\begin{align*}
& \E[\text{\ding{161}}\mid \F_t] = - \sum_{\z_i\in\D}\E\left[\left(\pi_i^{t+1} - \pi_i^t\right)g_i(\w^t;\B_{i,2}^t)\mid \F_t\right] +   \sum_{\z_i\in\D}\E\left[f_i^*(\pi_i^{t+1})\mid \F_t\right]  \\
& \leq  \left(1 - \frac{B_1}{n}\right)\sum_{\z_i\in\D} f_i^*(\pi_i^t) -  \frac{B_1}{n}\sum_{\z_i\in\D} \E[(\pi_i - \pi_i^t)g_i(\w^t;\B_{i,2}^t)\mid \F_t] + \frac{B_1}{n} \sum_{\z_i\in\D} f_i^*(\pi_i)  \\
& \quad\quad + \left(\tau_t + 1 - \frac{B_1}{n}\right) \bD(\pi^t, \pi)-  (\tau_t + 1) \E\left[\bD(\pi^{t+1}, \pi)\mid \F_t\right] - \tau_t \E\left[\bD(\pi^t,\pi^{t+1})\mid \F_t\right].
\end{align*}
Combine \eqref{eq:starter} and the equation above.
\begin{align*}
& \E\left[h(\w^t,\pi^{t+1})\mid \F_t\right] \\
& = -\sum_{\z_i\in\D} \pi_i^t g_i(\w^t) + \left(1 - \frac{B_1}{n}\right) \sum_{\z_i\in\D} f_i^*(\pi_i^t) - \frac{B_1}{n}\sum_{\z_i\in\D}(\pi_i - \pi_i^t)g_i(\w^t) + \frac{B_1}{n} \sum_{\z_i\in\D} f_i^*(\pi_i)\\
& \quad\quad  + \left(\tau_t + 1 - \frac{B_1}{n}\right)  \bD(\pi^t, \pi)-  (\tau_t + 1) \E\left[ \bD(\pi^{t+1}, \pi)\mid \F_t\right] - \tau_t \E\left[ \bD(\pi^t,\pi^{t+1})\mid \F_t\right] \\
& \quad\quad + \underbrace{\E\left[\sum_{\z_i\in\B_1^t}(g_i(\w^t, \B_{i,2}^t) - g_i(\w^t))(\pi_i^{t+1} - \pi_i^t)\mid \F_t\right]}_{\circledast}.
\end{align*}
Based on the $\frac{1}{L_f}$-strong convexity of $f_i^*(\pi_i)$, the $ \circledast$ term for $\z_i\in\B_1^t$ can be bounded as
\begin{align*}
 \circledast & \leq \frac{L_f}{\tau_t}\E\left[\sum_{\z_i\in\B_1^t}\Norm{g_i(\w^t;\B_{i,2}^t) - g_i(\w^t)}^2\mid \F_t\right] + \frac{\tau_t}{4L_f}\E\left[\sum_{\z_i\in\B_1^t}\Norm{\pi_i^{t+1} - \pi_i^t}^2\mid \F_t\right]\\
& \leq \frac{L_f B_1 \sigma^2}{\tau_t B_2} + \frac{\tau_t}{2} \E\left[\sum_{\z_i\in\B_1^t}D_{f_i^*}(\pi_i^t, \pi_i^{t+1})\mid \F_t\right] .
\end{align*}
Taking the fact $\sum_{\z_i\in\B_1^t}D_{f_i^*}(\pi_i^t, \pi_i^{t+1}) = \bD(\pi^t,\pi^{t+1})$ (due to the update formula of the dual variable) and the upper bound of $ \circledast$ into consideration leads to
\begin{align}\label{eq:intermediate}
& \E\left[h(\w^t,\pi^{t+1})\mid \F_t\right] \\\nonumber
& = -\sum_{\z_i\in\D} \pi_i^t g_i(\w^t) + \left(1 - \frac{B_1}{n}\right) \sum_{\z_i\in\D} f_i^*(\pi_i^t) - \frac{B_1}{n}\sum_{\z_i\in\D}(\pi_i - \pi_i^t)g_i(\w^t) + \frac{B_1}{n} \sum_{\z_i\in\D} f_i^*(\pi_i) +\frac{L_f B_1 \sigma^2}{\tau_t B_2}\\\nonumber
& \quad\quad  + \left(\tau_t + 1 - \frac{B_1}{n}\right)  \bD(\pi^t, \pi)-  (\tau_t + 1) \E\left[ \bD(\pi^{t+1}, \pi)\mid \F_t\right] - \frac{\tau_t}{2} \E\left[ \bD(\pi^t,\pi^{t+1})\mid \F_t\right].
\end{align}
Note that $\left(1 - \frac{B_1}{n}\right)(-\sum_{\z_i\in\D}\pi_i^t g_i(\w^t) + \sum_{\z_i\in\D} f_i^*(\pi_i^t)) + \frac{B_1}{n}\left(-\sum_{\z_i\in\D} g_i(\w^t) \pi_i +  \sum_{\z_i\in\D} f_i^*(\pi_i)\right) - h(\w^t,\pi) = \left(1 - \frac{B_1}{n}\right)(h(\w^t,\pi^t) - h(\w^t,\pi))$. Then, subtract $h(\w^t,\pi)$ from both sides of \eqref{eq:intermediate} and use the tower property of conditional expectation. 
\begin{align*}
& \E\left[h(\w^t,\pi^{t+1}) - h(\w^t,\pi_i)\right] \\
& \leq \left(1 - \frac{B_1}{n}\right)\E\left[h(\w^{t-1},\pi^t) - h(\w^{t-1},\pi)\right] + \underbrace{\E[\sum_{\z_i\in\D} (\pi_i-\pi_i^t)(g_i(\w^t) - g_i(\w^{t-1}))]}_{\heartsuit}\\
& \quad\quad  + \left(\tau_t + 1 - \frac{B_1}{n}\right) \E[\bD(\pi^t, \pi)]-  (\tau_t + 1)  \E\left[\bD(\pi^{t+1}, \pi)\right]  - \frac{\tau_t}{2}  \E\left[\bD(\pi^t,\pi^{t+1}) \right] + \frac{L_f B_1 \sigma^2}{\tau_t B_2}.
\end{align*}
Note that $\w^t = \arg\min_{\x\in\X} \left\{\frac{1}{B_1}\sum_{i\in \B_1^{t-1}} \pi_i^{t-1} \nabla g_i(\w^{t-1};\B_{i,2}^{t-1})\cdot \x + R(\w) + \frac{\eta_{t-1}}{2}\Norm{\w-\w^{t-1}}_2^2\right\}$. The optimality condition implies that
\begin{align*}
\inner{\frac{1}{B_1}\sum_{i\in \B_1^{t-1}} \pi_i^{t-1} \nabla g_i(\w^{t-1};\B_{i,2}^{t-1}) +  \nabla R(\w^t) + \eta_{t-1}(\w^t - \w^{t-1})}{\w^t - \w}\leq 0,\quad\forall \w\in\Omega.
\end{align*}
Plug in $\w = \w^{t-1}$ and re-arrange the terms.
\begin{align*}
\eta_{t-1} \Norm{\w^t - \w^{t-1}}_2^2 & \leq \inner{\frac{1}{B_1}\sum_{i\in \B_1^{t-1}} \pi_i^{t-1} \nabla g_i(\w^{t-1};\B_{i,2}^{t-1}) +  \nabla R(\w^t)}{\w^{t-1} - \w^t}\\
& \leq \Norm{\frac{1}{B_1}\sum_{i\in \B_1^{t-1}} \pi_i^{t-1} \nabla g_i(\w^{t-1};\B_{i,2}^{t-1}) +  \nabla R(\w^t)}_2\Norm{\w^{t-1} - \w^t}_2.
\end{align*}
Then, we have
\begin{align*}
\E[\Norm{\w^t - \w^{t-1}}_2\mid \F_{t-1}] & \leq \frac{1}{\eta_{t-1} }\E\left[\Norm{\frac{1}{B_1}\sum_{i\in \B_1^{t-1}} \pi_i^{t-1} \nabla g_i(\w^{t-1};\B_{i,2}^{t-1}) +  \nabla R(\w^t)}_2 \mid \F_{t-1}\right]\\
& \leq \frac{1}{\eta_{t-1} }(C_f \sqrt{C_g^2+\zeta^2} + C_R)
\end{align*}
Then, we can bound the $\heartsuit$ term as
\begin{align*}
\heartsuit & \leq \sum_{\z_i\in\D}\E[ (|\pi_i|+|\pi_i^t|)|g_i(\w^t) - g_i(\w^{t-1})|] \leq 2 n C_f C_g \E[\Norm{\w^t - \w^{t-1}}_2] \\
& \leq \frac{2n C_f C_g (C_f \sqrt{C_g^2+\zeta^2} + C_R)}{\eta_{t-1}}.
\end{align*}
Define $C_1\coloneqq 2C_f C_g (C_f \sqrt{C_g^2+\zeta^2} + C_R)$, $\z^t \coloneqq (\w^t,\pi^{t+1})$, $\z \coloneqq (\w,\pi)$, and $Q_1(\z^t,\z) \coloneqq \L(\w^t, \pi) - \L(\w^t, \pi^{t+1}) =  \frac{1}{n}\sum_{\z_i\in\D} (\pi_i  g_i(\w^t) -  f_i^*(\pi_i)) - \frac{1}{n}\sum_{\z_i\in\D} (\pi_i^{t+1} g_i(\w^t)- f_i^*(\pi_i^{t+1})) =  \frac{1}{n}(h(\w^t,\pi^{t+1})-h(\w^t,\pi))$.
\begin{align*}
& \E[Q_1(\z^t,\z)] \leq \left(1 - \frac{B_1}{n}\right)\E[Q_1(\z^{t-1},\z)] +  \frac{C_1}{\eta_{t-1}} + \frac{L_f B_1 \sigma^2}{\tau_t n B_2}\\
& \quad\quad  + \left(\tau_t + 1 - \frac{B_1}{n}\right) \frac{1}{n}\E[\bD(\pi^t, \pi)]-  (\tau_t + 1)  \frac{1}{n}\E\left[\bD(\pi^{t+1}, \pi)\right]  - \frac{\tau_t}{2}  \frac{1}{n}\E\left[\bD(\pi^t,\pi^{t+1}) \right].
\end{align*}
\end{proof}

\begin{lemma}\label{lem:primal_contraction}
 Under Assumptions~\ref{asm:var},~\ref{asm:bounded_dom},~\ref{asm:R}, it is satisfied that 
 \begin{align*}
\E[Q_0(\z^t,\z^*)] &\leq \frac{\eta_t}{2}\E[\Norm{\w^t-\w^*}_2^2] - \frac{\eta_t + \mu}{2}\E[\Norm{\w^{t+1}-\w^*}_2^2] - \left(\frac{\eta_t}{8} - \frac{L_R}{2}\right)\E[\Norm{\w^t-\w^{t+1}}_2^2]\\
& \quad\quad + \frac{2 C_f^2 (C_g^2 + \zeta^2)}{\eta_t \min\{B_1,B_2\}}  + \frac{L_f C_g^2 \rho_t}{n}\E[\bD(\pi^t,\pi^{t+1})] + \frac{C_\Omega^2}{2\rho_t} + \frac{2(C_f^2 C_g^2+C_R^2)}{\eta_t},
 \end{align*}
 where $Q_0(\z^t,\z^*)\coloneqq \L(\w^t,\pi^{t+1}) - \L(\w^*,\pi^{t+1})$, $\rho_t > 0$.
\end{lemma} 
\begin{proof}
Applying Lemma 3.8 in \citet{lan2020first} leads to 
\begin{align*}
& \inner{\frac{1}{B_1}\sum_{i\in\B_1^t}\pi_i^t \nabla g_i(\w^t;\B_{i,2}^t)}{\w^{t+1} - \w^*} + R(\w^{t+1}) - R(\w^*) \\
& \leq \frac{\eta_t}{2}\Norm{\w^t-\w^*}_2^2 - \frac{\eta_t + \mu}{2}\Norm{\w^{t+1} - \w^*}_2^2 - \frac{\eta_t}{2}\Norm{\w^t - \w^{t+1}}_2^2.
\end{align*}
Define $Q_0(\z^t,\z^*)\coloneqq \L(\w^t,\pi^{t+1}) - \L(\w^*,\pi^{t+1}) = \frac{1}{n}\sum_{i=1}^n \pi_i^{t+1} (g_i(\w^t) -g_i(\w^*))+ R(\w^t) - R(\w^*)$. By the convexity $g_i(\w^*) - g(\w^t) \geq \inner{\nabla g_i(\w^t)}{\w^* -\w^t}$ and the monotonicity of $f_i$ (i.e. $\pi_i^t\geq 0$) (or otherwise the linearity of $g_i$), we have \begin{align*}
\frac{1}{n}\sum_{i=1}^n \pi_i^{t+1} (g_i(\w^t) -g_i(\w^*)) \leq  \inner{\frac{1}{n}\sum_{i=1}^n \pi_i^{t+1} \nabla g_i(\w^t)}{\w^t-\w^*}.
\end{align*}
Then, $Q_0(\z^t,\z^*)$ can be bounded as
\begin{align}\nonumber
& Q_0(\z^t,\z^*) \leq \inner{\frac{1}{n}\sum_{i=1}^n \pi_i^{t+1} \nabla g_i(\w^t)}{\w^t-\w^*} + R(\w^t) - R(\w^*)\\\nonumber
& = \inner{\frac{1}{n}\sum_{i=1}^n \pi_i^{t+1} \nabla g_i(\w^t)}{\w^{t+1}-\w^*} +  \inner{\frac{1}{n}\sum_{i=1}^n \pi_i^{t+1} \nabla g_i(\w^t)}{\w^t-\w^{t+1}} + R(\w^t) - R(\w^*)\\\nonumber
& = \underbrace{\inner{\frac{1}{n}\sum_{i=1}^n \pi_i^t \nabla g_i(\w^t)}{\w^{t+1}-\w^*}  + R(\w^{t+1}) - R(\w^*)}_{\text{\FiveStarOpen}}  + \underbrace{\inner{\frac{1}{n}\sum_{i=1}^n (\pi_i^{t+1} - \pi_i^t) \nabla g_i(\w^t)}{\w^{t+1}-\w^*}}_{\text{\Asterisk}} \\\label{eq:Q0}
& \quad\quad +  \underbrace{\inner{\frac{1}{n}\sum_{i=1}^n \pi_i^{t+1} \nabla g_i(\w^t)}{\w^t-\w^{t+1}}}_{\clubsuit} + \underbrace{R(\w^t) - R(\w^{t+1})}_{\spadesuit}.
\end{align}
The $\text{\FiveStarOpen}$ term can be handled as
\begin{align*}
\text{\FiveStarOpen} & = \inner{\frac{1}{B_1}\sum_{i\in\B_1^t} \pi_i^t \nabla g_i(\w^t;\B_{i,2}^t)}{\w^{t+1}-\w^*}  + R(\w^{t+1}) - R(\w^*)\\
& \quad\quad + \inner{\frac{1}{n}\sum_{i=1}^n \pi_i^t \nabla g_i(\w^t) - \frac{1}{B_1}\sum_{i\in\B_1^t} \pi_i^t \nabla g_i(\w^t;\B_{i,2}^t)}{\w^{t+1} -\w^t} \\
& \quad\quad + \inner{\frac{1}{n}\sum_{i=1}^n \pi_i^t \nabla g_i(\w^t) - \frac{1}{B_1}\sum_{i\in\B_1^t} \pi_i^t \nabla g_i(\w^t;\B_{i,2}^t)}{\w^t - \w^*}.
\end{align*}
Note that $\E[\inner{\frac{1}{n}\sum_{i=1}^n \pi_i^t \nabla g_i(\w^t) - \frac{1}{B_1}\sum_{i\in\B_1^t} \pi_i^t \nabla g_i(\w^t;\B_{i,2}^t)}{\w^t - \w^*}\mid \F_t]=0$ and
\begin{align*}
& \E\left[\inner{\frac{1}{n}\sum_{i=1}^n \pi_i^t \nabla g_i(\w^t) - \frac{1}{B_1}\sum_{i\in\B_1^t} \pi_i^t \nabla g_i(\w^t;\B_{i,2}^t)}{\w^{t+1} -\w^t}\mid \F_t\right]\\
& \leq \frac{1}{\eta_t}\E\left[\Norm{\frac{1}{n}\sum_{i=1}^n \pi_i^t \nabla g_i(\w^t) - \frac{1}{B_1}\sum_{i\in\B_1^t} \pi_i^t \nabla g_i(\w^t;\B_{i,2}^t)}_2^2\mid \F_t\right] + \frac{\eta_t}{4}\E\left[\Norm{\w^t-\w^{t+1}}_2^2\mid \F_t\right]\\
& \leq \frac{2 C_f^2 (C_g^2 + \zeta^2)}{\eta_t \min\{B_1,B_2\}} + \frac{\eta_t}{8}\E\left[\Norm{\w^t-\w^{t+1}}_2^2\mid \F_t\right].
\end{align*}
Based on the Young's inequality with a constant $\rho_t>0$, the $\text{\Asterisk}$ term can be upper bounded as
\begin{align*}
& \inner{\frac{1}{n}\sum_{i=1}^n (\pi_i^{t+1} - \pi_i^t) \nabla g_i(\w^t)}{\w^{t+1}-\w^*}  \leq \frac{1}{n}\sum_{i=1}^n C_g\Norm{\pi_i^{t+1} - \pi_i^t}_2\Norm{\w^{t+1}-\w^*}_2 \\
& \leq \frac{C_g^2 \rho_t}{2n} \sum_{i=1}^n \Norm{\pi_i^{t+1} - \pi_i^t}_2^2 + \frac{\Norm{\w^{t+1} -\w^*}_2^2}{2\rho_t} \leq \frac{L_f C_g^2 \rho_t}{n}\bD(\y^t,\pi^{t+1}) + \frac{C_\Omega^2}{2\rho_t}.
\end{align*}
We bound the $\clubsuit$ and $\spadesuit$ terms by
\begin{align*}
& \spadesuit \leq \frac{2C_f^2 C_g^2 }{\eta_t}+ \frac{\eta_t}{8}\Norm{\w^t - \w^{t+1}}_2^2,\\
& \spadesuit\leq \inner{\nabla R(\w^{t+1})}{\w^t - \w^{t+1}} + \frac{L_R}{2}\Norm{\w^t - \w^{t+1}}_2^2 \leq \frac{2 C_R^2}{\eta_t} + \left(\frac{\eta_t}{8} + \frac{L_R}{2}\right)\Norm{\w^t - \w^{t+1}}_2^2.
\end{align*}
Plug the upper bounds of $\text{\FiveStarOpen}$, \text{\FiveStarOpen}, $\clubsuit$, $\spadesuit$ into \eqref{eq:Q0} and use the tower property of conditional expectation.
\begin{align*}
\E[Q_0(\z^t,\z^*)] &\leq \frac{\eta_t}{2}\E[\Norm{\w^t-\w^*}_2^2] - \frac{\eta_t + \mu}{2}\E[\Norm{\w^{t+1}-\w^*}_2^2] - \left(\frac{\eta_t}{8} - \frac{L_R}{2}\right)\E[\Norm{\w^t-\w^{t+1}}_2^2]\\
& \quad\quad + \frac{2 C_f^2 (C_g^2 + \zeta^2)}{\eta_t \min\{B_1,B_2\}}  + \frac{L_f C_g^2 \rho_t}{n}\E[\bD(\y^t,\pi^{t+1})] + \frac{C_\Omega^2}{2\rho_t} + \frac{2(C_f^2 C_g^2+C_R^2)}{\eta_t}.
\end{align*}
\end{proof}

Next, we are ready to present the proof of Theorem~\ref{thm:scvx_sox}.
\begin{proof}
We define $C_2\coloneqq \frac{2 C_f^2 (C_g^2 + \zeta^2)}{\min\{B_1,B_2\}} + 2(C_f^2 C_g^2+C_R^2)$ and set $\eta_t \geq 4L_R$. By Lemma~\ref{lem:primal_contraction} and Lemma~\ref{lem:dual_contraction}, we have
\begin{align*}
\E[Q_0(\z^t,\z^*)] &\leq \frac{\eta_t}{2}\E[\Norm{\w^t-\w^*}_2^2] - \frac{\eta_t + \mu}{2}\E[\Norm{\w^{t+1}-\w^*}_2^2]+ \frac{C_2}{\eta_t}  + \frac{L_f C_g^2 \rho_t}{n}\E[\bD(\pi^t,\pi^{t+1})] + \frac{C_\Omega^2}{2\rho_t},\\
\E[Q_1(\z^{t-1},\z^*)] & \leq \frac{n\E[Q_1(\z^{t-1},\z^*) - Q_1(\z^t,\z^*)]}{B_1} + \frac{n C_1}{\eta_{t-1} B_1} + \frac{L_f  \sigma^2}{\tau_t B_2} - \frac{\tau_t}{2 B_1} \E\left[\bD(\pi^t,\pi^{t+1}) \right]\\
& \quad\quad + \left(\tau_t + 1 - \frac{B_1}{n}\right) \frac{1}{B_1}\E[\bD(\pi^t, \hat{\pi}^*)]-  (\tau_t + 1)  \frac{1}{B_1}\E\left[\bD(\pi^{t+1}, \hat{\pi}^*)\right]. 
\end{align*}
Sum the first equation from $0$ to $T-1$ and the second equation from $1$ to $T$. Consider that $Q(\z^t,\z^*) = Q_1(\z^t,\z^*) + Q_0(\z^t,\z^*) = \L(\w^t,\hat{\pi}^*)- \L(\w^*,\pi^{t+1}) \geq 0$ and $\E[\bD(\pi,\pi')]\geq 0$.
\begin{align*}
& \sum_{t=0}^{T-1} \E[Q(\z^t,\z^*)] + \frac{\eta_{T-1}+\mu}{2}\E[\Norm{\w^T - \w^*}_2^2] + (\tau_T + 1)\frac{1}{B_1}\E[\bD(\pi^{T+1},\hat{\pi}^*)]\\
& \leq \frac{\eta_0}{2}\E[\Norm{\w^0-\w^*}_2^2] + \sum_{t=1}^{T-1}\frac{(\eta_t - (\eta_{t-1} + \mu))}{2} \E[\Norm{\w^t-\w^*}_2^2] + C_2\sum_{t=0}^{T-1}\frac{1}{\eta_t} + \frac{L_f C_g^2\rho_0 }{n}\E[\bD(\pi^0,\pi^1)]\\
& \quad\quad + \frac{C_\Omega^2}{2}\sum_{t=0}^{T-1}\frac{1}{\rho_t} + \frac{n\E[Q_1(\z^0,\z^*) - Q_1(\z^T,\z^*)]}{B_1} + \frac{n C_1}{B_1}\sum_{t=1}^T \frac{1}{\eta_{t-1}} + \frac{L_f\sigma^2}{B_2}\sum_{t=1}^T \frac{1}{\tau_t} \\
& \quad\quad +\sum_{t=1}^{T-1} \left(\frac{L_f C_g^2 \rho_t}{n} - \frac{\tau_t}{2B_1}\right) \E[\bD(\pi^t,\pi^{t+1})] + \left(\tau_1 + 1 - \frac{B_1}{n}\right) \frac{1}{B_1}\E[\bD(\pi^1,\hat{\pi}^*)] \\
& \quad\quad + \frac{1}{B_1}\sum_{t=2}^T \left(\tau_t-\frac{B_1}{n} - \tau_{t-1} \right) \E[\bD(\pi^t,\hat{\pi}^*)].
\end{align*}
We choose $\rho_t = \frac{n}{2B_1 L_f C_g^2} \tau_t$, $\eta_t = (t+1)\mu$, $\tau_t = \frac{B_1}{n}(t+1)$. 
\begin{align*}
& \sum_{t=0}^{T-1} \E[Q(\z^t,\z^*)] + \frac{(T+1)\mu}{2}\E[\Norm{\w^T - \w^*}_2^2] + (T+1) \frac{1}{n}\E[\bD(\pi^{T+1},\hat{\pi}^*)]\\
& \leq \frac{\mu}{2}\E[\Norm{\w^0-\w^*}_2^2] + \frac{L_f C_g^2 }{2n L_f C_g^2}\E[\bD(\pi^0,\pi^1)] + \left(1 + \frac{B_1}{n}\right) \frac{1}{B_1}\E[\bD(\pi^1,\hat{\pi}^*)] \\
& \quad\quad  + \left(C_2 + \frac{n C_1}{B_1}\right)\sum_{t=0}^{T-1}\frac{1}{\mu(t+1)} + C_\Omega^2  L_f C_g^2\sum_{t=0}^{T-1}\frac{1}{t+1} + \frac{n L_f\sigma^2}{B_1 B_2} \sum_{t=1}^T \frac{1}{t+1} \\
& \quad\quad + \frac{n\E[Q_1(\z^0,\z^*) - Q_1(\z^T,\z^*)]}{B_1}.
\end{align*}
If we set $u_i^0 = g_i(\w^t;\B_{i,2}^0)$ and $\pi_i^0 = \nabla f_i(u_i^0)$, we have $\pi^1 = \pi^0$. Define $G_0 \coloneqq \frac{1}{n}\sum_{i=1}^n g_i(\w^0)$ and $C_3\coloneqq 2(C_f C_g C_\Omega + C_f G_0 + C_{f^*}) $. Note that
\begin{align*}
& \frac{n (Q_1(\z^0,\z^*) - Q_1(\z^T,\z^*))}{B_1} \\
& = - \frac{1}{B_1}\sum_{i=1}^n\left((\pi_i^1-\hat{\pi}_i^*)g_i(\w^0) + (\hat{\pi}_i^* - \pi_i^{T+1})g_i(\w^T) - f_i^*(\pi_i^1)  + f_i^*(\pi_i^{T+1})\right) \leq \frac{n C_3}{B_1},\\
& D_{f_i^*}(\pi_i^0,\hat{\pi}_i^*) = D_{f_i^*}(\nabla f_i(u_i^0), \nabla f_i(g_i(\w^*))) = D_{f_i}(u_i^0, g_i(\w^*)) \leq \frac{L_f}{2}\Norm{u_i^0 - g_i(\w^*)}^2.
\end{align*}
Thus, $\E[\bD(\pi^0,\pi^*)]= \frac{L_f}{2}\sum_{i=1}^n\E\left[\Norm{g_i(\w^T;\B_{i,2}^0) - g_i(\w^*)}^2\right] \leq \frac{n C_4}{2}$, where we define $C_4\coloneqq L_f(\sigma^2/B_2 + C_g^2 C_\Omega^2)$. We have 
\begin{align*}
\frac{1}{T}\sum_{t=0}^{T-1} \E[Q(\z^t,\z^*)] & = \frac{\mu C_\Omega^2}{2 T} + \frac{n C_4}{B_1 T} +  \left(C_2 + \frac{n C_1}{B_1}\right)\frac{\log T}{\mu T} +  C_\Omega^2 L_f C_g^2 \frac{\log T}{T}\\
& \quad\quad + \frac{n L_f\sigma^2}{B_1 B_2}\frac{\log T}{T} + \frac{n C_3}{B_1 T} = \O\left(\frac{n\log T}{B_1\mu T}\right),\\
\E[\Norm{\w^T - \w^*}_2^2]  & \leq \frac{C_\Omega^2}{T+1} + \frac{2 n C_4}{\mu B_1 (T+1)} + 2\left(C_2 + \frac{n C_1}{B_1}\right)\frac{\log T}{\mu^2(T+1)} + 2 C_\Omega^2 L_f C_g^2 \frac{\log T}{\mu (T+1)}\\
& \quad\quad + \frac{2n L_f\sigma^2}{B_1 B_2}\frac{\log T}{\mu (T+1)} + \frac{2 n C_3}{\mu B_1 (T+1)} = \O\left(\frac{n\log T}{B_1\mu^2(T+1)}\right).
\end{align*}
Thus, for $\bar{\w}^T = \frac{1}{T}\sum_{t=0}^{T-1} \w^T$, $\bar{\pi}^T = \frac{1}{T}\sum_{t=0}^{T-1} \pi^{t+1}$ we have
\begin{align*}
F(\bar{\w}^T) - F(\w^*) \leq Q(\bar{\z}^T,\z^*) +  \frac{L_f C_g^2}{2}\Norm{\bar{\w}^T - \w^*}_2^2
\end{align*}
Note that $Q(\bar{\z}^T,\z^*) = \L(\bar{\w}^T,\hat{\pi}^*) - \L(\w^*,\bar{\pi}^T) \leq \frac{1}{T}\sum_{t=0}^{T-1}\left(\L(\w^T,\hat{\pi}^*) - \L(\w^*,\pi^{t+1})\right) = \frac{1}{T}\sum_{t=0}^{T-1} Q(\z^t,\z^*)$ since $\L(\x,\y)$ is convex-concave. Besides, Jensen's inequality implies that 
\begin{align*}
\E[\Norm{\bar{\w}^T - \w^*}_2^2] \leq \frac{1}{T}\sum_{t=0}^{T-1}\E[\Norm{\w^T - \w^*}_2^2] = \frac{1}{T}\sum_{t=0}^{T-1}\O\left(\frac{n \log t}{B_1\mu^2(t+1)}\right) \leq \O\left(\frac{n (\log T)^2}{B_1  \mu^2 T}\right).
\end{align*}
Then, we have that $\E[F(\bar{\w}^T) - F(\w^*) ]\leq \O\left(\frac{n (\log T)^2}{B_1  \mu^2 T}\right)$.
\end{proof}

\section{Extensions for a More general Class of Problems}\label{sec:ext}
In this section, we briefly discuss the extension when $f_i$ is also a stochastic function such that we can only get an unbiased estimate of its gradient, which has an application in MAML. To this end, we assume  a stochastic oracle of $f_i$ that given any $g(\cdot)$  returns $\nabla f_i(g(\cdot); \iota)$ such that $\E[\nabla f_i(g(\cdot); \iota)]= \nabla f_i(g)$, $\E[\|\nabla f_i(g(\cdot); \iota) - \nabla f_i(g(\cdot))\|^2]\leq \chi^2$. We can extend our results for the smooth nonconvex problems by the modifications as follows: First, we need to assume that $\nabla f_i(\cdot; \iota)$ is Lipschitz-continuous; Second, the $\frac{2\beta^2C_f^2(\zeta^2+C_g^2)}{\min\{B_1,B_2\}}$ term in Lemma~\ref{lem:grad_recursion} should be replaced by $\frac{2\beta^2(\chi^2C_g^2+(C_f^2+\chi^2/B_3))(\zeta^2 +  C_g^2)}{\min\{B_1,B_2,B_3\}}$, where $B_3$ is the batch size for sampling $\iota$. Note that Lemma~\ref{lem:fval_recursion} remains the same and Theorem~\ref{thm:sox} does not change (up to a constant factor). 

\newpage
\bibliography{all,draft}
\end{document}